\documentclass[11pt]{amsart}
\usepackage{graphicx, color}
\usepackage{amssymb,amsfonts,amsmath}
\usepackage{tikz}
\usetikzlibrary{arrows.meta}
\usepackage{subcaption}
\emergencystretch=\maxdimen
\newcommand{\ds}{\displaystyle}

\newtheorem{prop}{Proposition}[section]

\newtheorem{lemma}{Lemma}[section]
\newtheorem{remark}{Remark}[section]
\newtheorem{notations}{Notations}[section]

\newcommand{\R}{\mathbb R}
\newcommand{\N}{\mathbb N}
\newcommand{\Dx}{\Delta x}
\newcommand{\Dt}{\Delta t}
\def\maxd{\mathop{{\max}}\limits}
\def\sumd{\mathop{{\sum}}\limits}

\newcommand{\KK}{\eqref{jc}}
\newcommand{\Jmr}{\eqref{eq:junc1}}
\newcommand{\Jrs}{\eqref{eq:1in1_Rpb}}

\begin{document}


\title[A relaxation scheme for isentropic gas dynamics on networks]{A relaxation scheme for the equations of isentropic gas dynamics on a network with jump transmission conditions}


\author{Maya Briani}
\address{Istituto per le Applicazioni del Calcolo, Consiglio Nazionale delle Ricerche}
\email{maya.briani@cnr.it}

\author{Roberto Natalini}
\address{Istituto per le Applicazioni del Calcolo, Consiglio Nazionale delle Ricerche}
\email{roberto.natalini@cnr.it}

\author{Magali Ribot}
\address{Institut Denis Poisson, 
Universit\'e d'Orl\'eans, CNRS, Universit\'e de Tours \& Université Paris-Saclay, MaIAGE, INRAE Jouy-en-Josas }
\email{magali.ribot@univ-orleans.fr}

\begin{abstract}
In this paper we propose a new numerical scheme of relaxation type to approximate the Euler equations of isentropic gas dynamics on the arcs of a network. At the junction mass conservation and a jump transmission condition on the density are given, and a new solver is introduced to deal with both subsonic and supersonic cases. Consistency properties of the solver are proven and numerical tests are displayed to show its good performance also with respect to other possible solvers. 
\end{abstract}

\maketitle

\section{Introduction}

In this article, we consider the Euler equations of isentropic gas dynamics set on the one dimensional arcs of a network. On these arcs the Euler equations are, as usual, given by:
\begin{equation}\label{GasDynamic}
\left\{\begin{array}{l}
\ds{\partial_{t}\rho+ \partial_{x} (\rho u)=0,}\\
\ds{\partial_{t}(\rho u)+\partial_{x} (\rho u^2+ p(\rho))=0,}
\end{array}\right.
\end{equation}
where $x$ is a spatial coordinate, $t>0$ is the time,  $\rho$ is the density of a fluid, so we will assume all the time $\rho>0$, and  $q=\rho u$ its momentum.  Also,  we shall assume that  the pressure $p$ is given by  the pressure law for isentropic gases, that is to say 
\begin{equation}\label{eq:pfunct}
p(\rho) = p_0\ \rho^\gamma \text{ with } \gamma>1.
\end{equation}
Usually, for isentropic fluids, the value of $\gamma$ is taken such that $1<\gamma<5/3$. However, here, we have in mind various applications, and in particular to biology, see for instance \cite{preziosi1, preziosi2, NRT}, where the value of $\gamma$ is not clearly defined, so in the following we consider all the values  $\gamma>1$.

While on the external boundary of the network we can assign, for instance, standard no-flux conditions to the gas, at the junction we have to impose some physically motivated transmission conditions. In this paper we  address a special class of transmission conditions that, on a 2-arc network -- i.e.: a network composed by just two arcs which meet at just one junction point, see Figure\ref{fig:network_2arcs} -- at the junction point \(x\) reads as: 
\begin{equation}\label{KK1}
(\rho u)_{\ell} (t, x_{\ell})=(\rho u)_{r}(t, x_r)= \kappa (\rho_{\ell}(t, x_{\ell})-\rho_{r}(t, x_r)), \text{ for all time } t>0,
\end{equation}
where the subscripts \({\ell}\) and \(r\) stand respectively for the left and right side of the junction point and \(\kappa\) is a given positive constant. In the following we shall refer to these conditions as the Jump Transmission Condition (JTC). 

 In fact, (JTC) implies that the density flux incoming at the junction is equal to the one outgoing and so the total mass is conserved at the junction. This is similar to the Kirchoff’s law for an electrical circuit. This feature is common to most of the models of flows on networks. However almost all of these models impose   continuity condition for the density at the junction, as assumed for instance in \cite{DM, CG, DZ, VZ}, both for parabolic and hyperbolic models, see also \cite{garavello_review, BCGHP14} for some informed reviews in the hyperbolic case. This is the point where 
 our condition (JTC), given by \eqref{KK1}, differs, since it allows jump discontinuities at the junction.

Transmission conditions as those appearing in \eqref{KK1} were introduced, only for diffusive problems, in the simple case of two regions by Kedem and Katchalsky in \cite{KK}, as permeability conditions in the description of passive transport through a biological membrane, see also \cite{QVZ, CZ, CN} for various mathematical and related formulations. More recently they were generalized to hyperbolic problems and proposed as transmission conditions for hyperbolic-parabolic and hyperbolic-elliptic 
systems describing biological situations such as the movement of microorganisms on networks driven by chemotaxis \cite{GN15, GN21, G1,G2,GPS}, see also \cite{BNR14,BN18,BBN21} for other models and the numerical treatment of these conditions in related situations.

In general, these density jump conditions like (JTC) seem to be more appropriate than continuity conditions when dealing with mathematical models for movements of individuals, as in traffic flows \cite{GP, BNP06, GNPT07}  or in  biological phenomena involving bacteria, eukariotic cells, or other microorganism movements \cite{BNR14,GN15}, where, even if the flux is continuous and the mass is conserved, discontinuities  at the inner nodes  for the density functions are expected, as a sort of a mathematical counterpart of bottleneck phenomena.

The purpose of this article is to propose a numerical scheme for this system, based on the vector BGK approach \cite{AN2000, Bou04}, which guarantees that condition \eqref{KK1} is fulfilled at each junction of the network without solving the Riemann problem. Actually, we propose a way to solve the (JTC) relation with a  scheme  which also satisfies a momentum equality junction condition, namely  the fact that at the discrete level the equality of the momenta  $q=\rho u$ will be enforced both as the unknown of the second equation of system \eqref{GasDynamic} and as the flux of the first equation of system \eqref{GasDynamic}. 

This approach, as long as it is possible, will be compared with a more classical solver based on the Riemann invariants at the junction, which in the following will be referred as "Riemann invariant junction condition". Unfortunately, as we will discuss in this paper, it is not clear how to deal with this approach in the general "non-subsonic" case, so the comparison will be restricted to the situations where both approaches work. 

The article is organized as follows~: in Section~\ref{2-arcs_case} we study the mathematical formulation of equations \eqref{GasDynamic} on a 2-arc network, with the conditions \eqref{KK1} at the junction.  In Section~\ref{RIJunctionCondition}, we study how to solve the Riemann problem for \eqref{GasDynamic} on a 2-arc network, with the conditions \eqref{KK1} at the junction using the classical approach with the Riemann invariants. We are able to find a solution in the subsonic case, and for some special initial conditions in the supersonic case. 

Then, in Section~\ref{DEJunctionCondition}, we introduce our relaxation scheme for the same problem, which can be written for all initial data. In Section~\ref{NumericalTests}, we present some numerical tests, and in particular a comparison between the relaxation scheme and the cases where we have an explicit solution for the Riemann problem at the junction. Finally we extend the previous study to the case of a general network in Section~\ref{Extensions}.

\section{Isentropic gas dynamics on a 2-arc network}\label{2-arcs_case}
Consider the  1D system \eqref{GasDynamic}
on a simple network, composed of two arcs, as displayed in Figure \ref{fig:network_2arcs}. The quantities linked to the arc on the left, called arc $\ell$ (resp. on the right, called arc $r$ ) are denoted with subscript $\ell$ (resp. $r$).
\begin{figure}[htbp!]
\begin{center}
\begin{tikzpicture}
\draw  (-3,0) -- (3,0) ;
\draw (-3,0) -- (-1.5,0) ;
\draw (0,0) -- (1.5,0) ;
\draw (0,0) node {$\bullet$} ;
\draw (-.1,0) node[below]{$L_{\ell}$} ;
\draw (.2,0) node[above]{$0$} ;
\draw (-3,0) node {$\bullet$} ;
\draw (-3,0) node[below]{$0$} ;
\draw (3,0) node {$\bullet$} ;
\draw (3,0) node[above]{$L_{r}$} ;
\draw (-1.5,0.3) node[above]{arc $\ell$} ;
\draw (1.5,0.3) node[above]{arc $r$} ;
\draw (0.,0.8) node[above]{Junction point} ;
\draw[-{Latex[length=2mm]}] (0,0.8) -- (0,0.1) ;
\end{tikzpicture}
\end{center}
\caption{The 2-arc network : the quantities linked to the arc $[0,L_{\ell}]$ on the left (resp. $[0,L_{r}]$ on the right) are denoted with subscript $\ell$ (resp. $r$)} 
\label{fig:network_2arcs}
\end{figure}
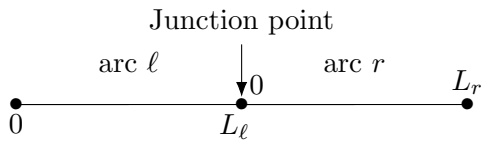

We are therefore dealing with the two following systems~:
\begin{equation}\label{GasDynamicLeft}
\left\{\begin{array}{l}
\ds{\partial_{t}\rho_{\ell}(t,x)+ \partial_{x} (\rho_{\ell} u_{\ell})(t,x)=0,}\\
\ds{\partial_{t}(\rho_{\ell} u_{\ell})(t,x)+\partial_{x} (\rho_{\ell} u_{\ell}^2+ p(\rho_{\ell}))(t,x)=0,}
\end{array}\right. \text{ for } x \in [0,L_{\ell}],
\end{equation}
and
\begin{equation}\label{GasDynamicRight}
\left\{\begin{array}{l}
\ds{\partial_{t}\rho_{r}(t,x)+ \partial_{x} (\rho_{r} u_{r})(t,x)=0,}\\
\ds{\partial_{t}(\rho_{r} u_{r})(t,x)+\partial_{x} (\rho_{r} u_{r}^2+ p(\rho_{r}))(t,x)=0,}
\end{array}\right. \text{ for } x \in [0,L_{r}],
\end{equation}
complemented with no-flux boundary conditions, at least for smooth solutions, at the external points of the two intervals, namely, always setting for the flux $q=\rho u$ in the following~:
\begin{equation}\label{bc}
\left\{\begin{array}{l}
\partial_{x} \rho_{\ell} (t, 0)=0, \quad q_{\ell} (t, 0)=0, \text{ for all } t \geq 0 \text{ on the left arc,}
\\
 \partial_{x} \rho_{r} (t, L_{r})=0 ,\quad q_{r}(t, L_{r}) =0, \text{ for all } t \geq 0 \text{ on the right arc}.
 \end{array}\right.
 \end{equation}
 Clearly the boundary conditions, as usual for hyperbolic problems, have to be intended in a proper sense, since sometimes, due to the direction of the characteristics with respect to the boundary, they need to be interpreted, see \cite{serre} and references therein for more details.
 
Finally, we define transmission conditions at the junction, which couple the unknowns of both systems. For that purpose, we consider the equality of  fluxes in order to have the mass conservation at the junction, i.e.~:
\begin{equation}\label{jcegal}
q_{\ell} (t, L_{\ell})=q_{r}(t, 0), \text{ for all } t \geq 0 .
\end{equation}
We supplement this condition with a jump transmission  condition (JTC) for the density at the junction, that is to say~:
\begin{equation}\label{jc}
q_{\ell} (t, L_{\ell})=q_{r}(t, 0)= \kappa (\rho_{\ell}(t, L_{\ell})-\rho_{r}(t, 0)), \text{ for all } t \geq 0,
\end{equation}
where $\kappa$ is a {\it permeability coefficient } that we suppose to be constant. 
These previous conditions represent some transport conditions through a jump (given by a membrane or a change of domain), and  the fluxes are linked with the difference of densities on both sides of the jump.

 This condition (JTC)  is the hyperbolic analogous of the parabolic Kedem-Katchalsky transmission condition, see for instance \cite{KK,QVZ,CZ,CN}. The connection between condition \eqref{jc} and the Kedem-Katchalsky condition can be seen in the diffusive limit of the damped version of equation \eqref{GasDynamic}. By scaling this equation as in \cite{DiFM02}, we obtain the perturbed system 
\begin{equation*}
\left\{\begin{array}{l}
\ds{\partial_{t}\rho+ \partial_{x} (\rho u)=0,}\\
\ds{\partial_{t}(\epsilon^2\rho u)+\partial_{x} (\epsilon^2\rho u^2+ p(\rho))=-\rho u,}
\end{array}\right.
\end{equation*}
where the inertial terms $\partial_{t}(\rho u)+\partial_{x} (\rho u^2)$ are supposed to be negligible. 

So, as $\epsilon$ tends to zero, the momentum $\rho u$ approaches $-\partial_x p(\rho)$ and, as shown in \cite{DiFM02},  we obtain the porous medium equation
\begin{equation*}
\partial_{t}\rho- \partial_{xx} p(\rho)=0,
\end{equation*}
see also \cite{NRT2} in the framework of asymptotic preserving numerical schemes for isentropic Euler equations coupled with chemotaxis. 

Replacing $q=\rho u$ by $-\partial_x p(\rho)$  in \eqref{jc}, we obtain the condition
\begin{equation*}
-\partial_x p(\rho_{\ell})  (t, L_{\ell})=-\partial_x p(\rho_{r})(t, 0) = \kappa (\rho_{\ell} (t, L_{\ell})-\rho_{r}(t, 0))
\end{equation*}
which is in the form of the diffusive Kedem-Katchalsky condition treated for instance in \cite{CGLP, CDP24}.

Clearly, condition  \eqref{jcegal} together with condition \eqref{bc} ensure the mass conservation of the system, that is to say, working with smooth solutions~:
\begin{equation}\label{consmass2}
 \frac{d}{dt}m (t)=0, \text{ with } m(t)= \int_{0}^{L_{\ell}} \rho_{\ell}(t, x) \, dx+\int_{0}^{L_{r}} \rho_{r}(t, x) \, dx.
  \end{equation}

\subsection{Entropy dissipation}\label{sec:entropy}
Let us  prove now that, at least for smooth solutions,  system \eqref{GasDynamicLeft}-\eqref{jc} is endowed with an entropy dissipation property at the continuous level under the condition that the density jump is small enough at the junction.

First, on an interval, there exists at least an entropy--entropy flux pair $(\eta,G)$ for system \eqref{GasDynamic},  see \cite{Dafermos} p.212, namely the functions 
\begin{equation}\label{eq:entropy_pair}
 \eta(\rho, \rho u)=\frac 1 2 \rho u^2+ \frac{p_{0}}{\gamma-1} \rho^\gamma \text{ and }G(\rho, \rho u)=\left(\frac 1 2 \rho u^2+ \frac{p_{0} \gamma}{\gamma-1} \rho^\gamma\right)u.
\end{equation}
The function $\eta$ is convex and, for smooth solutions $(\rho,q= \rho u)$ to system \eqref{GasDynamic} we have
\begin{equation*}
\partial_{t} \eta(\rho, q)+\partial_{x} G(\rho, q) = 0.
\end{equation*}
For weak solutions, the admissible (entropy) solutions have to verify (in the weak sense) the inequality
\begin{equation}\label{eq:entropy_inequ}
\partial_{t} \eta(\rho, q)+\partial_{x} G(\rho, q) \leq 0,
\end{equation}
and we can say that the admissible solutions dissipate the entropy (see again \cite{Dafermos}). In general, for one single interval $[0,L]$ with no flux conditions, the quantity 
 \[\int_{0}^{L} \eta(\rho(t,x), q (t,x))dx	\]
is decreasing in time.

Now, we want to investigate if condition (JTC) in \eqref{jc} preserves this dissipation property, at least in some range of solutions. Let us consider smooth solutions to equation  \eqref{eq:entropy_inequ} on the $2$ arcs of the previous  network and integrate it with respect to time and space~: 
\begin{equation*}
\begin{aligned}
S(t):= \int_{0}^{L_{\ell}}& \eta(\rho_{\ell}(t,x), q_{\ell} (t,x))dx+ \int_{0}^{L_{r}} \eta(\rho_{r}(t,x), q_{r} (t,x))dx  \\
&\leq  \int_{0}^{L_{\ell}} \eta(\rho_{\ell}(0,x), q_{\ell} (0,x))dx+ \int_{0}^{L_{r}} \eta(\rho_{r}(0,x), q_{r} (0,x))dx \\
 &-\int_{0}^t  \left(G (\rho_{\ell}(s,L_{\ell}), q_{\ell} (s,L_{\ell}))-G (\rho_{\ell}(s,0), q_{\ell} (s,0))\right) ds\\
 &
 -\int_{0}^t \left(G (\rho_{r}(s,L_{r}), q_{r} (s,L_{r}))-G (\rho_{r}(s,0), q_{r} (s,0))\right)ds. 
\end{aligned}
\end{equation*}

Thanks to boundary conditions on the external domain, we find that since $\gamma>1$, we have 
\[\ds G (\rho_{r}(s,L_{r}), q_{r} (s,L_{r}))=G (\rho_{\ell}(s,0), q_{\ell} (s,0))=0,\] which leads to the inequality~:
\begin{equation}\label{eq:entropy_inequ_estimate}
\begin{aligned}
 S(t)=\int_{0}^{L_{\ell}}& \eta(\rho_{\ell}(t,x), q_{\ell} (t,x))dx+ \int_{0}^{L_{r}} \eta(\rho_{r}(t,x), q_{r} (t,x))dx  \\
&\leq  \int_{0}^{L_{\ell}} \eta(\rho_{\ell}(0,x), q_{\ell} (0,x))dx+ \int_{0}^{L_{r}} \eta(\rho_{r}(0,x), q_{r} (0,x))dx \\
 &-\int_{0}^t  \left(G (\rho_{\ell}(s,L_{\ell}), q_{\ell} (s,L_{\ell}))-G (\rho_{r}(s,0), q_{r} (s,0))\right)ds. 
\end{aligned}
\end{equation}
Therefore, if the quantity $\ds \Delta G(t):= G (\rho_{\ell}(t,L_{\ell}), q_{\ell} (t,L_{\ell}))-G (\rho_{r}(t,0), q_{r} (t,0)) \geq 0$ for all $t>0$, we have entropy dissipation, that is to say  for all $t>0$ it holds
\begin{equation*}
S(t)\leq S(0)
\end{equation*}
for all $t>0$.

Now, using the conditions at the junction, we can compute $\ds  \Delta G$
 as follows. Using that $\rho>0$, we rewrite $G(\rho, q)$ as~: $$ G(\rho, q)=G(\rho, \rho u)=\frac 1 2 \rho u^3+ \frac{p_{0} \gamma}{\gamma-1} \rho^\gamma u =\frac 1 2 \frac{q^3}{\rho^2}+ \frac{p_{0} \gamma}{\gamma-1} \rho^{\gamma-1} q $$
and therefore 
\begin{equation*}
\begin{aligned}
 &\Delta G(t)= G (\rho_{\ell}(t,L_{\ell}), q_{\ell} (t,L_{\ell}))-G (\rho_{r}(t,0), q_{r} (t,0))\\
& = \frac 1 2 \left(\frac{(q_{\ell} (t,L_{\ell}))^3}{(\rho_{\ell}(t,L_{\ell}))^2}-\frac{(q_{r} (t,0))^3}{(\rho_{r}(t,0))^2}\right)+ \frac{p_{0} \gamma}{\gamma-1} \left((\rho_{\ell}(t,L_{\ell}))^{\gamma-1} q_{\ell} (t,L_{\ell})-(\rho_{r}(t,0))^{\gamma-1} q_{r} (t,0) \right)
\end{aligned}
\end{equation*}
Using the conditions \eqref{jc} at the junction, the previous expression can be simplified as~:
\begin{equation*}
\begin{aligned}
\Delta G(t)= G (\rho_{\ell}(t,L_{\ell}), q_{\ell} (t,L_{\ell}))&-G (\rho_{r}(t,0), q_{r} (t,0))\\
 = -\frac 1 2 \kappa^3 &(\rho_{\ell}(t,L_{\ell})-\rho_{r}(t,0))^4\frac{\rho_{r}(t,0)+\rho_{\ell}(t,L_{\ell})}{(\rho_{r}(t,0))^2(\rho_{\ell}(t,L_{\ell}))^2}\\
&
+ \frac{p_{0} \gamma}{\gamma-1} \kappa (\rho_{\ell}(t,L_{\ell})-\rho_{r}(t,0))\left((\rho_{\ell}(t,L_{\ell}))^{\gamma-1}-(\rho_{r}(t,0))^{\gamma-1}  \right).
\end{aligned}
\end{equation*}

The first term on the right hand-side is negative whereas the second one is positive. 
We remark that as long as the difference $\rho_{\ell}(t,L_{\ell})-\rho_{r}(t,0)$ is small, 
\begin{equation*}
\Delta G(t)\sim  \frac{p_{0} \gamma}{\gamma-1} \kappa (\rho_{\ell}(t,L_{\ell})-\rho_{r}(t,0))\left((\rho_{\ell}(t,L_{\ell}))^{\gamma-1}-(\rho_{r}(t,0))^{\gamma-1}  \right) \geq 0
\end{equation*}
and the dissipation entropy property holds, 
whereas as the difference $\rho_{\ell}(t,L_{\ell})-\rho_{r}(t,0)$ becomes large, we have
\begin{equation*}
\Delta G(t)\sim -\frac 1 2 \kappa^3 (\rho_{\ell}(t,L_{\ell})-\rho_{r}(t,0))^4\frac{\rho_{r}(t,0)+\rho_{\ell}(t,L_{\ell})}{(\rho_{r}(t,0))^2(\rho_{\ell}(t,L_{\ell}))^2} \leq 0,
\end{equation*}
and then, a priori, the dissipation property is no longer true.
In Section \ref{NumericalTests} we will test this dissipation property more closely, for various initial data, since in some cases the dissipation inside the arcs can compensate the growth of the entropy at the junction.

\section{The Riemann problem at the junction - the case of a 2-arc network } \label{RIJunctionCondition}

We want to solve the Riemann problem for \eqref{GasDynamicLeft}--\eqref{jc} at the junction using the classical approach with  Riemann invariants. 

\subsection{The Riemann problem on a single interval} 
Let us first describe Riemann invariants on an interval before generalizing it to the junction case, see for example \cite{Dafermos} for more details. 

Let consider isentropic gas dynamics given by Eq.\eqref{GasDynamic} on a single arc, under the form
 \begin{equation}\label{eq:GenSys}
\partial_{t} U+\partial_{x}F(U)=0 , 
\end{equation}
where $U \in \R^2$ is the vector of unknowns 
$\ds
U=\left( \begin{array}{c}  \rho \\ q=\rho u\end{array}\right)$,
and the function $F : \R^2 \to \R^2$ is the flux of the system defined by
$\ds F(U)=\left( \begin{array}{c}  \rho u \\  \rho u^2 +p(\rho)\end{array}\right).
$

The eigenvalues of the system are
\begin{equation*}
\mu_1(\rho,q) = \frac{q}{\rho}-\sqrt{p^\prime(\rho)}, \quad \mu_2(\rho,q) = \frac{q}{\rho}+\sqrt{p^\prime(\rho)},
\end{equation*}
where $q=\rho\,u$, and the eigenvectors are
\begin{equation*}
r_1(\rho,q) = \left(\begin{array}{c}
\rho \\ q-\rho\sqrt{p^\prime(\rho)}
\end{array}\right), \quad 
r_2(\rho,q) = \left(\begin{array}{c}
\rho \\ q+\rho\sqrt{p^\prime(\rho)}
\end{array}\right).
\end{equation*}
The system is strictly hyperbolic in $\mathcal{D}=\{(\rho,q) : \rho> 0 \mbox{ or } q=0\}$ with $\mu_1<\mu_2$. It is genuinely nonlinear in both characteristic fields, i.e. $\nabla\mu_i \cdot r_i \neq 0$, i=1,2.
The two functions $z_1 = z_1(\rho, q)$, $z_2 = z_2(\rho, q)$ 
are called the \textit{Riemann invariants} of system \eqref{GasDynamic} corresponding to $\mu_1$ and $\mu_2$ if they satisfy the equations
\begin{equation}\label{eq:gasRiemannInv}
\nabla z_i \cdot r_i = 0, \quad i=1,2.
\end{equation}
One solution of \eqref{eq:gasRiemannInv}, when $p$ follows the pressure law of isentropic gases \eqref{eq:pfunct} is
\begin{equation}\label{eq:solRinv}
z_1(\rho, q) = \frac{q}{\rho} + \alpha_\gamma \sqrt{p^\prime(\rho)} , \quad z_2(\rho, q) = \frac{q}{\rho} - \alpha_\gamma\sqrt{p^\prime(\rho)}, \quad  \text{ with }\alpha_\gamma = \frac{2}{\gamma-1}.
\end{equation}

Consider now the Riemann problem for \eqref{eq:GenSys}, that is to say system \eqref{eq:GenSys}   complemented with piecewise constant initial data such that 
\begin{align*}
U(x,0) = \left\{\begin{array}{lcl}
U^- &\mbox{if}& x< 0,
\\
U^+ &\mbox{if}& x> 0,
\end{array}\right.
\end{align*}
for two given constant left and right states $U^-=(\rho^-,q^-)$ and $U^+=(\rho^+,q^+)$.
The solution to the Riemann problem is then explicitly known and defined by an intermediate state $\hat U = (\hat\rho,\hat q)$ verifying the following conditions:
\begin{itemize}
\item[i)] The left state $U^-$ is connected to $\hat U$ by a 1-shock or a 1-rarefaction wave such that
$$z_1(\rho^-,q^-) = z_1(\hat\rho,\hat q).$$
Moreover, shock waves satisfy $\mu_1(U^-)> \mu_1(\hat U)$ (Lax admissibility condition) while rarefaction waves satisfy $\mu_1(U^-)< \mu_1(\hat U)$. We then have a 1-shock if $\hat \rho > \rho^-$ and $z_1(\rho^-,q^-) = z_1(\hat\rho,\hat q)$ or a 1-rarefaction if $\hat\rho < \rho^-$ and $z_1(\rho^-,q^-) = z_1(\hat\rho,\hat q)$.
\item[ii)] The intermediate state $\hat U$ is connected to the right state $U^+$ by a 2-shock 
if $\hat\rho > \rho^+$ and $z_2(\hat\rho,\hat q) = z_2(\rho^+,q^+)$ or by a 2-rarefaction if $\hat\rho < \rho^+$ and $z_2(\hat\rho,\hat q) = z_2(\rho^+,q^+)$.
\end{itemize}

The intermediate state $\hat U$ is then uniquely defined if, in the $(\rho,u)$-plane, the following two curves $u_1(\rho)= u^- + \alpha_\gamma \left(\sqrt{p^\prime(\rho^-)} - \sqrt{p^\prime(\rho)}\right)$ and $u_2(\rho)= u^+ - \alpha_\gamma \left(\sqrt{p^\prime(\rho^+)}-\sqrt{p^\prime(\rho)}\right)$ intersect. Since $u_1(\rho)$ is strictly decreasing, unbounded, and starts at $u^- + \alpha_\gamma \sqrt{p^\prime(\rho^-)}$ and $u_2(\rho)$ is strictly increasing, unbounded, with minimum $u^+ - \alpha_\gamma \sqrt{p^\prime(\rho^+)}$, the Riemann problem has a unique solution in the region where
\begin{equation*}
u^- + \alpha_\gamma \sqrt{p^\prime(\rho^-)} \geq u^+ - \alpha_\gamma \sqrt{p^\prime(\rho^+)}
\end{equation*}
or equivalently
\begin{equation}\label{eq:condRP}
z_1(\rho^-,q^-) \geq z_2(\rho^+,q^+).
\end{equation}

Furthermore, we say that a state $(\rho,q)\in\mathcal{D}\backslash\{(0,0)\}$ is \emph{subsonic} if $\mu_1(\rho,q) < 0 <\mu_2(\rho,q)$, \emph{sonic} if $\mu_1(\rho,q) =0$ or $ \mu_2(\rho,q) = 0$ and \emph{supersonic} if $\mu_1(\rho,q) < \mu_2(\rho,q)<0 \mbox{ or } 0 <\mu_1(\rho,q) < \mu_2(\rho,q)$.


\subsection{The Riemann problem at the junction: the subsonic case}\label{sec:RSsubsonic}

Now, we shall consider a simple junction connecting two arcs as sketched in Figure \ref{fig:network_2arcs}. We use the same notation as in Section \ref{2-arcs_case} and in particular, we shall use the subscript $\ell$ for the left 
 arc and $r$ for the right 
 one. 
 
 \begin{notations} From now on, we shall use the following notations~:
\begin{itemize}
\item $U^-_{\ell}=(\rho^-_{\ell},q^-_{\ell})$ and $U^+_{r}=(\rho^+_{r},q^+_{r})$ denote the constant states on the left arc and on the right arc;
\item $U^*_{\ell}=(\rho^*_{\ell},q^*_{\ell})$ and $U^*_{r}=(\rho^*_{r},q^*_{r})$ represent the unknown values of the left and right states at the junction.
\end{itemize}
\end{notations}

Following what has been done previously, we define the two sonic curves that determine the boundary between subsonic  and supersonic regions, see the black curves on Figure \ref{fig:junction_Rpb_curves},  thanks to the following functions
\begin{equation}
\left\{\begin{array}{c}
q_s^+(\rho) =  \rho\sqrt{p^\prime(\rho)},
\smallskip\\
q_s^-(\rho) = - \rho\sqrt{p^\prime(\rho)}.
\end{array}\right.
\label{subsonic_curves}
\end{equation}
We define the constant data $U^-_{\ell}=(\rho^-_{\ell},q^-_{\ell})$ and $U^+_{r}=(\rho^+_{r},q^+_{r})$ given on the left and right
arc respectively.
We will assume that we are in the subsonic case at the junction, that is to say that $q_s^-(\rho^-_{\ell}) < q^-_{\ell}< q_s^+(\rho^-_{\ell})$ and  $q_s^-(\rho^+_{r}) < q^+_{r}<q_s^+(\rho^+_{r})$.

Now, in order to solve the Riemann problem, the unknown states $U^*_\ell$ and $U^*_r$ at the junction have to lie on the Riemann invariant curves defined from $U^-_{\ell}$ and  $U^+_{r}$.
Considering the jump condition (JTC) in \eqref{jc}, and assuming that the states are all subsonic, we then get the following system of equations
\begin{equation}\label{eq:1in1_Rpb}
\left\{\begin{array}{c}
z_1(\rho^*_\ell,q^*_\ell) = z_1(\rho^-_{\ell},q^-_{\ell}),
\smallskip\\
z_2(\rho^*_r,q^*_r) = z_2(\rho^+_{r},q^+_{r}),
\smallskip\\
q^*_\ell = q^*_r = \kappa(\rho^*_\ell -\rho^*_r),
\end{array}\right.
\end{equation}
subject to the constraints 
\begin{equation}\label{eq:subsonic_lim}
\left\{
\begin{array}{c}
q_s^-(\rho^*_{\ell}) < q^*_{\ell}< q_s^+(\rho^*_{\ell}),
\medskip\\
 q_s^-(\rho^*_{r}) < q^*_{r}<q_s^+(\rho^*_{r}).
\end{array}\right.
\end{equation}

\begin{figure}[htbp!]
\begin{center}
\begin{tikzpicture}
\draw[-{Latex[length=3mm]}]  (0,-3.5) -- (0,3.5) ;
\draw[-{Latex[length=3mm]}]  (-0.5,0) -- (2.5,0) ;
\draw [thick,domain=0:2] plot (\x,{2*(0.5*\x*sqrt(\x))});
\draw [thick,domain=0:2] plot (\x,{2*(-0.5*\x*sqrt(\x))});
\draw [blue,thick,domain=0:2] plot (\x,{2*(\x-\x*sqrt(\x))});
\draw [red,thick,domain=0:2] plot (\x,{2*(-1.5*\x+\x*sqrt(\x))});
\draw (-.2,2.5) node[above]{$q$};
\draw (2.7,-0.2) node[above]{$\rho$};
\draw (1.6,-.8) node {$\bullet$} ;
\draw (2.5,-1) node[above]{$(\rho_{I}, q_{I})$};
\draw [blue] (.5,.6) node[above]{$q_{\ell}^*$};
\draw [red](.5,-1.6) node[above]{$q_{r}^*$};
\draw (2.5,2.5) node[above]{$q_{s}^+$};
\draw (2.5,-3) node[above]{$q_{s}^-$};
\draw [dashed]   (0,0.3) -- (2.5,0.3) ;
\draw (-0.5,0.2) node[above]{$q_{max}^*$};
\draw [dashed]   (0,-1.02) -- (2.5,-1.02) ;
\draw (-0.5,-1.2) node[above]{$q_{min}^*$};
\end{tikzpicture}
\end{center}
\caption{Riemann problem in the subsonic case with $p(\rho)=p_0\rho^2$ (i.e. $\gamma=2$): sonic curves $q_{s}^-$ and $q_{s}^+$ defined at Eq.\eqref{subsonic_curves} delimiting the subsonic region and Riemann invariant curves  $q_{\ell}^*$ and $q_{r}^*$ defined at Eq.\eqref{eq:curve_cond_Rpb}  }
\label{fig:junction_Rpb_curves}
\end{figure}
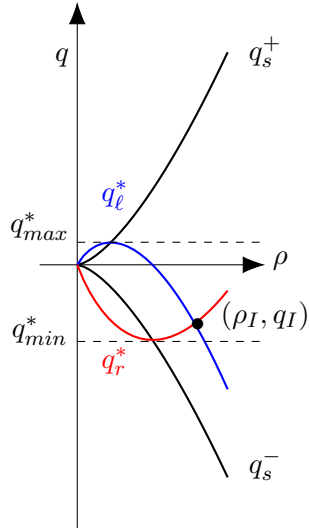

 Following Eq.  \eqref{eq:solRinv}, we define the Riemann invariant curves
\begin{equation}\label{eq:curve_cond_Rpb}
\left\{\begin{array}{c}
q_\ell^*(\rho) = z_1(\rho^-_{\ell},q^-_{\ell})\rho  - \alpha_\gamma \rho\sqrt{p^\prime(\rho)},
\smallskip\\
q_r^*(\rho) = z_2(\rho^+_{r},q^+_{r})\rho  + \alpha_\gamma \rho\sqrt{p^\prime(\rho)}.
\end{array}\right.
\end{equation}

We notice that in the case where $1 <\gamma \leq 3$ and in the subsonic case, then $\mu_1(\rho^-_{\ell},q^-_{\ell}) < 0 <\mu_2(\rho^+_{r},q^+_{r})$, implies $z_1(\rho^-_{\ell},q^-_{\ell}) > 0 >z_2(\rho^+_{r},q^+_{r})$. Moreover, $q_\ell^*$ is a strictly concave function of $\rho$ and $q_r^*$ is strictly convex. 
 Thus condition \eqref{eq:condRP} holds and there exists a unique intersection point, apart from $0$,  for the two curves of functions $q_\ell^*$ and $q_r^*$.
From now on, we shall call this point $U_I=(\rho_I,q_I)$ and it satisfies $q_{I}=q_\ell^*(\rho_I)=q_r^*(\rho_I)$. 

Now, in the case of  the conditions \eqref{eq:1in1_Rpb} we consider here, we may allow a discontinuity on densities, i.e. $\rho_\ell^*\ne\rho_r^*$.
In the next lemma, we shall show that system \eqref{eq:1in1_Rpb}, subject to \eqref{eq:subsonic_lim}, 
admits a unique solution,  which lies therefore  in the \textit{subsonic} region.
\begin{lemma} We assume $p$ of the form \eqref{eq:pfunct} with $1 <\gamma \leq 3$. Let be given two points $U^-_{\ell}=(\rho^-_{\ell},q^-_{\ell})$ and $U^+_{r}=(\rho^+_{r},q^+_{r})$ such that $q_s^-(\rho^-_{\ell}) < q^-_{\ell}< q_s^+(\rho^-_{\ell})$ and  $q_s^-(\rho^+_{r}) < q^+_{r}<q_s^+(\rho^+_{r})$, then there exists a unique couple of points $U_\ell^*$, $U_r^*$ verifying system \eqref{eq:1in1_Rpb} subject to  constraints \eqref{eq:subsonic_lim}.

\end{lemma}
\begin{proof} First of all,   note that  the solution $q^*=q_\ell^*(\rho_\ell^*) = q_r^*(\rho_r^*)$  to system \eqref{eq:1in1_Rpb} should be sought, see Figure\ref{fig:junction_Rpb_curves},  in the region defined by 
\begin{equation*}
\{ q :  q^*_{min}=\min_{\rho>0}q_r^*(\rho) \leq q \leq \max_{\rho>0}q_\ell^*(\rho)=q^*_{max} \}.
\end{equation*}
Note that  the intersection point $(\rho_{I}, q_{I})$ satisfies $q^*_{min}<  q_{I} < q^*_{max}$.

Then, it is easy to check that the intersection point, apart from $0$, between the sonic curve defined by $q_s^+$ and the curve defined by $q_\ell^*$ 
coincides with the maximum point for $q_\ell^*$. Indeed,
\begin{equation*}
q_s^+(\rho) = q_\ell^*(\rho) \text{ and }  \rho>0 \quad\iff\quad \sqrt{p^\prime(\rho)} = \frac{\gamma-1}{\gamma+1}\,z_1(\rho^-_{\ell},q^-_{\ell})
\end{equation*}
and 
\begin{equation*}
(q_\ell^*)^\prime(\rho)=0 \quad\iff\quad \sqrt{p^\prime(\rho)} = \frac{\gamma-1}{\gamma+1}\,z_1(\rho^-_{\ell},q^-_{\ell}).
\end{equation*}
Analogously, the minimum point for $q_r^*$ coincides with the intersection point, apart from $0$, between the sonic curve defined by $q_s^-$ and the curve defined by $q_r^*$.

Therefore, in the subsonic region, $q_\ell^*$ (resp. $q_r^*$)   is a continuous and strictly decreasing (resp. strictly increasing) function of $\rho$ and we can therefore define  its inverse function $\phi_\ell$ (resp. $\phi_r$), which is also a continuous and strictly decreasing (resp. strictly increasing) function of $q$.

We can now prove that for all $\kappa\in[0,+\infty)$ there exists a unique solution to system \eqref{eq:1in1_Rpb}   subject to subsonic conditions  \eqref{eq:subsonic_lim}. 

Such a solution satisfies 
 $\rho^*_\ell=\phi_\ell(q^*)$, $\rho^*_r=\phi_r(q^*)$ and
$
q^* =\kappa(\rho^*_\ell-\rho^*_r)= \kappa\big(\phi_\ell(q^*)-\phi_r(q^*)\big)
$
and the existence and  unicity of the solution follows from the fact that the function $K$ defined by 
$$
K(q) = \frac{q}{\phi_\ell(q)-\phi_r(q)}
$$
is continuous and strictly monotone, i.e. we have that 
\begin{itemize}
\item[(i)] if $q_I> 0$ then $K:[0,q_I]\rightarrow[0,+\infty)$ 
 is continuous and strictly increasing;
\item[(ii)] if $q_I <0$ then $K:[q_I,0]\rightarrow[0,+\infty)$ 
 is continuous and strictly decreasing.
\end{itemize}
\end{proof}

Staying in the subsonic region is equivalent to limiting the amplitude of the jump, specifically for $\kappa \rightarrow 0$ then $q^*\rightarrow 0$ which belongs to the subsonic region and for $\kappa\rightarrow+\infty$ then $q^*\rightarrow q_I$ and the jump amplitude goes to zero, i.e. $\rho^*_\ell - \rho^*_r \rightarrow 0$.

\begin{remark}
When $\gamma > 3$,  if we impose  condition \eqref{eq:condRP} for the left and right states,  the intersection point $q_I$ exists and system \eqref{eq:1in1_Rpb}  admits a solution, but we cannot guarantee that this solution is subsonic.
\end{remark}

\subsection{Riemann invariants and the junction - the supersonic case}\label{sec:RSsupersonic}
In the supersonic case, the Riemann solver leading to  system \eqref{eq:1in1_Rpb} does not apply. A different formulation is needed, which in general is still an open problem. A number of possible approaches have been proposed in the literature, see for instance \cite{HHW20}, but here we follow the approach proposed in \cite{BP18}. The idea stated in this paper can be resumed as follows~:
\begin{itemize}
\item Assume to be on a left 
arc and to have to the left of the junction a supersonic state $U^-_\ell$ with $\mu_1(\rho^-_\ell,q^-_\ell) > 0$,  namely $q^-_\ell>q_s^+(\rho^-_\ell)$. Introducing the point $\widetilde U^- _\ell= (\widetilde\rho^-_\ell,\widetilde q^-_\ell)$ such that $\widetilde q^-_\ell = q^-_\ell$, $z_1(\widetilde\rho^-_\ell,q^-_\ell)=z_1(\rho^-_\ell,q^-_\ell)$ and $\mu_1(\widetilde U^-_\ell)<0$,
the region of possible junction values $U^*_\ell$ is defined as the set of points verifying
\begin{equation*}
z_1(\rho_\ell^*,q_\ell^*) =  z_1(\widetilde\rho^-_\ell,q^-_\ell) \quad\mbox{with}\quad \rho_\ell^* > \widetilde\rho^-_\ell.
\end{equation*}
\item Assume to be on a right 
arc and to have to the right of the junction a supersonic state $U^+_r$ with $\mu_2(\rho^+_r,q^+_r) < 0$,  namely $q^+_r<q_s^-(\rho^+_r)$. Introducing the point $\widetilde U^+_r = (\widetilde\rho^+_r,\widetilde q^+_r)$ such that $\widetilde q^+_r = q^+_r$, $z_2(\widetilde\rho^+_r,q^+_r)=z_2(\rho^+_r,q^+_r)$ and $\mu_2(\widetilde U^+_r)>0$,
the region of possible junction values $U^*_r$ is defined as the set of points verifying
\begin{equation*}
z_2(\rho_r^*,q_r^*) =  z_2(\widetilde\rho^+_r,q^+_r) \quad\mbox{with}\quad \rho_r^* < \widetilde\rho^+_r.
\end{equation*}
\end{itemize}
To summarize, on the one hand, if we are in the supersonic region with $\mu_1(\rho^-_\ell,q^-_\ell) > 0$ (resp. $\mu_2(\rho^+_r,q^+_r) < 0$) the idea is to connect first the left (resp. right) state $U^-_\ell$ (resp. $U^+_r$) with a shock with zero speed
to the subsonic state $\widetilde U^-_\ell$ (resp.  $\widetilde U^+_r$), i.e. 
$\displaystyle  \frac{q^-_\ell -\widetilde{q}^-_\ell}{\rho^-_\ell -\widetilde{\rho}^-_\ell}=0$ (resp. $\displaystyle \frac{q^+_r-\widetilde q^+_r}{\rho^+_r -\widetilde \rho^+_r}=0$), and then the latter to $U_\ell^*$ (resp. $U_r^*$) with a shock with negative (resp. positive) speed. 
\\

On the other hand, if we are in the supersonic region with $\mu_1(\rho^-_\ell,q^-_\ell) < 0$, that is $q^-_\ell$ is negative and $q^-_\ell < q^-_s(\rho^-_\ell)$ (resp.  $\mu_2(\rho^+_r,q^+_r) > 0$, namely $q^+_r$ is positive and $q^+_r>q^+_s(\rho^+_r)$), see Figure \ref{fig:junction_Rpb_curves_super}, an admissible solution in the sense of the definition of Riemann solver given in \cite{BP18} is to stay on the invariant curve $z_1(\rho_\ell^*,q_\ell^*) =  z_1(\rho^-_\ell,q^-_\ell)$ (resp. $z_2(\rho_r^*,q_r^*) =  z_2(\rho^+_r,q^+_r)$) with the constraint $\mu_2(\rho_\ell^*,q_\ell^*)<0$ (resp.  $\mu_1(\rho_r^*,q_r^*)>0$). 

Moreover, in both cases, the values $U_\ell^*$ and $U_r^*$ have  to verify the 
jump condition \eqref{jc}. 
It is clear that this procedure does not always lead to a solution. However, as stated in \cite{BP18}, the region of admissible junction solutions may be enlarged but losing the uniqueness, and additional conditions have to be added to define the solution at a junction. This discussion is beyond the scope of our work, and we will simply compare simulations in cases where the above procedure is well defined.

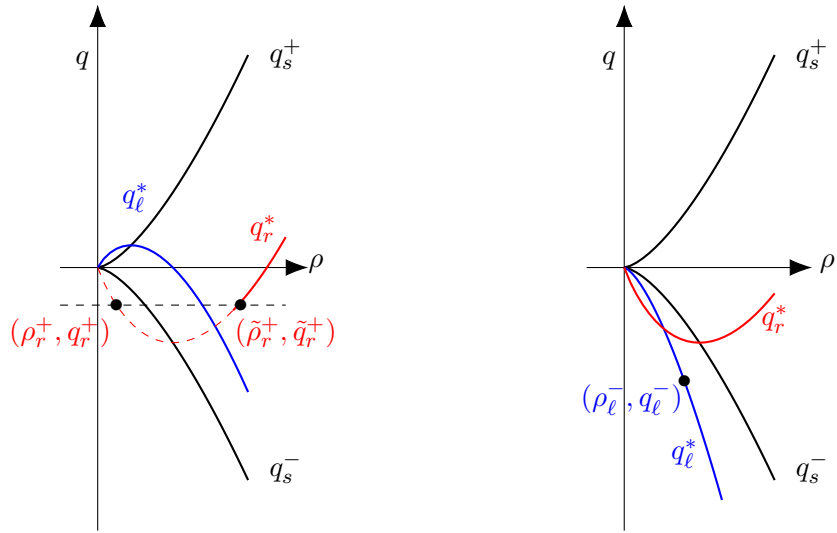
\begin{figure}[htbp!]
\begin{center}
\begin{tikzpicture}
\begin{scope}
\draw[-{Latex[length=3mm]}]  (0,-3.5) -- (0,3.5) ;
\draw[-{Latex[length=3mm]}]  (-0.5,0) -- (2.8,0) ;
\draw [thick,domain=0:2] plot (\x,{2*(0.5*\x*sqrt(\x))});
\draw [thick,domain=0:2] plot (\x,{2*(-0.5*\x*sqrt(\x))});
\draw [blue,thick,domain=0:2] plot (\x,{2*(\x-\x*sqrt(\x))});
\draw [red,dashed,domain=0:2.3] plot (\x,{2*(-1.5*\x+\x*sqrt(\x))});
\draw [red,thick,domain=1.8:2.5] plot (\x,{2*(-1.5*\x+\x*sqrt(\x))});
\draw (-.2,2.5) node[above]{$q$};
\draw (2.9,-0.2) node[above]{$\rho$};
\draw (0.25,-0.5) node {$\bullet$} ;
\draw [red](-0.5,-1.2) node[above]{$(\rho^+_{r}, q^+_{r})$};
\draw (1.9,-0.5) node {$\bullet$} ;
\draw [red](2.5,-1.2) node[above]{$(\tilde\rho^+_{r}, \tilde q^+_{r})$};
\draw [dashed]  (-0.5,-0.5) -- (2.5,-0.5) ;
\draw [blue] (.5,.6) node[above]{$q_{\ell}^*$};
\draw [red](2.2,.2) node[above]{$q_{r}^*$};
\draw (2.5,2.5) node[above]{$q_{s}^+$};
\draw (2.5,-3) node[above]{$q_{s}^-$};
\end{scope}
\begin{scope}[xshift=7cm]
\draw[-{Latex[length=3mm]}]  (0,-3.5) -- (0,3.5) ;
\draw[-{Latex[length=3mm]}]  (-0.5,0) -- (2.5,0) ;
\draw [thick,domain=0:2] plot (\x,{2*(0.5*\x*sqrt(\x))});
\draw [thick,domain=0:2] plot (\x,{2*(-0.5*\x*sqrt(\x))});
\draw [blue,thick,domain=0:1.3] plot (\x,{(-0.1*\x-2*\x*sqrt(\x))});
\draw [red,thick,domain=0:2] plot (\x,{2*(-1.5*\x+\x*sqrt(\x))});
\draw (-.2,2.5) node[above]{$q$};
\draw (2.7,-0.2) node[above]{$\rho$};
\draw (0.8,-1.511) node {$\bullet$} ;
\draw [blue](.1,-2.1) node[above]{$(\rho^-_{\ell}, q^-_{\ell})$};
\draw [blue] (.8,-2.8) node[above]{$q_{\ell}^*$};
\draw [red](2,-1) node[above]{$q_{r}^*$};
\draw (2.5,2.5) node[above]{$q_{s}^+$};
\draw (2.5,-3) node[above]{$q_{s}^-$};
\end{scope}
\end{tikzpicture}
\end{center}
\caption{Riemann problem in the supersonic case with $p(\rho)=p_0\rho^2$ :  on the left,    right  supersonic case with $\mu_2(\rho^+_{r},q^+_{r}) < 0$, i.e. $q^+_{r}< q^-_s(\rho^+_{r})$, and  construction of the point  $\widetilde U^+_r = (\widetilde\rho^+_r,\widetilde q^+_r)$. On the right, left supersonic case with $\mu_1(\rho^-_{\ell},q^-_{\ell}) < 0$, i.e. $q^- _{\ell}< q^-_s(\rho^-_{\ell})$; here, we stay on the invariant curve $q_\ell^*$.}
\label{fig:junction_Rpb_curves_super}
\end{figure}

\section{A relaxation scheme for the jump junction condition - case of a 2-arc network}\label{DEJunctionCondition}

For the sake of simplicity, we  explain in this section  the numerical scheme and the numerical jump junction condition we develop in the case of a simple network, composed of two arcs. An extension to more general networks will be detailed in Section~\ref{Extensions}.

\subsection{A relaxation scheme on a single interval}

Now, let us describe the numerical scheme we use to approximate  \eqref{GasDynamicLeft}--\eqref{jc}. To begin with, we use a relaxation scheme, first introduced in \cite{JX}, see also \cite{AN96, NA98, AN2000, Bou04},   which we will see, away from the junction, is just another interpretation of the  HLL scheme introduced in \cite{HLL}.

We rewrite our equations in the vector form already presented in Section \ref{RIJunctionCondition}, the same as equation \eqref{eq:GenSys}:

 \begin{equation}\label{eq:vector}
\partial_{t} U+\partial_{x}F(U)=0 , 
\end{equation}
where $U \in \R^2$ is the vector of unknowns,
and the function $F : \R^2 \to \R^2$ is the flux of the system. 

Now, we introduce its relaxation approximation, namely a  semilinear hyperbolic system with a singular perturbation parameter $\epsilon$:
\begin{equation}\label{relaxmodel}
\left\{\begin{array}{l}
\ds{\partial_{t}f_1^\epsilon+ c_1\partial_{x} f_1^\epsilon =\frac{1}{\epsilon}\left(M_1(U^\epsilon)-f_1^\epsilon\right),}\\ \\
\ds{\partial_{t}f_2^\epsilon+ c_2\partial_{x} f_2^\epsilon =\frac{1}{\epsilon}\left(M_2(U^\epsilon)-f_2^\epsilon\right).}
\end{array}\right.
\end{equation}
Here $c_1,\ c_2$ are real constants such that $c_1<c_2$,  $f_i^\epsilon\in \R^2$, for $i=1,2$, and $U^\epsilon:= f_1^\epsilon+f_2^\epsilon$. The functions $M_i(U): \R^2 \to \R^2$, for $i=1,2$, are called Maxwellians for the system and are vector functions 
  satisfying the two following relations~:
\begin{equation}\label{maxw}
\left\{
\begin{array}{ll}
M_{1}(U)+M_{2}(U)&=U,
\smallskip\\ \\
c_{1}M_{1}(U)+ c_{2}M_{2}(U)&=F(U).
\end{array}
\right.
\end{equation}                                      
From \eqref{maxw}, we can obtain explicit expressions for $M_{1}(U)$ and $M_{2}(U)$ as
\begin{equation}\label{eq:maxwellian}
M_{1}(U)=\frac{c_{2}U-F(U)}{c_{2}-c_{1}} \quad\text{ and }\quad M_{2}(U)=\frac{c_{1}U-F(U)}{c_{1}-c_{2}}.
\end{equation}
Notice that, setting $V^\epsilon=c_1 f_1^\epsilon +c_2 f_2^\epsilon$, from system \eqref{relaxmodel} we easily obtain
\begin{equation*}
\left\{\begin{array}{l}
\ds{\partial_{t}U^\epsilon+ \partial_{x} V^\epsilon =0,}\\ \\
\ds{\partial_{t}V^\epsilon+\partial_{x}\left( (c_1+c_2)V^\epsilon -c_1c_2 U^\epsilon\right)=\frac{1}{\epsilon}\left(F(U^\epsilon)-V^\epsilon\right).}
\end{array}\right.
\end{equation*}
So, as $\epsilon$ goes to zero, we expect $U^\epsilon$ converges to the solution to \eqref{eq:vector}, as shown for the scalar case in \cite{NA98}, see also \cite{serre00, Bou99, Bianchini06} for the system case in one space dimension. Clearly, when $c_1=-\lambda<0<c_2=\lambda$, for some real constant $\lambda>0$, we recover the classical Jin-Xin model of relaxation \cite{JX}.  

Notice that, as shown in \cite{Bou99} (see also \cite{AN2000}), a necessary condition to have the stability for this approximation is that the Jacobians of the Maxwellian functions $M_i(U)$, are strictly positive matrices for $i=1,2$ and all values of $U$ under consideration. Using the explicit expression  \eqref{eq:maxwellian}, it is easy to see that the 
so-called subcharacteristic condition has to be verified. 
Namely, if $\mu(U)$ is a real eigenvalue of the Jacobian matrix $F'(U)$, we need to have that 
\begin{equation}	\label{subchar}
c_1\leq \mu(U) \leq c_2 \text{ for all } U \text{ under consideration}.
\end{equation}
This condition will be required not only to assess the consistency of the correspondent scheme, but also its stability. 

The strategy to use this relaxation approach to create a numerical scheme for equation \eqref{eq:vector}, is in general to solve numerically at each time step the homogeneous part of system \eqref{relaxmodel}
\begin{equation*}
\left\{\begin{array}{l}
\ds{\partial_{t}f_{1}+c_{1} \partial_{x} f_{1}=0,}\\ \\
\ds{\partial_{t}f_{2}+c_{2}\partial_{x} f_{2}=0},
\end{array}\right.
\end{equation*}
and then to reconstruct the solution to \eqref{eq:vector} from a suitable projection of the functions $(f_1,f_2)$ on the Maxwellian functions $M_1(U), M_2(U)$. This strategy has been fully implemented in \cite{AN2000} in the framework of Finite Difference schemes, see also the recent developments in the framework of Lattice Boltzmann schemes in \cite{Aregba24}. In the following, as already done in \cite{NRT}, we prefer to follow the Finite Volume approach of \cite{Bou04}. So, let us introduce some notations to present our scheme. 

We begin by describing the discretization process on a single interval, which will be  the numerical strategy that we will use both on arc $\ell$ and on arc $r$. 

Let us denote by $\Dx$ the  spatial discretization step.
We discretize the interval  using $J$  cells  
$$
\mathcal C_{j}=[x_{j-1/2},x_{j+1/2}], \, 1\leq j\leq J 
$$
centered at the nodes $x_{j}=(j-\frac 1 2)\Dx, \, 1\leq j\leq J$.

We also denote by $\Dt$ the time step. Discrete times are denoted by $t_{n}=n \Dt, n \in \mathbb{N}$. The time index will be omitted when possible for  the sake of simplicity.

Consider system \eqref{GasDynamic}  under the form \eqref{eq:vector}.
We define $U^n_{j}$ to be the approximation of the mean value of vector $U$ on cell $\mathcal C_{j}$ at time $t_{n}$, that is to say 
$\ds U^n_{j} \sim \frac{1}{\Dx}\int_{\mathcal C_{j}} U(t^n, x) dx. $
According to the finite volume framework, an explicit  conservative finite volume scheme for discretizing equation  \eqref{eq:vector} can be expressed as follows
\begin{equation}\label{schemenum}
U^{n+1}_j = U^{n}_j - \frac{\Dt}{\Dx}\left(\mathcal{F}^n_{j+1/2} - \mathcal{F}^n_{j-1/2}\right),
\end{equation}
with $\mathcal{F}^n_{j+1/2}=\mathcal{F}(U^n_j,U^n_{j+1})$ for an explicit three point scheme. The scheme is therefore defined by the function $\ds \mathcal{F}$ such that $\ds \mathcal{F} (U^{-},U^{+})=\left(
\begin{array}{l}
\mathcal{F}^{\rho}(U^{-},U^{+})\\
\mathcal{F}^{q}(U^{-},U^{+})
\end{array}
\right) \in \mathbb{R}^2$, which is called  the \textit{numerical flux}.

We now describe the  scheme we will use on each arc, which  will be equivalent to the HLL scheme. Here we follow the presentation given in \cite{Bou04}.

Consider again the homogeneous counterpart to the  \textit{relaxation system} \eqref{relaxmodel},  with $f_{i}=f_{i}(x,t)\in \R^2$ for $i=1,2$,
\begin{equation}\label{eq:fsyst}
\left\{\begin{array}{l}
\ds{\partial_{t}f_{1}+c_{1} \partial_{x} f_{1}=0,}\\ \\
\ds{\partial_{t}f_{2}+c_{2}\partial_{x} f_{2}=0},
\end{array}\right.
\end{equation}
where $c_1<c_2$ are two speeds to be chosen later on.

Noticing that system \eqref{eq:fsyst} can be solved explicitely, we obtain an approximate solution to system  \eqref{eq:vector} as a function of the solution to system  \eqref{eq:fsyst}.
Indeed, an exact Riemann solver for \eqref{eq:fsyst} gives
\begin{equation}\label{eq:fRsolver}
\mathcal{R}(x/t,f^-,f^+) = \left\{
\begin{array}{lcr}
(f^-_1,f^-_2) & \mbox{ if } & x/t < c_1,
\\
(f^+_1,f^-_2) & \mbox{ if } & c_1 < x/t < c_2,
\\
(f^+_1,f^+_2) & \mbox{ if } & c_2 < x/t,
\end{array}
\right.
\end{equation}
with $\ds f^{\pm}_i\sim M_i(U^{\pm}), \, i=1,2$.
Thus 
we get the following approximate 
HLL numerical flux for the scheme in conservative form \eqref{schemenum}
\begin{equation}\label{flux}
\mathcal{F}(U^-, U^+)=\\
\left\{\begin{array}{ll}
c_{1}M_{1}(U^-)+ c_{2}M_{2}(U^-)=\ds{F(U^-),} & \text{ if } 0<c_{1}, \\
\ds{c_{1}M_{1}(U^+)+ c_{2}M_{2}(U^-)=}
\ds{\frac{c_{2} F(U^-)-c_{1} F(U^+)}{c_{2}-c_{1}}}&+\ds{\frac{c_{2}c_{1}}{c_{2}-c_{1}}(U^+-U^-)}\\& \qquad \qquad \text{ if }  c_{1}<0<c_{2},\\
c_{1}M_{1}(U^+)+ c_{2}M_{2}(U^+)=\ds{F(U^+),} & \text{ if }  c_{2}<0.
\end{array}\right.
\end{equation}
Note that for stability reasons and to have some entropy dissipation, $c_{1}$ and $c_{2}$ have to be chosen accordingly to the inequality \eqref{subchar}.

From now on, to deal with the transmission condition (JTC), we  always stay in the case $c_1<0<c_2$. For simplicity reasons,  we set $c_1=-c_2=\lambda$ with $\lambda> max_U |\mu(U)|>0$, where $\mu(U)$ is the maximal eigenvalue of $F'(U)$, even if the same considerations could be made in the general case. The numerical 
flux \eqref{flux} therefore reduces to 
\begin{equation}\label{flux2}
\mathcal{F}(U^-, U^+)=
\ds{\frac{ F(U^-)+F(U^+)}{2}-\frac{\lambda}{2}(U^+-U^-)}
\end{equation}
and the Maxwellian functions \eqref{eq:maxwellian} to 
\begin{equation}\label{eq:maxwellian_simpl}
M_{1}(U)=\frac 1 2 \left( U -\frac{F(U)}{\lambda} \right) \quad\text{ and }\quad M_{2}(U)=\frac 1 2 \left( U+\frac{F(U)}{ \lambda}\right).
\end{equation}

Finally, let us describe the discrete numerical boundary conditions in the case of a single interval. We assume that system  \eqref{eq:vector} is complemented with Neumann boundary conditions on both extremities, which guarantee the mass conservation,  that is to say  $\ds \partial_{x} \rho (t, 0)= q (t, 0)=0 $ and $\ds \partial_{x} \rho (t, L)= q (t, L)=0 $  for all $t\geq 0$.

In the discrete setting, we have to define the quantities in the external ghost cells of the network,  
$\mathcal  C_{0}$ and $\mathcal C_{J+1}$.
To ensure the mass conservation also at the discrete level, we impose on both boundaries    
\begin{equation*}
\left\{
\begin{aligned}
\mathcal{F}^{\rho}(U_{0},U_{1})&=\frac 1 2 (q_{0}+q_{1})-\frac \lambda 2 (\rho_{1}-\rho_{0})=0,\\
\mathcal{F}^{\rho}(U_{J},U_{J+1})&=\frac 1 2 (q_{J}+q_{J+1})-\frac \lambda 2 (\rho_{J+1}-\rho_{J})=0,
\end{aligned}
\right.
\end{equation*}
  see Eq. \eqref{flux2}, which is satisfied if 
\begin{equation}\label{bord-gauche}
\rho_{0}=\rho_{1} \text{ and } q_{0}=-q_{1}
\end{equation}
and
\begin{equation}\label{bord-droit}
\rho_{J}=\rho_{J+1} \text{ and } q_{J}=-q_{J+1}.
\end{equation}

\subsection{The relaxation  scheme on a 2-arc network }\label{HLL2arcs}

Let us now go back to the case of a simple network with two arcs sketched in Figure \ref{fig:network_2arcs}. 

We still  denote by $\Dx$ the  spatial discretization step, that we assume  to be the same for both arcs for simplicity reasons and by $\Dt$ the time step.
We discretize the left arc (resp. right arc)  using $J_{\ell}$ (resp. $J_{r}$)  cells  
$$
\mathcal C_{j, \ell}=[x_{j-1/2, \ell},x_{j+1/2,\ell}], \, 1\leq j\leq J_{\ell} \quad\mbox{ (resp. }\mathcal C_{j, r}=[x_{j-1/2, r},x_{j+1/2,r}], \, 1\leq j\leq J_{r}\mbox{)}
$$
centered at the nodes $x_{j, \ell}=(j-\frac 1 2)\Dx, \, 1\leq j\leq J_{\ell}$ (resp. $x_{j,r}=(j-\frac 1 2)\Dx, \, 1\leq j\leq J_{r}$).

On each arc we solve the system \eqref{eq:vector} by the approximation \eqref{schemenum}-\eqref{flux2} with $\lambda=\lambda_\ell$ and $\lambda=\lambda_r$ for arc $\ell$ and arc $r$ respectively. Again, more general speeds could be considered, but we limit this presentation to the simplest case. 

Boundary conditions at the extremities of the network are simply deduced from \eqref{bord-gauche}-\eqref{bord-droit}, that is to say
\begin{equation}\label{bord-gauche2}
\rho_{0, \ell}=\rho_{1, \ell} \text{ and } q_{0, \ell}=-q_{1, \ell}.
\end{equation}
and 
\begin{equation}\label{bord-droit2}
\rho_{J_{r},r}=\rho_{J_{r}+1, r} \text{ and } q_{J_{r}, r}=-q_{J_{r}+1, r}.
\end{equation}

\subsection{Numerical  treatment of the junction}\label{numerical_junction}

The current objective is to establish 
precise and consistent conditions at the junction of the two arcs, that is to say  \textit{consistent} values for the ghost cells at the junction, $\mathcal  C_{J_{\ell}+1, \ell}$ and $\mathcal  C_{0, r}$. 

\begin{notations} From now on we shall use the following notations, see notations in blue in Figure \ref{fig:notation}~:
\begin{itemize}
\item $U^*_{\ell}=(\rho^*_{\ell},q^*_{\ell})$ and $U^*_{r}=(\rho^*_{r},q^*_{r})$ represent the unknown values at the junction on the ghost cell $\mathcal  C_{J_{\ell}+1, \ell}$ on arc $\ell$ and on the cell $\mathcal  C_{0, r}$ on arc $r$, respectively;
\item $U^-_{\ell}=(\rho^-_{\ell},q^-_{\ell})=(\rho_{J_\ell,\ell},q_{J_\ell,\ell})$ denotes the values on the last cell $\mathcal  C_{J_{\ell}, \ell}$ on arc $\ell$;
\item $U^+_{r}=(\rho^+_{r},q^+_{r})=(\rho_{1,r},q_{1,r})$ stands for values on the first cell $\mathcal  C_{1,r}$ on arc $r$,
\item the notation $v_{j,a}^n$ will be employed to denote  variable $v$ ($v=\rho$ or $q$) with $j$ the cell index, $a$ denotes the arc and can take the values $a=\ell$ or $a=r$ and $n$ denotes the time index;
\item in the relaxation setting, $f^{-}_{i,\ell}$ and $f^{+}_{i,r}$, $i=1,2$, denote the left and right states and $f^{*}_{i,\ell} $ and  $f^{*}_{i,r}$,  $i=1,2$, denote the left and right intermediate states, see Figure \ref{fig:notation}.
\end{itemize}
\end{notations}
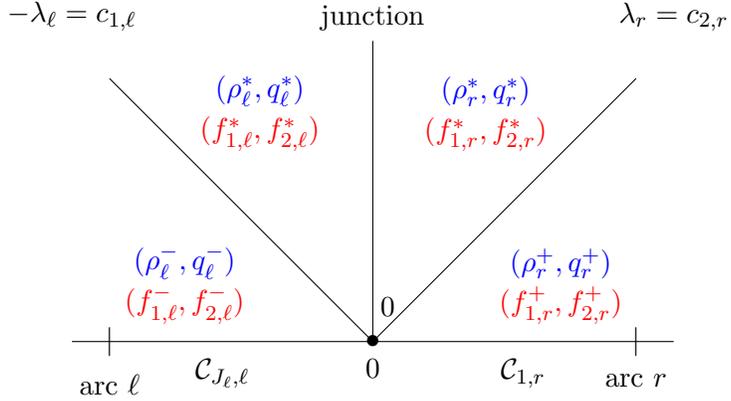
\begin{figure}[htbp!]
\begin{center}
\begin{tikzpicture}
\draw  (-4,0) -- (4,0) ;
\draw  (0,0) -- (0,4) ;
\draw (0,0) node {$\bullet$} ;
\draw (-3.5,0) node {$|$} ;
\draw (3.5,0) node {$|$} ;
\draw (0,4) node[above]{junction} ;
\draw (.2,.2) node[above]{$0$} ;
\draw (0,-.1) node[below]{$0$} ;
\draw (-3.5,-0.3) node[below]{arc $\ell$} ;
\draw (3.5,-0.3) node[below]{arc $r$} ;
\draw (-2,-0.1) node[below]{$\mathcal  C_{J_{\ell}, \ell}$ } ;
\draw (2,-0.1) node[below]{$\mathcal  C_{1,r}$} ;
\draw  (0,0) -- (3.5,3.5) ;
\draw  (0,0) -- (-3.5,3.5) ;
\draw (4,4) node[above]{$\lambda_{r}=c_{2,r}$} ;
\draw (-4,4) node[above]{$-\lambda_{\ell}=c_{1,\ell}$} ;
\draw (-2.5,0.7) node[above]{\textcolor{blue}{$(\rho^-_{\ell},q^-_{\ell})$}} ;
\draw (2.5,0.7) node[above]{\textcolor{blue}{$(\rho^+_{r},q^+_{r})$}} ;
\draw (-1.5,3) node[above]{\textcolor{blue}{$(\rho^*_{\ell},q^*_{\ell})$}} ;
\draw (1.5,3) node[above]{\textcolor{blue}{$(\rho^*_{r},q^*_{r})$}} ;
\draw (-2.5,0.1) node[above]{\textcolor{red}{$(f_{1,\ell}^{-},f_{2,\ell}^{-})$}} ;
\draw (2.5,0.1) node[above]{\textcolor{red}{$(f_{1,r}^{+},f_{2,r}^{+})$}} ;
\draw (-1.5,2.4) node[above]{\textcolor{red}{$(f_{1,\ell}^{*},f_{2,\ell}^{*})$}} ;
\draw (1.5,2.4) node[above]{\textcolor{red}{$(f_{1,r}^{*},f_{2,r}^{*})$}} ;
\end{tikzpicture}
\end{center}
\caption{Notations for the quantities at the junction : in blue,  $(\rho^-_{\ell},q^-_{\ell})$ and  $(\rho^+_{r},q^+_{r})$ are the left and right states and $(\rho^*_{\ell},q^*_{\ell})$ and  $(\rho^*_{r},q^*_{r})$ the two intermediate states. In red, the two corresponding   vectors $f_{1}$ and $f_{2}$ for the relaxation approximation.} 
\label{fig:notation}
\end{figure}

Therefore, to complement the scheme described in Section \ref{HLL2arcs}, we need now to find some values for the four quantities at the junction, i.e. $q^*_{\ell}, q^*_{r} , \rho^*_{\ell} $ and $\rho^*_{r}$.
As stated above, we need to impose some transmission conditions at the junction. We consider a discrete version of 
the equality of the fluxes \eqref{jcegal} in order to have the mass conservation at the junction 
and the (JTC) transmission conditions \eqref{jc} linking the fluxes with the difference of densities.
Namely, with the notations introduced above, 
\begin{equation*}
	 q^*_{\ell} = q^*_{r} = \kappa (\rho^*_{\ell}-\rho^*_{r}).
\end{equation*}
Since we have a total of $4$ unknowns to be found at the junction, we need to impose two supplementary conditions.
Let us now consider the relaxation system \eqref{eq:fsyst} on both arcs $\ell$ and $r$ with $c_{1,_\ell}=-c_{2,_\ell}=\lambda_\ell>0$ and $c_{1,r}=-c_{2,r}=\lambda_r>0$ respectively. By the exact Riemann solver \eqref{eq:fRsolver} at the junction and considering that the junction stands for a shock with $0$-velocity, we get
\begin{align}
f^{*}_{2,\ell} &= f^{-}_{2,\ell} \mbox{ on arc }\ell, \label{eq:fRsolJ_sx}
\\
f^{*}_{1,r} &= f^{+}_{1,r} \mbox{ on arc }r, \label{eq:fRsolJ_dx}
\end{align} 
where here $f^{*}_{i,\ell} $ (resp. $f^{*}_{i,r}$), $  i=1,2 $ represents the (unknown) value of the solution for $-\lambda_{l}<x/t<0$ (resp. for $0<x/t<\lambda_{r}$), $f^{-}_{i,\ell}$ is the solution for $x/t<-\lambda_\ell$ on arc $\ell$ and $f^{+}_{i,r}$ the solution for $x/t>\lambda_r$ on arc $r$. 

Switching to the first components of maxwellian functions given by \eqref{eq:maxwellian_simpl}, we can rewrite \eqref{eq:fRsolJ_sx} and \eqref{eq:fRsolJ_dx} as
\begin{equation*}\left\{
\begin{aligned}
		\frac{1}{2}\left(\rho^*_{\ell}+\frac{q^*_{\ell}}{\lambda_{\ell}}\right) &= \frac{1}{2}\left(\rho^-_{\ell}+\frac{q^-_{\ell}}{\lambda_{\ell}}\right),
		\\
\frac{1}{2}\left(\rho^*_{r}-\frac{q^*_{r}}{\lambda_{r}}\right) &= \frac{1}{2}\left(\rho^+_{r}-\frac{q^+_{r}}{\lambda_{r}}\right),
	\end{aligned}\right.
	\end{equation*}
	which gives us the two extra conditions at the junction we were looking for.
	
In summary, at the junction, we need to solve the following linear system for the unknowns $q^{*}=q^*_{\ell} = q^*_{r},\; \rho^*_{\ell}$ and $\rho^*_{r}$~:
\begin{equation}\label{eq:junc1}
\left\{\begin{array}{l}
q^{*}= \kappa (\rho^*_{\ell}-\rho^*_{r}),
\medskip\\
\rho^*_{\ell}+\ds\frac{q^{*}}{\lambda_{\ell}} = \rho^-_{\ell}+\frac{q^-_{\ell}}{\lambda_{\ell}},
\medskip\\
\rho^*_{r}-\ds\frac{q^{*}}{\lambda_{r}} = \rho^+_{r}-\frac{q^+_{r}}{\lambda_{r}}.
\end{array}\right.
\end{equation}
This system  can be solved explicitely as~:
\begin{equation}\label{jonction}
\left\{\begin{array}{l}
q^{*}=q^*_{ \ell}=q^*_{r}=\frac{\kappa}{\kappa(\lambda_{\ell}+\lambda_{r})+\lambda_{\ell}\lambda_{r}} (\lambda_{r}q^-_{\ell}+\lambda_{\ell}q^+_{ r}+\lambda_{\ell}\lambda_{r}( \rho^-_{\ell}-\ \rho^+_{ r})),\\
\rho^*_{ r}=\rho^+_{r}+\frac{1}{\lambda_{r}}(q^{*}-q^+_{ r}),\\
\rho^*_{ \ell}=\rho^-_{\ell}-\frac{1}{\lambda_{\ell}}(q^{*}-q^-_{\ell}).
\end{array}\right.
\end{equation}
The solution can then be used as boundary values for the arcs on the ghost cells at the junction. 

Notice that, the  system given by \eqref{eq:junc1} describing the quantities at the junction, can be recovered looking at the numerical flux  $\mathcal F$ by considering directly  the discretization of  system \eqref{GasDynamicLeft}-\eqref{jc}.

We still use here the previous notations, namely, we denote by 
$U^-_{\ell}=(\rho^-_{\ell},q^-_{\ell})=(\rho_{J_\ell,\ell},q_{J_\ell,\ell})$ (resp. $U^+_{r}=(\rho^+_{r},q^+_{r})=(\rho_{1,r},q_{1,r})$) the values on the last cell $\mathcal  C_{J_{\ell}, \ell}$ on arc $\ell$ (resp. on the first cell $\mathcal  C_{1,r}$ on arc $r$).

At the junction, the mass conservation at the discrete level  gives the condition
\begin{equation}\label{jc_discret_mass}
\mathcal{F}^{\rho}(U^-_{\ell},U^*_{\ell})=\mathcal{F}^{\rho}(U^*_{r},U^+_{r})
 \end{equation}
 that is coupled with the following relations, coming from the discretization of conditions \eqref{jc}~: 
\begin{equation*}
 \mathcal{F}^{\rho}(U^-_{\ell},U^*_{\ell})=\mathcal{F}^{\rho}(U^*_{r},U^+_{r})=\kappa (\rho^*_{\ell}-\rho^*_{r}).
 \end{equation*}
Again, we have $4$ unknowns at the junction and we need two supplementary conditions. We find them by imposing the equality between the numerical flux for the mass equation $F^{\rho}=\rho u=q $ and the numerical discretization for the momentum $q=\rho u$, namely~: 
\begin{equation}\label{fluxJunc}
\mathcal{F}^{\rho}(U^-_{\ell},U^*_{\ell})=q^*_{\ell}, \quad    \mathcal{F}^{\rho}(U^*_{r},U^+_{r})=q^*_{ r}.
\end{equation}

We obtain therefore a linear system of $4$ equations with $4$ unknowns~:
\begin{equation}\label{bc_junction}
\left\{\begin{array}{l}
\mathcal{F}^{\rho}(U^-_{\ell},U^*_{ \ell})=\mathcal{F}^{\rho}(U^*_{ r},U^+_{ r})=\kappa (\rho^*_{ \ell}-\rho^*_{ r}), \\
\mathcal{F}^{\rho}(U^-_{\ell},U^*_{ \ell})=q^*_{\ell},\\
\mathcal{F}^{\rho}(U^*_{ r},U^+_{r})=q^*_{ r}.
\end{array}\right.
\end{equation}
Note that system \eqref{bc_junction}  holds  for any  numerical flux $\mathcal{F}$ and not only for  the HLL flux.

In the case of the HLL numerical flux, the previous system resumes to 
\begin{equation}\label{system_junction}
\left\{\begin{array}{l}
q^{*}=q^*_{\ell}=q^*_{r}
=\kappa (\rho^*_{\ell}-\rho^*_{r}), \\
\ds \frac{ q^-_{\ell}+ q^{*}}{2}-\frac{\lambda_{\ell}}{2}(\rho^*_{\ell}-\rho^-_{\ell})=q^{*},\\
\ds\frac{ q^{*}+ q^+_{r}}{2}-\frac{\lambda_{r}}{2}(\rho^+_{ r}-\rho^*_{ r})=q^{*},\
\end{array}\right.
\end{equation}
which is the same as system~\eqref{eq:junc1}.

\medskip

\begin{remark}
Let us briefly compare the two junction conditions that have been considered in this paper, i.e.  system \eqref{eq:junc1}  and system \eqref{eq:1in1_Rpb}.

\begin{itemize}
\item
Note that system \eqref{eq:junc1}  admits always a solution, both for subsonic and supersonic states, unlike system \eqref{eq:1in1_Rpb}. It can therefore be seen as an extension of system \eqref{eq:1in1_Rpb} that works in both cases. 
\item
Then, Riemann conditions leading to system \eqref{eq:1in1_Rpb} should not preserve the mass conservation at the discrete level, although conditions in system \eqref{eq:junc1} are designed to do so.
\item Finally, the reasoning for Riemann conditions at Subsection \ref{sec:RSsubsonic} only holds for a power $\gamma$ in Eq.\eqref{eq:pfunct} such that $1 <\gamma \leq 3$, whereas system \eqref{eq:junc1}  holds a solution for every $\gamma>1$. Note that the parameter $\gamma$  is only present in Eq.\eqref{eq:junc1} through the values of $\lambda_{\ell}$ and $\lambda_{r}$.
\end {itemize}
\end{remark}

\begin{remark} In the subsonic region, the two last  conditions given in \eqref{eq:junc1} 
may be viewed as a local linearization of relations \eqref{eq:curve_cond_Rpb}.
In fact, substituting \eqref{eq:curve_cond_Rpb} with the tangent lines of the two Riemann invariants in the points $(\rho^-_{\ell},q^-_{\ell})$ and $(\rho^+_{r},q^+_{r})$ respectively, we get
\begin{equation}\label{eq:sysRiemann_tg_lin}
\left\{\begin{array}{l}
q_\ell^* (\rho)= q^-_{\ell} + \mu_1(\rho^-_{\ell},q^-_{\ell}) (\rho-\rho^-_{\ell}),
\smallskip\\
q_r^*(\rho) = q^+_{r} + \mu_2(\rho^+_{r},q^+_{r}) (\rho-\rho^+_{r}).
\end{array}\right.
\end{equation}
On the other hand, the 
 two last equations in \eqref{eq:junc1} can be rewritten as
\begin{equation}\label{eq:rette_cond_mr}
\left\{\begin{array}{l}
q_\ell^* = q^-_{\ell}  - \lambda_{\ell} (\rho_\ell^*-\rho^-_{\ell}),
\smallskip\\
q_r^* = q^+_{r} + \lambda_{r} (\rho_r^*-\rho^+_{r}).
\end{array}\right.
\end{equation}
Then, the two unknown values $U_\ell^*$ and $U_r^*$ belong to the decreasing line 
$q(\rho) = q^-_{\ell} - \lambda_{\ell} (\rho -\rho^-_{\ell})$ and to the increasing line $q(\rho) = q^+_{r} + \lambda_{r} (\rho -\rho^+_{r})$ respectively.
Since we are assuming $\mu_1(\rho^-_{\ell},q^-_{\ell})<0$ and $\mu_2(\rho^+_{r},q^+_{r})>0$, \eqref{eq:sysRiemann_tg_lin} reads as
\begin{equation}\label{eq:sysRiemann_sub_lin}
\left\{\begin{array}{l}
q_\ell^* = q^-_{\ell} - |\mu_1(\rho^-_{\ell},q^-_{\ell})| (\rho_\ell^*-\rho^-_{\ell}),
\smallskip\\
q_r^* = q^+_{r} + \mu_2(\rho^+_{r},q^+_{r}) (\rho_r^*-\rho^+_{r}).
\end{array}\right.
\end{equation}
The two systems \eqref{eq:rette_cond_mr} and \eqref{eq:sysRiemann_sub_lin} differ only in the choice of the angular coefficient of the two lines.
\end{remark}

\subsection{Consistency properties of the numerical scheme}\label{properties}

\subsubsection{Mass conservation}

First we prove the following property, which is equivalent to the continuous mass preservation \eqref{consmass2}.

\begin{prop}
Let us  consider system \eqref{GasDynamicLeft}--\eqref{jc} set on a $2$-arc network displayed in Figure\ref{fig:network_2arcs} with boundary conditions \eqref{bc} on the outer nodes and junction conditions \eqref{jc}.  The numerical scheme \eqref{schemenum} with  discrete boundary conditions \eqref{bord-gauche2}-\eqref{bord-droit2} on the outer nodes and discrete junction conditions \eqref{bc_junction} is mass-preserving, that is to say 
the discrete mass $\ds m^n={h} \sum_{j=1}^{J_{\ell}} \rho_{j,\ell}^n+h \sum_{j=1}^{J_{r}} \rho_{j,r}^n $ is independent of $n \in \N$.
\end{prop}
\begin{proof}
We consider the first component of equation \eqref{schemenum}. The conservation of the discrete mass is given by  the boundary conditions for the outer nodes \eqref{bord-gauche2}-\eqref{bord-droit2} and  by condition \eqref{jc_discret_mass} at the junction.\end{proof}

\subsubsection{Positivity of the solution}

Now, we prove   that the positivity of the solution is preserved in the particular case of HLL numerical flux~:
\begin{prop}\label{positivity}
Let us consider scheme \eqref{schemenum} with flux \eqref{flux2}, boundary conditions  \eqref{bord-gauche2} -- \eqref{bord-droit2} and condition \eqref{system_junction} at the junction in order to discretize system \eqref{GasDynamicLeft}-\eqref{jc} on a $2$-arc network.
We define $\lambda^n_{\ell}>0$ (resp. $\lambda^n_{r}>0$) such that
\begin{equation}  \label{subchar_pos}
\lambda^n_{\ell}\geq\maxd_{1 \leq j \leq J_{\ell}}{|u^n_{j, \ell}|+ \sqrt{p'(\rho^n_{j, \ell})}} (\text{ resp. } \lambda^n_{r}\geq\maxd_{1 \leq j \leq J_{r}}{|u^n_{j, r}|+ \sqrt{p'(\rho^n_{j, r})}}).\end{equation}

 If the initial condition is positive, that is to say 
 \[
  \rho^0_{j,\ell}\geq 0 \text{ for all } 1  \leq j \leq J_{\ell} \,  (\text{ resp. } \rho^0_{j,r} \geq 0 \text{ for all }  1 \leq j \leq J_{r}),
  \]
   then the solution remains positive in time, that is to say 
    \[
  \rho^n_{j,\ell}\geq 0 \text{ for all }  1 \leq j \leq J_{\ell} \, (\text{ resp. } \rho^n_{j,r} \geq 0 \text{ for all }  1 \leq j \leq J_{r}), \text{ for all }  n \in \N.
  \]
 \end{prop}

Notice that condition \eqref{subchar_pos} is just a formulation in this case of condition \eqref{subchar}, which is necessary to obtain the stability of the numerical scheme. 

\begin{proof}
We prove the proposition by induction. 

Assume that $\ds \rho^n_{j,\ell} \geq 0$ for all $1 \leq j \leq J_{\ell}$ (resp. $\ds  \rho^n_{j,r} \geq 0$ for all $1 \leq j \leq J_{r}$).
Since $\lambda^n_{\ell} \geq |u_{j, \ell}^n|$ for all $1 \leq j \leq J_{\ell}$ (resp.  $\lambda^n_{r} \geq |u^n_{j, r}|$ for all $1 \leq j \leq J_{r}$), we use the explicit expressions obtained at Eq.\eqref{jonction}  to prove that $\rho_{r}^{*,n}\geq 0$ and  $\rho_{ \ell}^{*,n}\geq 0$. From these two inequalities, we demonstrate that  $\rho_{j, \ell}^{n+1} \geq 0$ for all $1 \leq j \leq J_{\ell}$ (resp.  $\rho_{j, r}^{n+1} \geq 0$  for all $1 \leq j \leq J_{r}$). 
\end{proof}

\section{ Numerical tests in the case of a 2-arc network}\label{NumericalTests}

In this section, we examine the characteristics of the numerical method proposed in Section \ref{DEJunctionCondition}
and we evaluate its accuracy by examining a simple network consisting of two arcs. 
The pressure is given by the pressure law given at Eq.\eqref{eq:pfunct} with $\gamma=2$, unless another value is specified. 
 We therefore consider the numerical scheme defined by Eq.\eqref{schemenum} with flux \eqref{flux} on the two arcs, conditions \eqref{bord-gauche2} - \eqref{bord-droit2} at the extremities of the network  and conditions $\Jmr$ at the junction; this scheme will be referred as the HLL-$\Jmr$ scheme in the following. 
 
We consider a uniform spatial grid with a fixed spatial step $\Dx=0.05$ on the two arcs with equal length of $L_\ell=L_r=2$. Then, the number of cells verify $J_\ell=J_r=J$. 
The time step is dynamically computed at each time iteration $n$ by the CFL condition
\begin{equation*}
\Dt^n \leq \frac{\Dx}{\max\{\lambda^n_\ell,\lambda^n_r\}},
\end{equation*}
where $\lambda^n_\ell$ and $\lambda^n_r$ have been defined in Proposition \ref{positivity}, such that the subscharacteristic condition \eqref{subchar_pos} is verified.

\subsection{Numerical convergence results for the  numerical scheme HLL-$\Jmr$}
Let the following Gaussian function be given
\begin{equation*}
g(x;\sigma,c) = e^{-(x-c)^2/2\sigma^2}.
\end{equation*}
To study the numerical order of convergence of the scheme applied to the whole network, we set as initial data (at time $t=0$), on the left arc, for $x\in[0,L_\ell]$, 
\begin{equation}\label{eq:reg_init_data_left}
\rho_\ell(0,x) = 1.5 + g(x,0.2,0.8) \mbox{ and } q_\ell(0,x)=0.5,
\end{equation}
and on the right-arc, for $x\in[0,L_r]$,
\begin{equation}\label{eq:reg_init_data_right}
\rho_r(0,x) = 1 + g(x,0.1,0.5) \mbox{ and } q_r(0,x)=0.5.
\end{equation}
This initial condition is subsonic all over the network and verifies at the junction the continuity of the flux and the (JTC) condition $\KK$ with $\kappa=1$, i.e. 
$$
q_\ell(0,L_\ell)=q_r(0,0)= \rho_\ell(0,L_\ell)-\rho_r(0,0) = 0.5.
$$
The accuracy is measured by the $L^1$ norm of the error at the final time $T=1$, defined as
$$
\mbox{err}_{\rho, \ell}=\Dx\ds\sum_{j=1}^{J_{\ell}} | \rho^N_{j,{\ell}} - \rho^{ref}_{j,{\ell}} |,
\quad \mbox{err}_{q, \ell}=\Dx\ds\sum_{j=1}^{J_{\ell}} | q^N_{j,{\ell}} - q^{ref}_{j,{\ell}} |,
$$
with the same definition for arc $r$, where $(\rho^{ref}_\cdot,q^{ref}_\cdot)$ is a reference solution obtained with a small $\Dx\sim 10^{-5}$ on both arcs $\ell$ and $r$.
Considering the system $\Jmr$ at the junction, Table \ref{tab:err_ord_gauss_mr} shows that the conditions imposed at the junction do not affect the convergence order of the scheme, which remains first order.

Furthermore, we show that the system $\Jmr$ does not alter the asymptotic behavior of the solution along the entire 2-arc network.
We consider again the initial data \eqref{eq:reg_init_data_left}--\eqref{eq:reg_init_data_right} and calculate the rate of convergence of the solution to the expected asymptotic state 
\begin{equation}\label{asymptotic}
\bar\rho_{asympt} = \ds\frac{1}{L_\ell+L_r}\int_0^{L_\ell+L_r} \left(\rho_\ell(0,x)+\rho_r(0,x)\right) dx, \; \bar q_{asympt} =0.
\end{equation}
Given N data points $(t_i , e_{\rho, \ell}(t_i ))_{i=1,\cdots, N}$ , we shall define $\gamma$ and $C$ as the solution of the following least square problem,
\begin{equation}\label{eq:opt_prb}
\min_{C,\gamma} \sum_{i=1}^N | \ln{(e_{\rho, \ell}(t_i))} - \ln{(C\,t_i^{-\gamma})} |^2,
\end{equation}
where $e_{\rho, \ell}(t_i) = \| \rho_{\ell}(t_i,\cdot) - \bar\rho_{asympt} \|_{L^1,L^\infty}$, with the same definition for $e_{\rho, r}$, $e_{q, \ell}$ and $e_{q, r}$.
Table \ref{tab:vel_simpt} presents the optimal values of $\gamma$ and $C$. Once again, the system imposed at the junction does not alter the asymptotic behavior in time of the solution,  see \cite{GL}, where it is proved that the solution decays as $t^{-1}$, in the case of a single interval.

\begin{table}[h!]
\centering
\begin{tabular}{|c|cc|cc|cc|cc|}
\hline
 & \multicolumn{2}{|c|}{$\rho_\ell$} &\multicolumn{2}{|c|}{$\rho_r$} & \multicolumn{2}{|c|}{$q_\ell$} & \multicolumn{2}{|c|}{$q_r$}\\
\hline
$\Delta x$ & $\mbox{err}_{\rho, \ell}$  & order & $\mbox{err}_{\rho, r}$ & order & $\mbox{err}_{q, \ell}$ & order & $\mbox{err}_{q,r}$ & order \\
\hline
0.0063 & 0.02651 & 0.70286 & 0.04090 & 0.81492 & 0.05104 & 0.77079 & 0.06378 & 0.63251 
\\\hline
0.0031 & 0.01495 & 0.82633 & 0.02151 & 0.92712 & 0.02727 & 0.90406 & 0.03558 & 0.84207 
\\\hline
0.0016 & 0.00743 & 1.00890 & 0.01030 & 1.06292 & 0.01407 & 0.95527 & 0.01742 & 1.02998 
\\\hline
0.0008 & 0.00334 & 1.15396 & 0.00453 & 1.18482 & 0.00655 & 1.10264 & 0.00775 & 1.16828 
\\\hline
\end{tabular}
\caption{$L^1$ error and order of convergence of the numerical scheme HLL-$\Jmr$ on both arcs starting from the regular initial data \eqref{eq:reg_init_data_left}--\eqref{eq:reg_init_data_right}.}\label{tab:err_ord_gauss_mr}
\end{table}

\begin{table}[h!]
\centering
\begin{tabular}{c|cc|cc|cc|cc|}
\hline
 &\multicolumn{2}{|c|}{$\rho_\ell$} &\multicolumn{2}{|c|}{$\rho_r$} & \multicolumn{2}{|c|}{$q_\ell$} & \multicolumn{2}{|c|}{$q_r$}\\
\hline
&C  & $\gamma$ & C  & $\gamma$ & C  & $\gamma$ & C  & $\gamma$ \\
\hline
$L^1$ & 0.2356 & 1.0248 & 0.2301 & 1.0274 & 0.5642 & 0.9802 & 0.5193 & 0.9544
\\\hline
$L^\infty$ & 0.2295 & 1.0558 & 0.2295 & 1.0558 & 0.5662 & 1.0493 & 0.5446 & 1.0371
\\\hline
\end{tabular}
\caption{Optimal values of $\gamma$ and $C$  given by the least square problem \eqref{eq:opt_prb} with $e_{\rho, \ell}(t_i) = \| \rho_\ell(t_i,\cdot) - \bar\rho_{asympt} \|_{L^1,L^\infty}$ for $i=1,\ldots,N$ computed in $L^1$ and $L^\infty$ norms, same definition for  $e_{\rho, r}$, $e_{q, \ell}$ and $e_{q, r}$.}\label{tab:vel_simpt}
\end{table}

\subsection{Comparison between the two numerical schemes  HLL-$\Jmr$ and  HLL-$\Jrs$}
Now, we also consider the so-called HLL-$\Jrs$ scheme, which is defined as the previous  HLL-$\Jmr$ scheme, unlike the junction condition which is taken here as the solution of system $\Jrs$ defined at Sec.\ref{RIJunctionCondition} thanks to Riemann invariant considerations.

In this section we compare the dynamics obtained by the two schemes HLL-$\Jmr$ and  HLL-$\Jrs$, examining the subsonic and supersonic cases. In all tests, we start from data verifying condition $\KK$ with $\kappa=1$ at the junction, i.e.
$$
q_\ell(0,L_\ell)=q_r(0,0)= \rho_\ell(0,L_\ell)-\rho_r(0,0).
$$
In the following figures, these initial data are drawn in dashed black lines. 

We consider the following test problems:
\begin{description}
\item[Test C1] Case of a subsonic initial datum, constant on both arcs,
\begin{equation*}\left\{\begin{array}{lll}
\rho_\ell(0,x)=4.5,& q_\ell(0,x)=0.5, & \mbox{ for all } x\in[0,L_\ell],
\\
\rho_r(0,x)=4,& q_r(0,x)=0.5,  & \mbox{ for all }x\in[0,L_r].
\end{array}\right.
\end{equation*}
As can be seen in  Figure \ref{fig:test1}, the two schemes HLL-$\Jmr$ and HLL-$\Jrs$ give exactly the same results on both the arcs and the junction. We also notice that solutions are converging  in large times to the asymptotic constant state described at Eq.\eqref{asymptotic}.
\begin{figure}[htbp!]
\centering
\begin{subfigure}[b]{0.45\textwidth}
\includegraphics[scale=0.48]{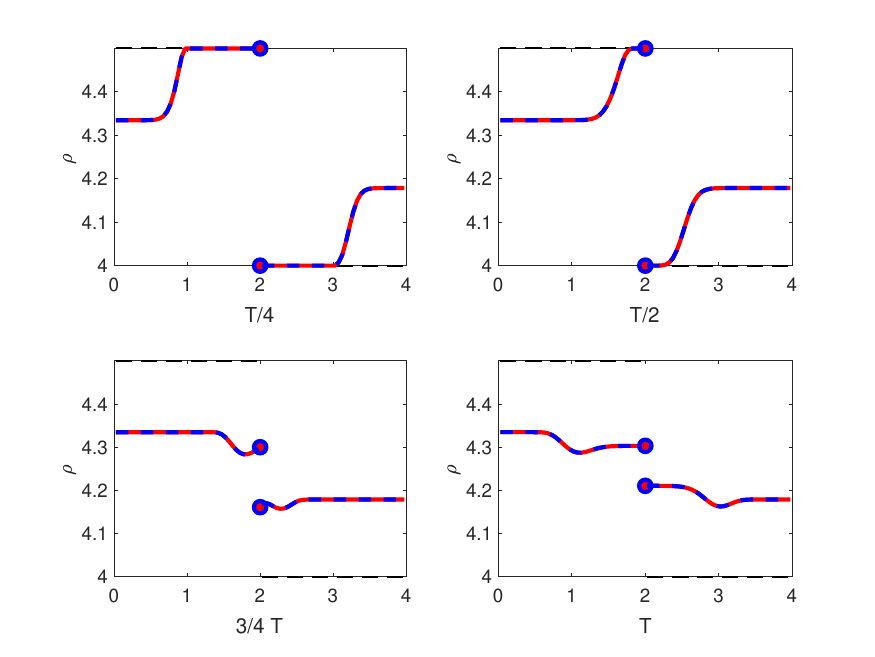}
\caption{Function $\rho$}
\end{subfigure}
\begin{subfigure}[b]{0.45\textwidth}
\includegraphics[scale=0.51]{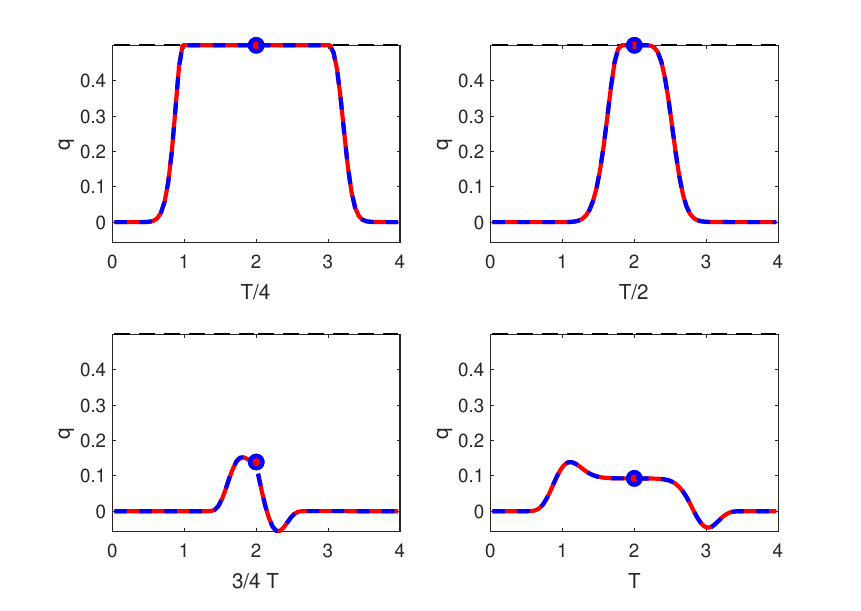}
\caption{Function $q$}
\end{subfigure}
\caption{\textbf{Test C1}. Comparison of the solutions given by scheme HLL-$\Jmr$ (in  solid red line) and by scheme HLL-$\Jrs$ (in dashed blue line), starting from a constant subsonic initial datum on both arcs. The blue and red dots represent the values at the junction. On the left, function $\rho$ and on the right, function $q$  for different times $t=0.25, \, 0.5$ and $0.75$  with final time $T=1$.}
\label{fig:test1}
\end{figure}

\item[Test C2] Case of a  non constant initial datum given by
\begin{equation*}\left\{\begin{array}{lcll}
\rho_\ell(0,x) = 0.4 + g(x,0.2,0.8) &\mbox{ and }& q_\ell(0,x)=0.1, &  \mbox{ for all } x\in[0,L_\ell],
\\
\rho_r(0,x) = 0.3 + g(x,0.1,0.5) &\mbox{ and }& q_r(0,x)=0.1, & \mbox{ for all } x\in[0,L_r],
\end{array}\right.
\end{equation*}
which is subsonic at the junction.
As can be seen in  Figure \ref{fig:test8}, the two approaches show some differences in the calculated values at the junction (look particularly at Figure \ref{fig:test8}-(a) at time $T/4$), but these do not propagate along the arcs and vanish  over time. Here again the solutions converge to the expected asymptotic profile for large times.
\begin{figure}[htbp!]
\centering
\begin{subfigure}[b]{0.45\textwidth}
\includegraphics[scale=0.5]{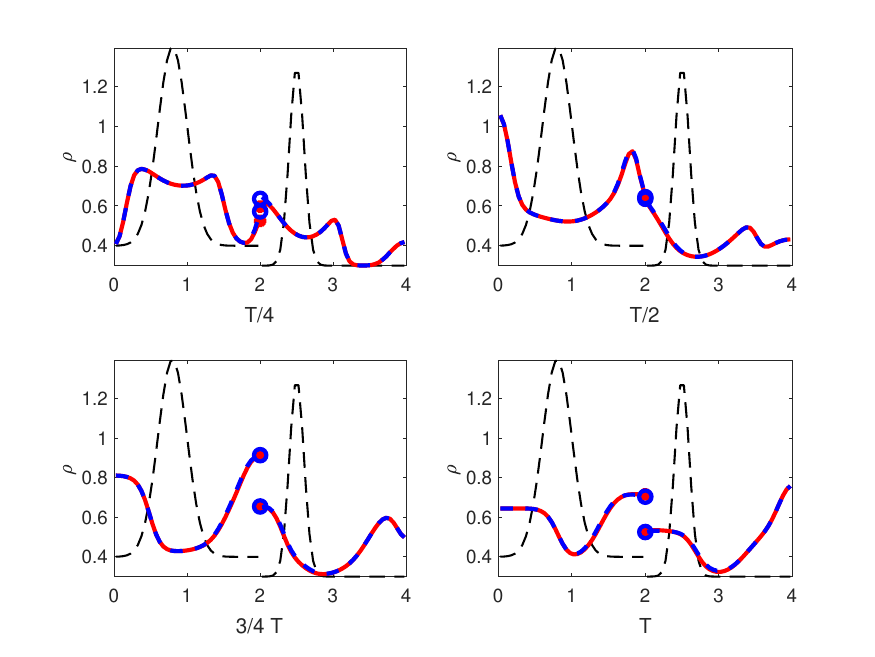}
\caption{Function $\rho$}
\end{subfigure}
\begin{subfigure}[b]{0.45\textwidth}
\includegraphics[scale=0.5]{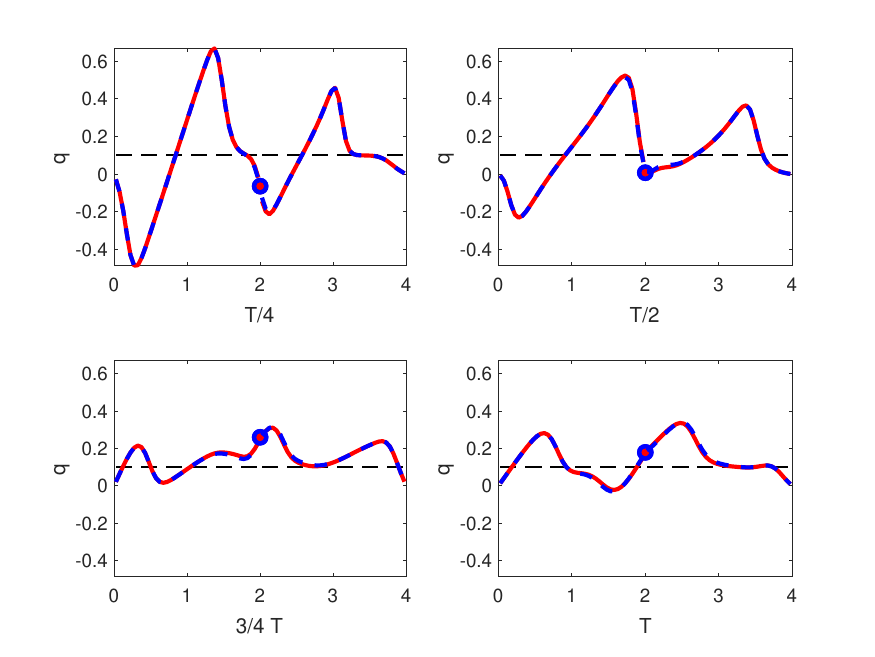}
\caption{Function $q$}
\end{subfigure}
\caption{\textbf{Test C2}. Comparison of the solutions given by scheme HLL-$\Jmr$ (in  solid red line) and by scheme HLL-$\Jrs$ (in dashed blue line), starting from a subsonic  non constant initial datum (in dashed black line). The blue and red dots represent the values at the junction. On the left, function $\rho$ and on the right, function $q$  for different times $t=0.25, \, 0.5$ and $0.75$  with final time $T=1$.
}
\label{fig:test8}
\end{figure}
\item[Test C3] Case of a constant initial datum supersonic to the right of the junction (right-arc) and subsonic on the left,
\begin{equation*}\left\{\begin{array}{lll}
\rho_\ell(0,x)=2.5,& q_\ell(0,x)=0.5, & \mbox{ for all } x\in[0,L_\ell],
\\
\rho_r(0,x)=2,& q_r(0,x)=0.5,  & \mbox{ for all }x\in[0,L_r].
\end{array}\right.
\end{equation*}
 As can be seen in Figure \ref{fig:test3}, on the right arc the solution at the junction remains supersonic and the two algorithms give the same results. This is an example  where the datum on the right is supersonic but with positive velocity, namely $q_r(0,L_r)>0$ and $q_r(0,L_r)>q_s^+(\rho_r(0,L_r))$. In this case, the Riemann based HLL-$\Jrs$ scheme selects the solution on the right invariant curve, see Subsec.\ref{sec:RSsupersonic}, and so does the HLL-$\Jmr$ scheme .
\begin{figure}[htbp!]
\centering
\begin{subfigure}[b]{0.45\textwidth}
\includegraphics[scale=0.5]{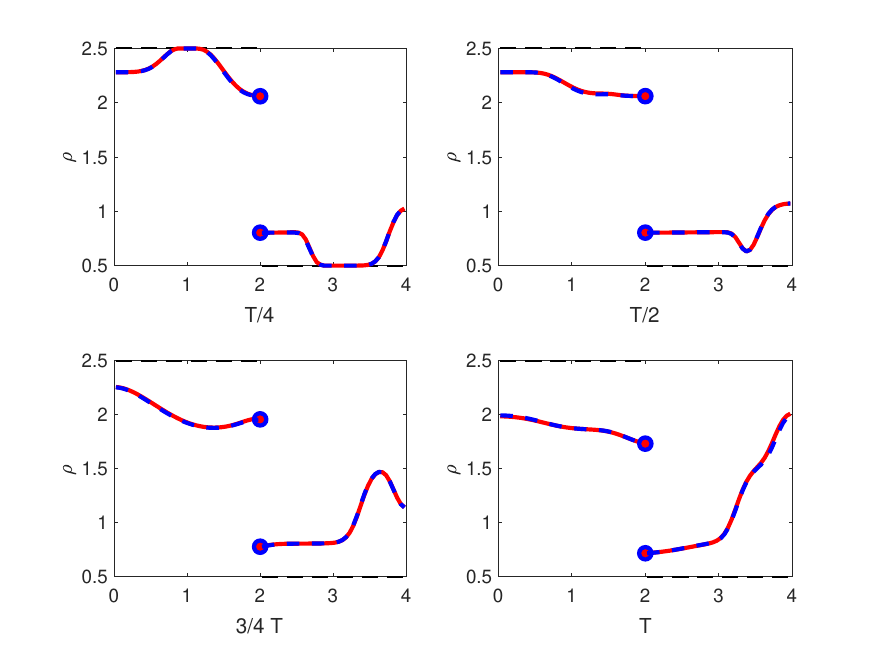}
\caption{Function $\rho$}
\end{subfigure}
\begin{subfigure}[b]{0.45\textwidth}
\includegraphics[scale=0.5]{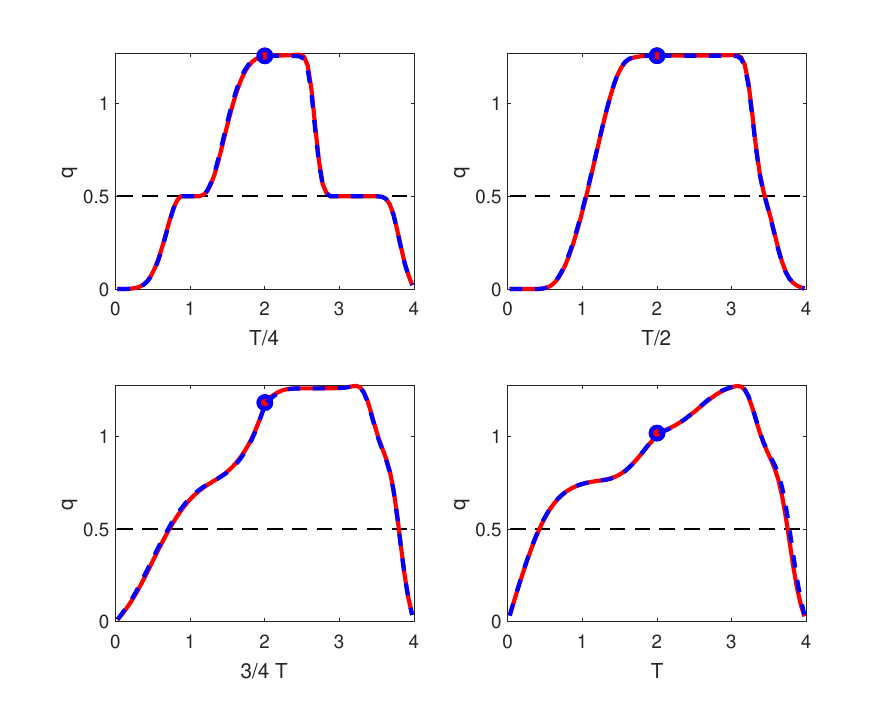}
\caption{Function $q$}
\end{subfigure}
\caption{\textbf{Test C3}. 
Comparison of the solutions given by scheme HLL-$\Jmr$ (in  solid red line) and by scheme HLL-$\Jrs$ (in dashed blue line), starting from a constant initial datum supersonic on the right-arc and subsonic on the left-arc. The blue and red dots represent the values at the junction and the dashed black line the initial value for $q$. On the left, function $\rho$ and on the right, function $q$  for different times $t=0.25, \, 0.5$ and $0.75$  with final time $T=1$.}
\label{fig:test3}
\end{figure}

\item[Test C4] Case of a non constant initial datum given by
\begin{equation*}\left\{\begin{array}{lcll}
\rho_\ell(0,x) = 0.1 + g(x,0.2,0.8) &\mbox{ and }& q_\ell(0,x)=-0.2, & \mbox{ for all } x\in[0,L_\ell],
\\
\rho_r(0,x) = 0.3 + g(x,0.1,0.5) &\mbox{ and }& q_r(0,x)=-0.2, & \mbox{ for all } x\in[0,L_r],
\end{array}\right.
\end{equation*}
which is supersonic to the left of the junction (left-arc) and subsonic to the right of the junction. As can be seen in  Figure \ref{fig:test4},  the two methods produce different results, particularly on the left arc. This is even more visible in Figure \ref{fig:test4J}, where we plot the evolution in time of the four junction values  $\rho^*_{\ell}$, $\rho^*_{r}$, $q^*_{\ell}$  and $q^*_{r}$. We observe, on the left arc the flow transition from the supersonic region to the subsonic one, and that the two solvers, HLL-$\Jmr$ and HLL-$\Jrs$, produce some slightly different results, especially before the transition. 
This is an example  where the datum on the left is supersonic with negative velocity, namely $q_\ell(0,L_\ell)<0$ and $q_\ell(0,L_\ell)<q_s^-(\rho_\ell(0,L_\ell))$. In this case the Riemann based algorithm selects the solution on the right invariant curve, see Subsection \ref{sec:RSsupersonic}, while the relaxation one $\Jmr$ gives a result which is not quite the same.
\begin{figure}[htbp!]
\centering
\begin{subfigure}[b]{0.45\textwidth}
\includegraphics[scale=0.5]{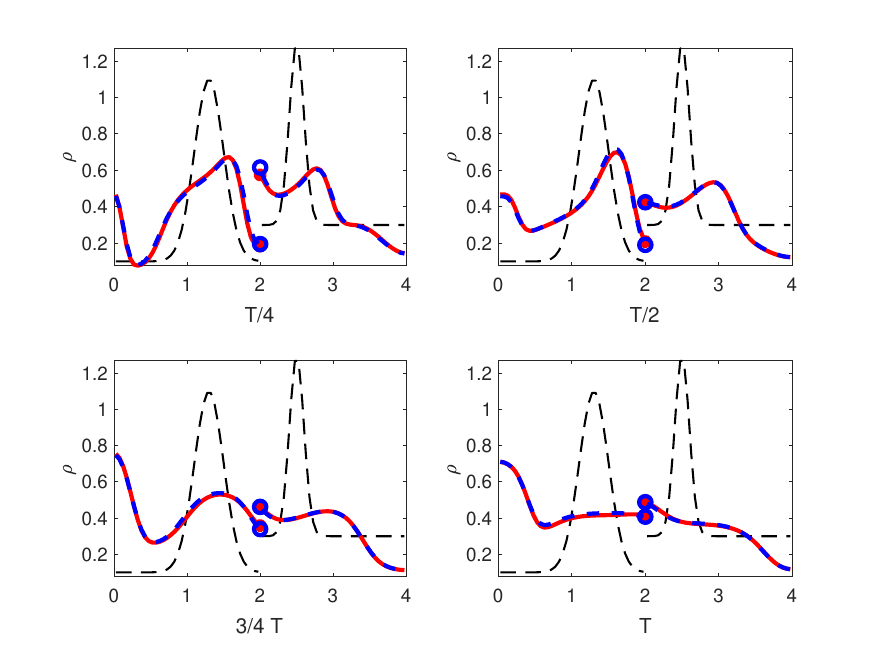}
\caption{Function $\rho$}
\end{subfigure}
\begin{subfigure}[b]{0.45\textwidth}
\includegraphics[scale=0.5]{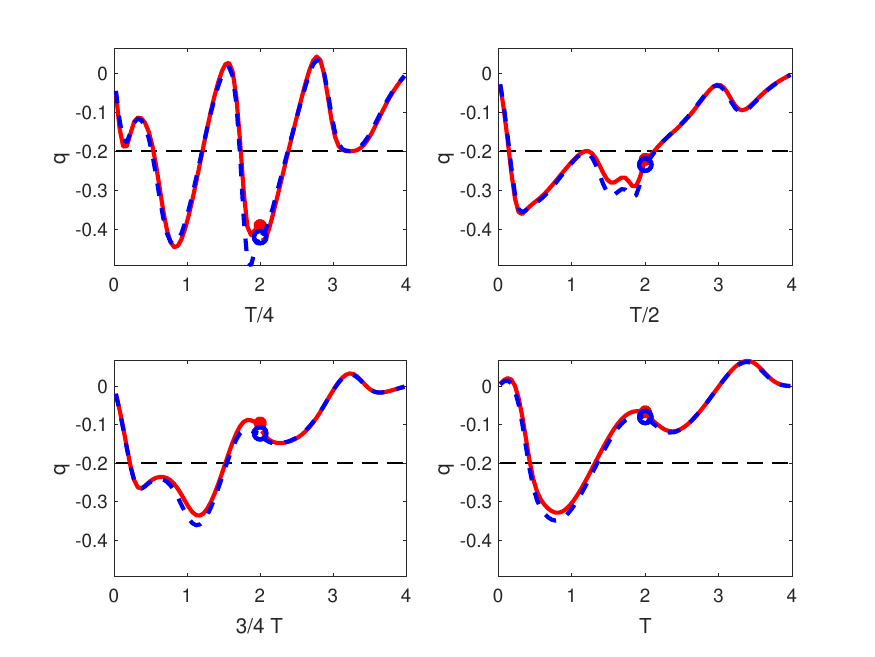}
\caption{Function $q$}
\end{subfigure}
\caption{\textbf{Test C4}. Comparison of the solutions given by scheme HLL-$\Jmr$ (in  solid red line) and by scheme HLL-$\Jrs$ (in dashed blue line), starting from a non-constant initial datum supersonic on the left-arc and subsonic on the left-arc (in dashed black line). The blue and red dots represent the values at the junction. On the left, function $\rho$ and on the right, function $q$  for different times $t=0.25, \, 0.5$ and $0.75$  with final time $T=1$.}
\label{fig:test4}
\end{figure}
\begin{figure}[htbp!]
\centering
\includegraphics[scale=0.5]{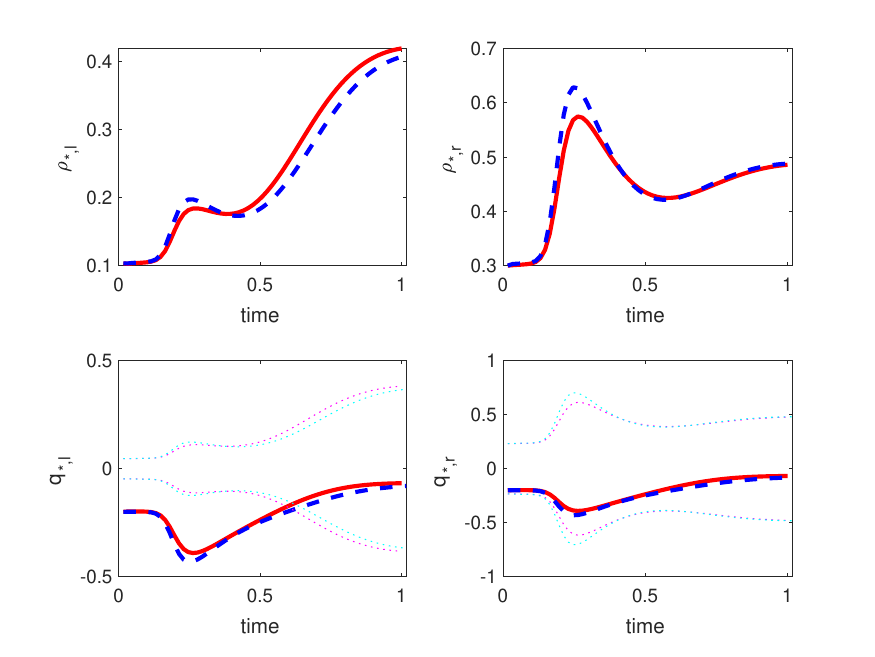}
\caption{\textbf{Test C4}. Evolution with respect to time of the 4  junction values:   $\rho_{*,\ell}$ on the top left subfigure, $\rho_{*,r}$ on the top right, $q_{*,\ell}$ on the bottom left and $q_{*,r}$ on the bottom right, computed by scheme HLL-$\Jmr$ (in  solid red line) and by scheme HLL-$\Jrs$ (in dashed blue line). The two dotted  lines in the  figures on the bottom delineate the boundary between the subsonic and supersonic regions (the sonic curve given for the corresponding value of $\rho$), in pink for the HLL-$\Jmr$ scheme and in light blue for the HLL-$\Jrs$  scheme. The regime is subsonic inside these two curves and supersonic  outside.}
\label{fig:test4J}
\end{figure}
\end{description}

\subsection{Study of different values for parameter $\kappa$}
 In this test case, we show the numerical dynamics obtained by scheme HLL-$\Jmr$ for two different values of parameter $\kappa$.
We assume a constant initial datum given by
\begin{equation*}\left\{\begin{array}{ll}
\rho_\ell(0,x)=2.5, & \mbox{ for all } x\in[0,L_\ell],
\\
\rho_r(0,x)=0.5, & \mbox{ for all }x\in[0,L_r],
\end{array}\right.
\end{equation*}
and 
\begin{equation*}
q_\ell(0,x)=q_r(0,x)=\kappa(\rho_\ell(0,x)-\rho_r(0,x))=2\kappa.
\end{equation*}
In Figure \ref{fig:test10_11}, we show the density evolution obtained with the two values $\kappa=100$ and $\kappa=0.1$. 

On the one hand, as expected, for $\kappa=100$ the density values at the junction start out different and then tend to converge until they coincide. This is coherent since for large values for $\kappa$, we expect the system to behave similarly  to the system with   the continuity of densities at the junction, that is to say the junction condition that is usually considered in other articles. On the other hand, for small value of $\kappa=0.1$, the density remains discontinuous at the interface. 

Note that the Riemann-based HLL-$\Jrs$ scheme does not provide a solution to test case with $\kappa=100$, whereas scheme HLL-$\Jmr$ does; it would be necessary to expand the region of admissible solutions and add further conditions to uniquely define the unknown values at the junction.
\begin{figure}[htbp!]
\centering
\includegraphics[scale=1]{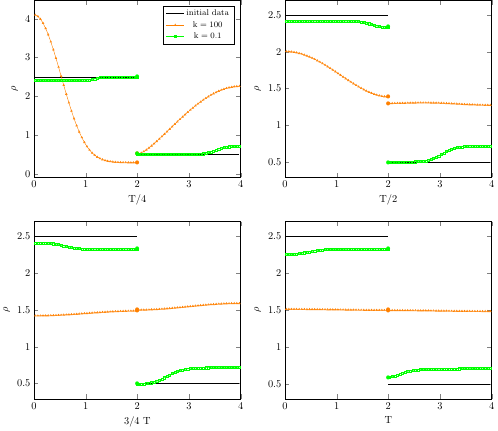}
\caption{\textbf{Different values for parameter $\kappa$}. Density solution obtained with the approximation HLL-$\Jmr$, starting from a constant initial datum on both arcs (in solid black line), for two different values of parameter $\kappa$ : $\kappa=100$ (dotted orange line) and $\kappa=0.1$ (dotted green line),  for different times $t=0.5, \, 1$ and $1.5$  with final time  $T=2$.}
\label{fig:test10_11}
\end{figure}

\subsection{Entropy dissipation at the discrete level} 
In this subsection, following the discussion in Subsec. \ref{sec:entropy}, we investigate the evolution of the total entropy in the network,
\begin{equation}\label{eq:numtotEntr}
S(t)=\int_0^{L_\ell+L_r} \eta(\rho(t,x),q(t,x)) dx \approx S^n= \Dx\sum_{j=0}^{J_\ell+J_r} \eta^n_j,
\end{equation}
where $\eta$ is the entropy function defined in \eqref{eq:entropy_pair} and $\eta^n_j=\eta(\rho^n_j,q^n_j)$ for $j$ running along the nodes of the two arcs forming the network.
We assume two different constant initial data 
which verify $\KK$~:
\begin{equation*}
q_\ell(0,x)=q_r(0,x)=\kappa(\rho_\ell(0,x)-\rho_r(0,x)),
\end{equation*}
with $\kappa=1$, and which differ in the amplitude of the density jump at the junction:
\begin{itemize}
\item[(E1)] $\rho_\ell(0,x)=1.5$ and $\rho_r(0,x)=1$ such that $q_\ell(0,x)=q_r(0,x)=\rho_\ell(0,x)-\rho_r(0,x)=0.5$.
\item[(E2)] $\rho_\ell(0,x)=6$ and $\rho_r(0,x)=1$ such that $q_\ell(0,x)=q_r(0,x)=\rho_\ell(0,x)-\rho_r(0,x)=5$.
\end{itemize}
We recall here  the entropy dissipation inequality \eqref{eq:entropy_inequ_estimate}~: $\displaystyle S(t) \leq S(0)-\int_{0}^t \Delta G(s) ds$.

As can be seen in Figure \ref{fig:test14}-(right) and Figure \ref{fig:test15}-(right), for both initial data the total entropy \eqref{eq:numtotEntr} is decreasing in time and numerically verify  a discrete version of the inequality \eqref{eq:entropy_inequ_estimate}, i.e.
\begin{equation}\label{eq:ineqEntr} 
S^n=\Dx\sum_{j=0}^{J_\ell+J_r} \eta^n_j \leq \Dx\sum_{j=0}^{J_\ell+J_r} \eta^0_j - \Dt\sum_{k=0}^n \left(G^{*,k}_{\ell}-G^{*,k}_{r}\right)=S^0- \Dt\sum_{k=0}^n \Delta G^{*,k},
\end{equation}
where $G^{*,n}_{\ell} = G(\rho^{*,n}_{\ell}, q^{*,n}_{\ell})$ and $G^{*,n}_{r} = G(\rho^{*,n}_{r}, q^{*,n}_{r})$ represent the values of the entropy flux, as defined in \eqref{eq:entropy_pair}, computed at the left and right junction states obtained from the solution of system $\Jmr$ at time step $n$. As in  Subsec. \ref{sec:entropy}, we also define the  quantity $\Delta G^{*,n}:= G^{*,n}_{\ell}-G^{*,n}_{r}$, which is a discrete version at time $t_{n}$ of quantity $\Delta G$.

Specifically, for the small jump assumed in test case (E1), the  quantity  $\Delta G$,  displayed in Figure \ref{fig:test14}-(a), continuous line, is  positive for all time $t>0$ as expected from the discussion in  Subsec. \ref{sec:entropy}, thus ensuring the dissipation of entropy with respect to time. We notice  in Figure \ref{fig:test14}-(b) that  the estimate  $\displaystyle  S(0)-\int_{0}^t \Delta G(s) ds$ is above the entropy for all times, as expected, and is equal up to a constant to the integral  displayed on Figure \ref{fig:test14}-(a). We also observe that in that case the estimate is sharp.
 
However,  for big density jumps at the junction, like in test case (E2), the variation $\Delta G(t)$ is  negative for all times, as displayed  in Figure \ref{fig:test15}-(a), continuous line. But, even if the right hand part of the inequality \eqref{eq:ineqEntr} is increasing, and the  bound from above $\displaystyle  S(0)-\int_{0}^t \Delta G(s) ds$  is therefore not accurate, the total entropy remains decreasing in time, see Figure \ref{fig:test15}-(b). We can interpret this as follows~: the dissipation of entropy along the arcs is large enough to compensate  the small enough non-dissipation of entropy at the junction.

\begin{figure}[htbp!]
\centering
\begin{subfigure}[b]{0.45\textwidth}
\includegraphics[scale=1.3]{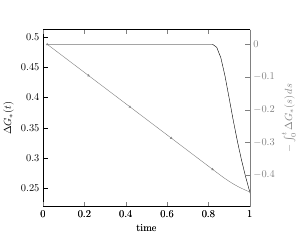}
\caption{}
\end{subfigure}
\begin{subfigure}[b]{0.45\textwidth}
\includegraphics[scale=1.3]{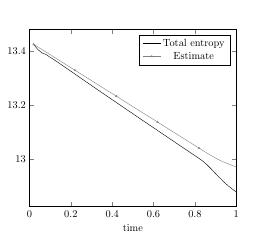}
\caption{}
\end{subfigure}
\caption{\textbf{Entropy dissipation}, test case (E1), small jump at the junction. (a) Plot of  $\Delta G$ with respect to time (solide black line and range of values on the left) and of $-\int_0^t \Delta G(s)\,ds$ (dotted grey line and range of values on the right). (b) Plot of the total entropy  $S$ (solid black line) as a function of time and of the right hand side $\displaystyle  S(0)-\int_{0}^t \Delta G(s) ds$ of inequality \eqref{eq:ineqEntr} (dotted grey line).}
\label{fig:test14}
\end{figure}
\begin{figure}[htbp!]
\centering
\begin{subfigure}[b]{0.45\textwidth}
\includegraphics[scale=1.3]{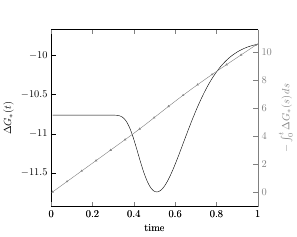}
\caption{}
\end{subfigure}
\begin{subfigure}[b]{0.45\textwidth}
\includegraphics[scale=1.3]{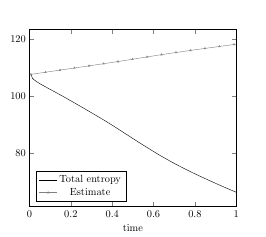}
\caption{}
\end{subfigure}
\caption{\textbf{Entropy dissipation}, test case (E2), large jump at the junction.  (a) Plot of  $\Delta G$ with respect to time (solide black line and range of values on the left) and of $-\int_0^t \Delta G(s)\,ds$ (dotted grey line and range of values on the right). (b) Plot of the total entropy  $S$ (solid black line) as a function of time and of the right hand side $\displaystyle  S(0)-\int_{0}^t \Delta G(s) ds$ of inequality \eqref{eq:ineqEntr} (dotted grey line).}
\label{fig:test15}
\end{figure}

\subsection{Simulations in the case $\gamma=4/3$}

To conclude this part of the numerical tests, let us consider the problem with a different value of $\gamma$ in the pressure law \eqref{eq:pfunct}. 
We take $\gamma=4/3$ and the constant initial datum (same as (E2)) given by
\begin{equation*}\left\{\begin{array}{lll}
\rho_\ell(0,x)=6,& q_\ell(0,x)=5, & \mbox{ for all } x\in[0,L_\ell],
\\
\rho_r(0,x)=1,& q_r(0,x)=5,  & \mbox{ for all }x\in[0,L_r].
\end{array}\right.
\end{equation*}
Note that here the simulations are performed using only the scheme HLL-$\Jmr$, since the scheme HLL-$\Jrs$ is not able to handle this case (there is no simple solution to the corresponding Riemann problem). 

 The results are shown in Figure \ref{fig:test12}. We notice first that the evolution is much slower than in the case $\gamma=2$ and that function $q$  is decreasing towards $0$ but very slowly. We also observe that before going asymptotically towards a constant state, the density $\rho$ tends to form a bump like solution on the right arc of the network, see  the function $\rho$ at different times on the left subfigure.

\begin{figure}[htbp!]
\centering
\begin{subfigure}[b]{0.45\textwidth}
\includegraphics[scale=0.5]{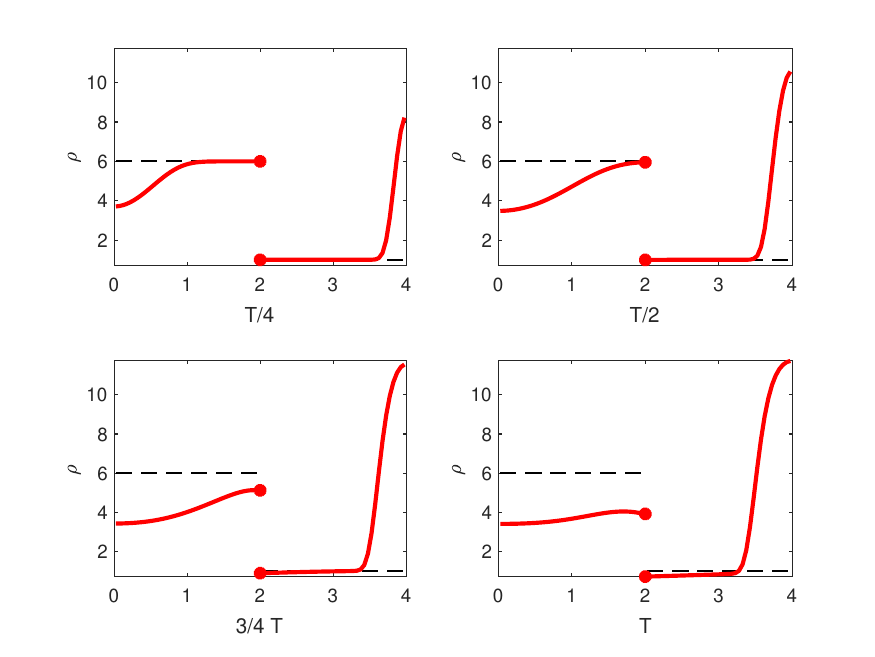}
\caption{Function $\rho$}
\end{subfigure}
\begin{subfigure}[b]{0.45\textwidth}
\includegraphics[scale=0.46]{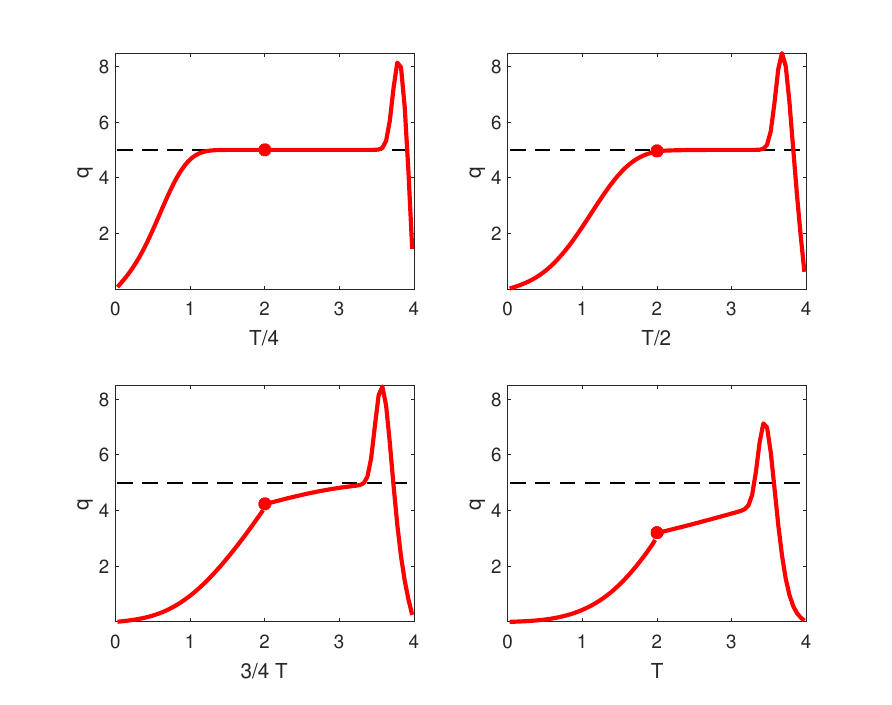}
\caption{Function $q$}
\end{subfigure}
\caption{\textbf{Simulation for $\gamma=4/3$}.   Solutions given by scheme HLL-$\Jmr$ (in  solid red line), starting from a constant  initial datum (in dashed black line) on both arcs. The red dots represent the values at the junction. On the left, function $\rho$ and on the right, function $q$  for different times $t=0.25, \, 0.5$ and $0.75$  with final time $T=1$.}
\label{fig:test12}
\end{figure}


\section{Extensions to a general network 
}\label{Extensions}

To begin with, we consider an extension of what has been previously done to more complicated network, i.e. networks with any number of arcs where nodes can be connected to any number of arcs.


Let us define a network or a connected  graph $G=(\mathcal{N},\mathcal{A})$, as  composed of two finite sets,  a set of 
 nodes (or  vertices) $\mathcal{N}$  and  a set of $N$ 
 arcs  $\mathcal{A}$, such that an arc connects a pair of nodes. Since arcs are bidirectional the graph is non-oriented, but we need to fix an artificial orientation in order to fix a sign to the velocities. 
 The network is therefore  composed of "oriented" arcs and there are  two different types of intervals at a node  $p \in \mathcal{N}$ : incoming ones -- the set of these intervals is denoted by $\mathcal{I}_{p}$ -- and outgoing ones  -- whose set  is denoted by $\mathcal{O}_{p}$. See,  for example,  Figure \ref{network_4arcs} for the representation of a $4$-arc and $1$node -network, with $2$ incoming arcs and $2$ outgoing arcs, where $1,2 \in \mathcal{I}_{1}$ and $3,4 \in \mathcal{O}_{1}$.

\subsection{Extension of junction conditions \eqref{system_junction} }

Let us consider now a node $p \in \mathcal{N}$, which is the junction of $N_p$ arcs, denoted by $i$, with $1 \leq i \leq N_p$.
We denote by $\mathcal{I}_{p}$ the set of incoming arcs and by $\mathcal{O}_{p}$ the set of outgoing arcs.

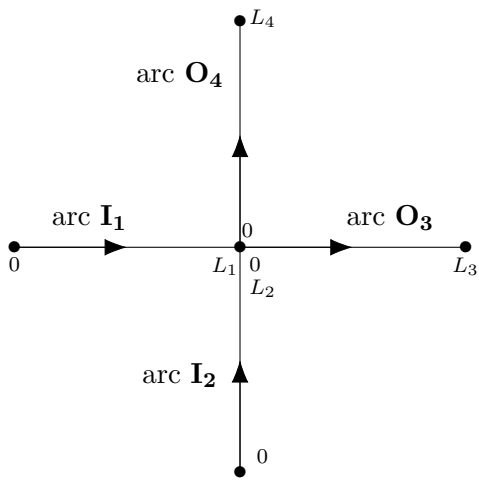
\begin{figure}[htbp!]
\begin{center}
\begin{tikzpicture}
\draw  (-3,0) -- (3,0) ;
\draw  (0,-3) -- (0,3) ;
\draw[-{Latex[length=3mm]}] (-3,0) -- (-1.5,0) ;
\draw[-{Latex[length=3mm]}] (0,0) -- (1.5,0) ;
\draw[-{Latex[length=3mm]}] (0,-3) -- (0,-1.5) ;
\draw[-{Latex[length=3mm]}] (0,0) -- (0,1.5) ;
\draw (0,0) node {$\bullet$} ;
\draw (-3,0) node {$\bullet$} ;
\draw (3,0) node {$\bullet$} ;
\draw (0,-3) node {$\bullet$} ;
\draw (0,3) node {$\bullet$} ;
\draw (-2,0.1) node[above]{arc $\mathbf{I_{1}}$} ;
\draw (2,0.1) node[above]{arc $\mathbf{O_{3}}$} ;
\draw (-0.8,-2) node[above]{arc $\mathbf{I_{2}}$} ;
\draw (-0.8,2) node[above]{arc $\mathbf{O_{4}}$} ;
\draw (-3,0) node[below]{\scriptsize{$0$}} ;
\draw (-.2,0) node[below]{\scriptsize{$L_{1}$}} ;
\draw (.2,0) node[below]{\scriptsize{$0$}} ;
\draw (3,0) node[below]{\scriptsize{$L_{3}$}} ;
\draw (0.1,0) node[above]{\scriptsize{$0$}} ;
\draw (0.3,-0.3) node[below]{\scriptsize{$L_{2}$}} ;
\draw (0.3,-3) node[above]{\scriptsize{$0$}} ;
\draw (0.3,2.8) node[above]{\scriptsize{$L_{4}$}} ;
\end{tikzpicture}
\end{center}
\caption{An example of a 4 arc network with $2$ incoming arcs and $2$ outgoing arcs} 
\label{network_4arcs}
\end{figure}

We will first generalize at the continuous level conditions \eqref{bc} for the outer nodes and then conditions \eqref{jc}   for junction nodes. Then, we give a numerical discretization of these conditions, namely a generalization of discrete conditions \eqref{bord-gauche2}- \eqref{bord-droit2} on outer nodes and conditions \eqref{system_junction} at  junction nodes. 
Note that  the "discrete equality junction conditions" \eqref{system_junction} are straightforward to generalize unlike Riemann invariant junction conditions \eqref{eq:1in1_Rpb}, which would be much  more involved to consider for general networks. 

Therefore, the generalization of conditions \eqref{bc} at the outer nodes reads as~:
 \begin{equation}\label{bcN}
 \left\{\begin{array}{ll}
\partial_{x} \rho_{i} (\cdot, L_{i})=0 \text{ and } q_{i}(\cdot, L_{i}) =0,
& \text{ if } i \in \mathcal{O}_{p},\\
\partial_{x} \rho_{i} (\cdot, 0)=0 \text{ and } q_{i} (\cdot, 0) =0.
 & \text{ if } i \in \mathcal{I}_{p},
 \end{array}\right.
 \end{equation}

We now consider the following  generalization of conditions \eqref{jc} at the junction node $p \in \mathcal{N}$, which says that the flux is equal to a linear combination of the different differences of densities on the various arcs~:
 \begin{equation}\label{jcN}
\left\{\begin{array}{ll}
  -q_{i} (\cdot,0)=\sumd_{j\in \mathcal{I}_{p}}  \kappa_{i,j} (\rho_{i} (\cdot,0)- \rho_{j} (\cdot,L_{j})) +\sumd_{j\in \mathcal{O}_{p}}  \kappa_{i,j} (\rho_{i} (\cdot,0)- \rho_{j}(\cdot,0))& \text{ if } i \in \mathcal{O}_{p},  \\
    q_{i} (\cdot,L_{i}) =\sumd_{j\in \mathcal{I}_{p}}  \kappa_{i,j} (\rho_{i} (\cdot,L_{i})- \rho_{j} (\cdot,L_{j})) +\sumd_{j\in \mathcal{O}_{p}}  \kappa_{i,j} (\rho_{i} (\cdot,L_{i})- \rho_{j}(\cdot,0))  & \text{ if } i \in \mathcal{I}_{p},
\end{array}\right.
\end{equation}
where $\kappa_{i,j} \geq 0$ are generalizations of permeability coefficients between two arcs. 

In order to have the mass conservation property at the junction $p \in \mathcal{N}$, we also impose a generalization of condition \eqref{jcegal}, which reads as~:
 \begin{equation}\label{consflux}
  \sum_{i \in \mathcal{O}_{p}} -q_{i} (\cdot,0)+  \sum_{i \in \mathcal{I}_{p}}  q_{i} (\cdot,L_{i}) =0.
  \end{equation}
If we impose that the transmission coefficients  are symmetric, namely $$\kappa_{i,j}=\kappa_{j,i},$$
condition \eqref{consflux} is automatically satisfied. 
This condition ensures that the total mass of the whole system is conserved, namely
 \begin{equation}\label{consmass}
 \frac{d}{dt}m(t)=0, \text{ with } m(t)=\sum_{i \in  \mathcal{A}} \int_{0}^{L_{i}} \rho_{i}(t, x) \, dx.
  \end{equation}

Now, let us consider the discretization of the previous conditions on the outer nodes and at the junction, following the ideas of Sec.\ref{DEJunctionCondition}.

For the outer nodes, we set, in the spirit of conditions \eqref{bord-gauche2}- \eqref{bord-droit2}~:
 \begin{equation}\label{bcNdiscr}
 \left\{\begin{array}{ll}
 \rho_{0, i}=\rho_{1,i} \text{ and } q_{0, i}=-q_{1, i}& \text{ if } i \in \mathcal{O}_{p},\\
 \rho_{J_{i},i}=\rho_{J_{i}-1, i} \text{ and } q_{J_{i}, i}=-q_{J_{i}-1, i}& \text{ if } i \in \mathcal{I}_{p}.
 \end{array}\right.
 \end{equation}
Now, at each node $p \in \mathcal{N}$, we will need to solve a linear system. Namely, at node $p \in \mathcal{N}$, considering  the junction conditions with $N_p$ arcs, we need to compute $2N_p$ values, that we denote $\rho^*_{i}$ and $q^*_{i}$ for $1 \leq i \leq N_p$. Equations \eqref{jcN} give at the discrete level the following $N_p$ equations~:
 \begin{equation}\label{jcNdiscr}
\left\{\begin{array}{ll}
  -q_{i}^* =\ds \sum_{j\in \mathcal{I}_{p}}  \kappa_{i,j} (\rho_{i}^*  - \rho_{j}^*  ) +\sum_{j\in \mathcal{O}_{p}}  \kappa_{i,j} (\rho_{i}^* - \rho_{j}^* )& \text{ if } i \in \mathcal{O}_{p},  \\
    q_{i}^*  =\ds\sum_{j\in \mathcal{I}_{p}}  \kappa_{i,j} (\rho_{i}^*  - \rho_{j}^* ) +\sum_{j\in \mathcal{O}_{p}}  \kappa_{i,j} (\rho_{i}^*  - \rho_{j}^* )  & \text{ if } i \in \mathcal{I}_{p}. 
\end{array}\right.
\end{equation}
We obtain $N_{p}$ other equations by imposing a generalization of equation \eqref{fluxJunc}~:
 \begin{equation}\label{fluxdiscr}
 \left\{\begin{array}{ll}
 \mathcal{F}^{\rho}(U_{1,i},U^*_{i})=q^*_{i}, & \text{ if } i \in \mathcal{O}_{p},  \\
 \mathcal{F}^{\rho}(U_{J_{i},i},U^*_{i})=q^*_{i} & \text{ if } i \in \mathcal{I}_{p}. 
 \end{array}\right.
\end{equation}
In the case of the HLL  flux, the system \eqref{fluxdiscr} becomes
 \begin{equation*}
 \left\{\begin{array}{ll}
\ds \frac{q^*_{i}+q_{1,i}}{2}+\frac{\lambda_{i}}{2}(\rho^*_{i}-\rho_{1,i})=q^*_{i}, & \text{ if } i \in \mathcal{O}_{p},  \\
\ds  \frac{q^*_{i}+q_{J_{i},i}}{2}-\frac{\lambda_{i}}{2}(\rho ^*_{i}-\rho_{J_{i},i})=q^*_{i} & \text{ if } i \in \mathcal{I}_{p}. 
 \end{array}\right.
\end{equation*}
or equivalently
 \begin{equation}\label{fluxdiscr3}
 \left\{\begin{array}{ll}
q^*_{i}=q_{1,i}+\lambda_{i}(\rho^*_{i}-\rho_{1,i}), & \text{ if } i \in \mathcal{O}_{p},  \\
q^*_{i}=q_{J_{i},i}-\mu_{i}(\rho ^*_{i}-\rho_{J_{i},i}), & \text{ if } i \in \mathcal{I}_{p}. 
 \end{array}\right.
\end{equation}

Combining Eq.\eqref{fluxdiscr3} and \eqref{jcNdiscr}, we obtain the following system composed of $N_{p}$ equations~:
 \begin{equation}\label{linear_system}
 \left\{\begin{array}{ll}
 \ds \sum_{j\in \mathcal{I}_{p}}  \kappa_{i,j} (\rho_{i}^*  - \rho_{j}^*  ) +\sum_{j\in \mathcal{O}_{p}}  \kappa_{i,j} (\rho_{i}^* - \rho_{j}^* )+\mu_{i}  \rho^*_{i} = -q_{1,i}+\mu_{i} \rho_{1,i}& \text{ if } i \in \mathcal{O}_{p},  \\
 \ds\sum_{j\in \mathcal{I}_{p}}  \kappa_{i,j} (\rho_{i}^*  - \rho_{j}^* ) +\sum_{j\in \mathcal{O}_{p}}  \kappa_{i,j} (\rho_{i}^*  - \rho_{j}^* ) +\mu_{i}  \rho^*_{i}=q_{J_{i},i}+\mu_{i}\rho_{J_{i},i} & \text{ if } i \in \mathcal{I}_{p}. 
\end{array}\right.
\end{equation}

Therefore, at node $p \in \mathcal{N}$,  the vector $\ds (\rho^*_{1}, \cdots, \rho^*_{N_{p}})$ is solution to a linear system, which is invertible, as soon as $\kappa_{i,j} \geq0$ and $\mu_{i}>0$, since 
the matrix 
\begin{equation}\label{matrix}
\left(
\begin{array}{cccc}
\ds\sum_{j \neq 1} \kappa_{1,j}+\mu_{1} &  &   &\\ 
 &  \ddots&   &- \kappa_{i,j} \\
 &  &   \ddots & \\
 - \kappa_{i,j} &  &   &  \ds \sum_{j \neq N_{p}} \kappa_{N_{p},j}+\mu_{N_{p}}
\end{array}
\right),
\end{equation}
 is a strictly diagonally dominant matrix. 
Values for momenta $q^*_{i}, \, 1 \leq i \leq N_{p}$ are then given by Eq.\eqref{fluxdiscr3}.

\subsection{Properties of the numerical scheme}

It is now easy to generalize the mass conservation and the positivity of solution, already demonstrated at Sec.\ref{properties} in the case of  a $2$ arc network, to the case of a $N$ arc network. However, the entropy dissipation property is far more complicated and will not be treated here. 
 First,  we prove the following property, which is equivalent to the continuous mass preservation \eqref{consmass}.

\begin{prop}
Let us  consider system \eqref{GasDynamicLeft} set on a $N$-arc network with boundary conditions \eqref{bcN} on the outer nodes and junction conditions \eqref{jcN} with symmetric coefficients, that is to say $\kappa_{i,j}=\kappa_{j,i}$.  The numerical scheme \eqref{schemenum} with  discrete boundary conditions \eqref{bcNdiscr} on the outer nodes and discrete junction conditions \eqref{jcNdiscr} is mass-preserving, that is to say 
the discrete mass $\ds m^n_{tot}={h}\sum_{i=1}^N  \sum_{j=1}^{J_{i}} \rho_{j,i}^n $  is independent of $n \in \N$.
\end{prop}
\begin{proof}
We consider the first component of equation \eqref{schemenum}. The conservation of mass is given by Eq.\eqref{bcNdiscr} and \eqref{jcNdiscr} with $\kappa_{i,j}=\kappa_{j,i}$, which leads to 
$$ \sum_{i \in \mathcal{O}_p} -q_{i}^*+  \sum_{i \in \mathcal{I}_p}   q_{i}^* =0.$$
\end{proof}

Now, we prove   the generalization of the positivity of the solution  in the particular case of HLL numerical flux~:
\begin{prop}
Let us consider scheme \eqref{schemenum} with flux \eqref{flux2} on a $N$ arc network, boundary conditions  \eqref{bcNdiscr} and conditions \eqref{jcNdiscr}-\eqref{fluxdiscr} at the junction in order to discretize system \eqref{GasDynamicLeft}-\eqref{bcN}-\eqref{jcN} on a general network.
We define $\lambda^n_i>0$ such that 
\[
\lambda^n_i \geq\maxd_{1 \leq j \leq J_{i}}{|u^n_{j, i}|+ \sqrt{p'(\rho^n_{j,i})}}, \text{ for } 1 \leq i \leq N.
\]
 If the initial condition is positive, that is to say 
 \[
  \rho^0_{j,i}\geq 0 \text{ for all } 1 \leq i \leq N, \text{ for all } 1  \leq j \leq J_{i}, \,  
  \]
   then the solution remains positive in time, that is to say 
    \[
  \rho^n_{j,i}\geq 0 \text{ for all } 1 \leq i \leq N, \text{ for all } 1  \leq j \leq J_{i}, \,   \text{ for all }  n \in \N.
  \]
 \end{prop}
\begin{proof}
We prove the proposition by induction, following the proof of Prop.\ref{positivity}.
Assume that $\ds \rho^n_{j,i} \geq 0$ for all $1 \leq i \leq N, \, 1 \leq j \leq J_{i}$. 

We consider node $p\in\mathcal{N}$.
Since $\lambda^n_{i} \geq |u_{j, i}^n|$ for all $1 \leq j \leq J_{i}$, we can show that the components of the second member of linear system \eqref{linear_system} are nonnegative. Moreover, since  $\kappa_{i,j}  \geq0$ and $\mu_{i}^n>0$,  it can be proved easily that the matrix \eqref{matrix} of linear system \eqref{linear_system}  is  a monotone matrix. Consequently the solution  $\ds (\rho^*_{1}, \cdots, \rho^*_{N_{p}})$ has positive components and therefore, we can deduce that 
  $\rho_{j, i}^{n+1} \geq 0$ for all $1 \leq i \leq N, \, 1 \leq j \leq J_{i}$.
  \end{proof}

\subsection{Numerical tests}

We  present now some numerical tests for two different general networks.

\begin{description}
\item[Network 1]
The first test is performed on a 12-arc  network with 12 nodes, see \cite{BNR14}~: 4 junction nodes and 8 outer nodes, see Figure\ref{Narc_network2}.  The  $4$ junction  nodes are forming a square and each of them is connected to 2 outer nodes.  Each junction node connects therefore 4 arcs together and the permeability coefficients  $\kappa_{i,j}$ are given in the following matrix, the same for each junction node, note that $\kappa_{i,j}=\kappa_{j,i}$~:
$$
K=\left( 
\begin{matrix}
0 &0.3 &0.2 &0.5 \\
            0.3 &0 &0.2&0.1\\
            0.2&0.2&0&0.2\\
            0.5&0.1&0.2&0
            \end{matrix}
   \right).         
$$
Here, all the arcs have the same length.

We start initially with density $\rho$ and momentum $q$ null on each arc except the first arc, where $\rho$ is a small perturbation of a constant state equal to $100$ and the velocity is constant equal to $10$. We let the system evolves until $T=50$. 

At time $T=10$, we see the propagation of the density along the network : the system is nearly at equilibrium but we notice some larger values for the density on the arcs which are connected to the first arc, especially the "neighboring" one which is also connected to an outer node. The value for $q$ has also strongly decreased in the first arc and the arcs where the modulus of the momentum is the highest are the ones linked to node $11$.

We notice that at time $T=50$, the system has reached an asymptotic state with some null momentum $q$ on each arc of the network and a constant density $\rho$ which is the same on each arc of the network; the constant density can be computed thanks to mass conservation and the initial mass on the whole network, see Figure\ref{Narc_test2}.

\begin{figure}[htbp!]
\centering
\includegraphics[scale=0.4]{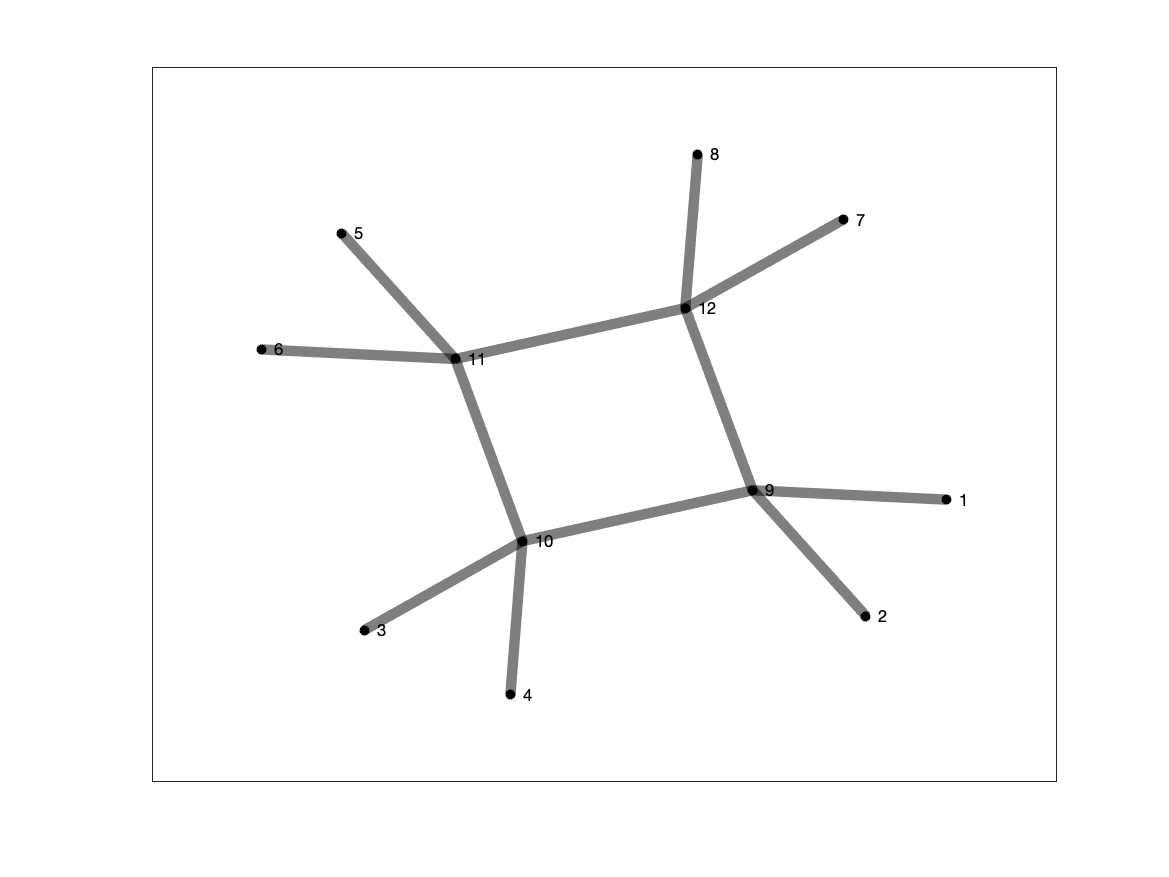}
\caption{Network 1 configuration  with 12 arcs and 12 nodes}
\label{Narc_network2}
\end{figure}

\begin{figure}[htbp!]
\centering
\begin{subfigure}[b]{0.45\textwidth}
\includegraphics[scale=0.4]{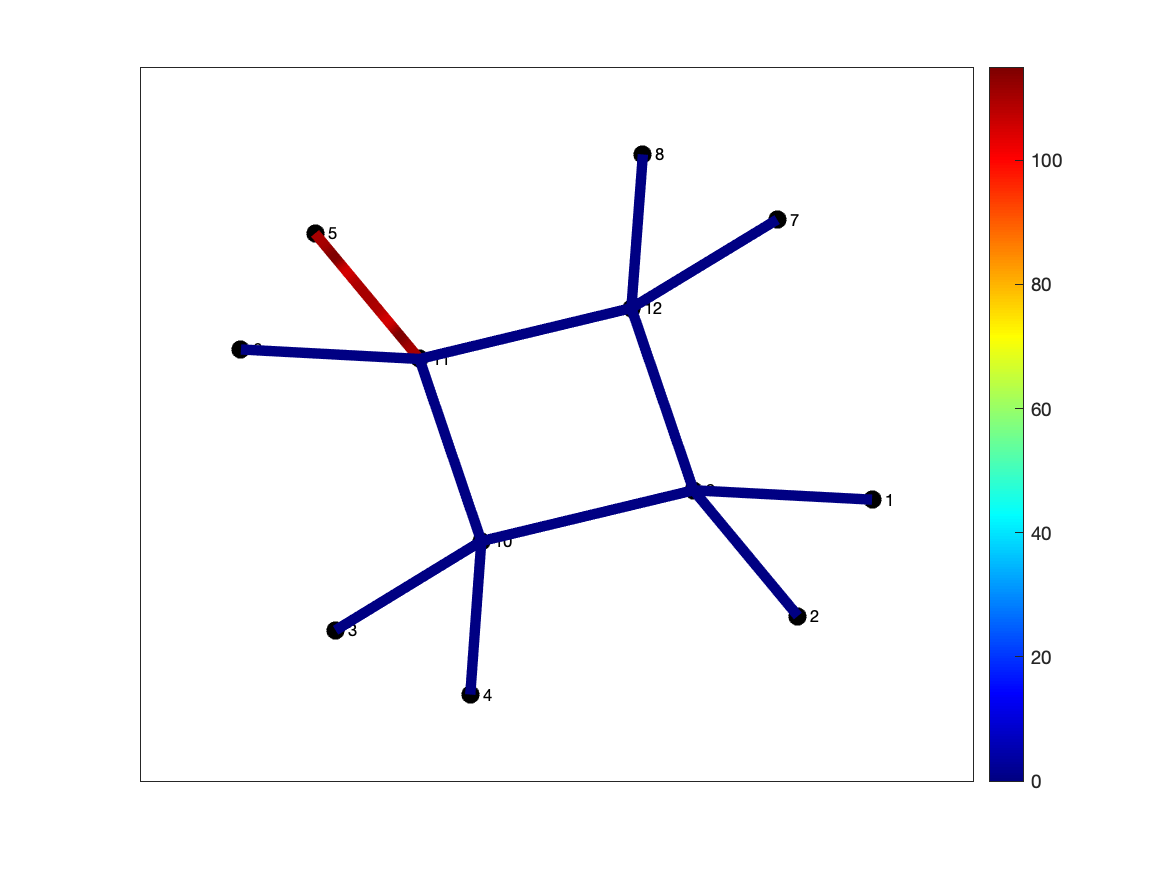}
\caption{initial condition for $\rho$}
\end{subfigure}
\begin{subfigure}[b]{0.45\textwidth}
\includegraphics[scale=0.4]{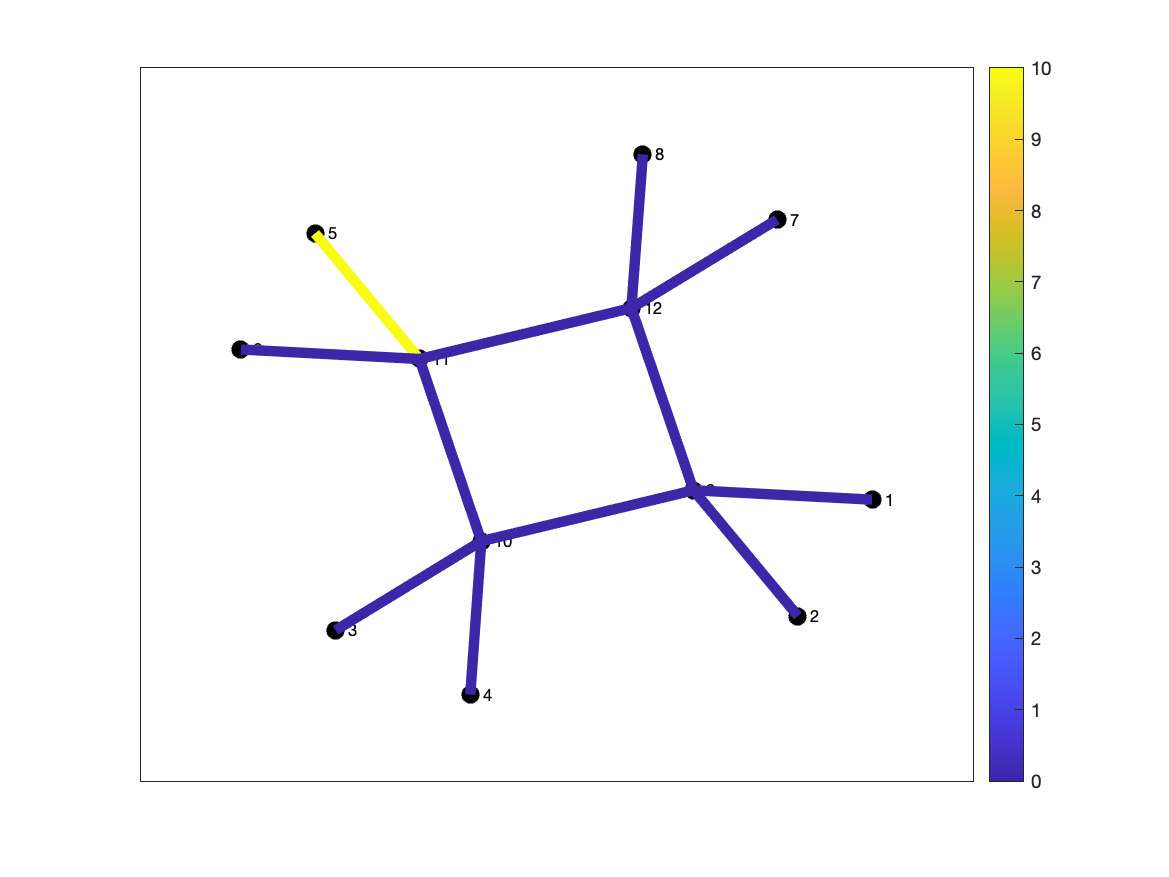}
\caption{initial condition for $q$}
\end{subfigure}
\begin{subfigure}[b]{0.45\textwidth}
\includegraphics[scale=0.4]{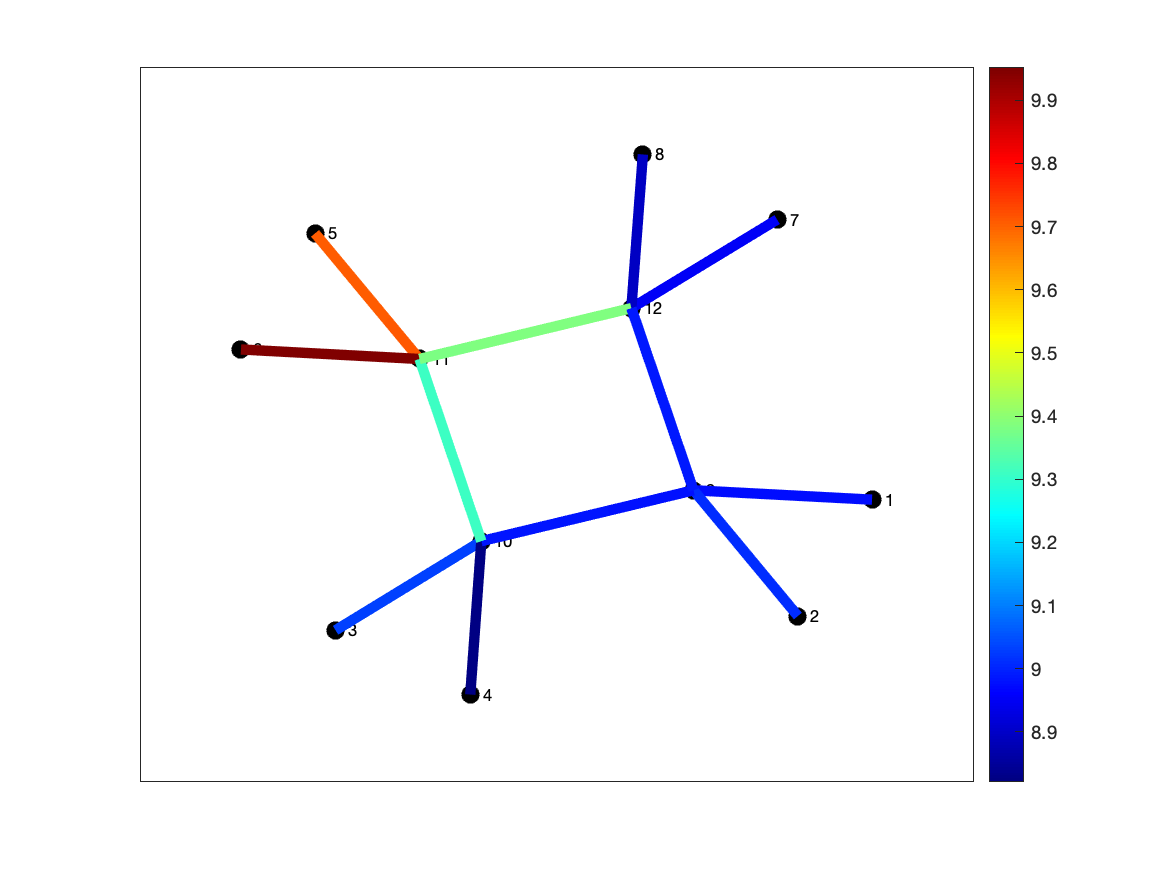}
\caption{Function $\rho$ at time $T=10$}
\end{subfigure}
\begin{subfigure}[b]{0.45\textwidth}
\includegraphics[scale=0.4]{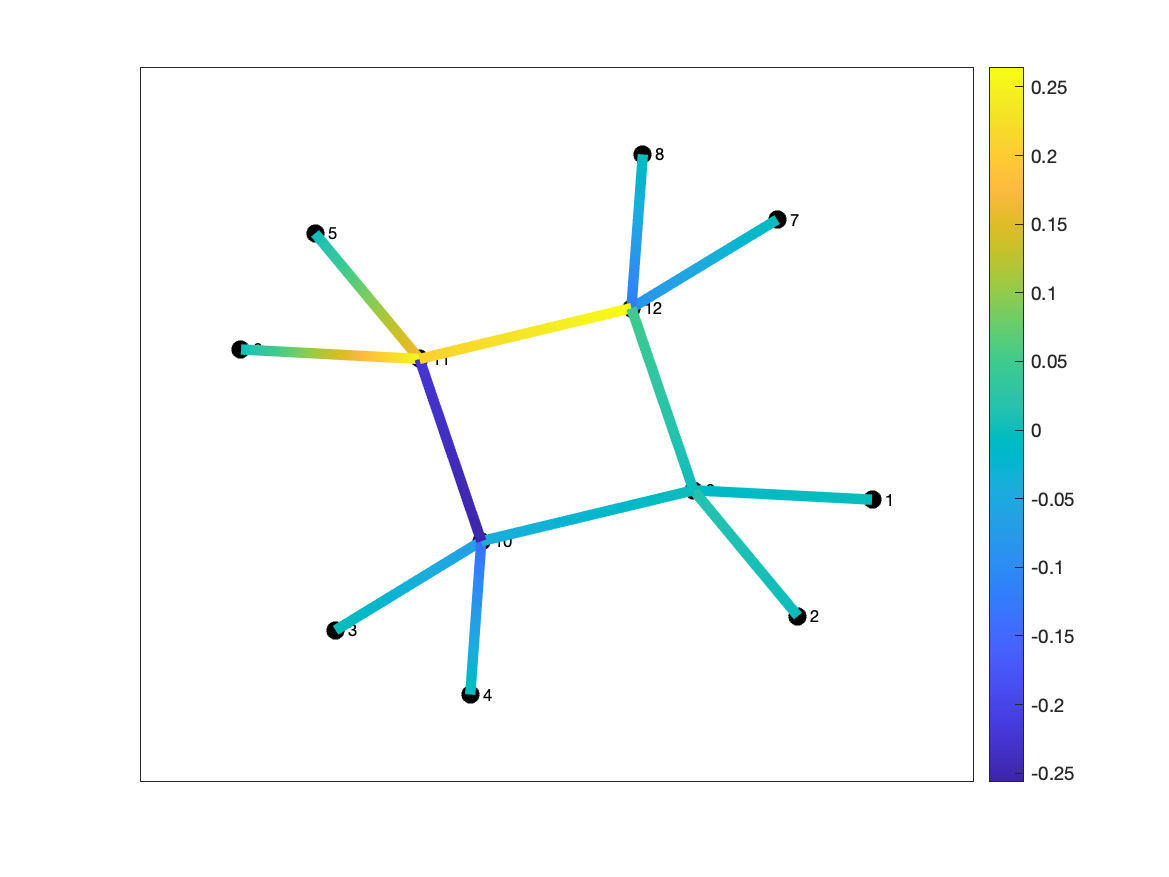}
\caption{Function $q$ at time $T=10$}
\end{subfigure}
\begin{subfigure}[b]{0.45\textwidth}
\includegraphics[scale=0.4]{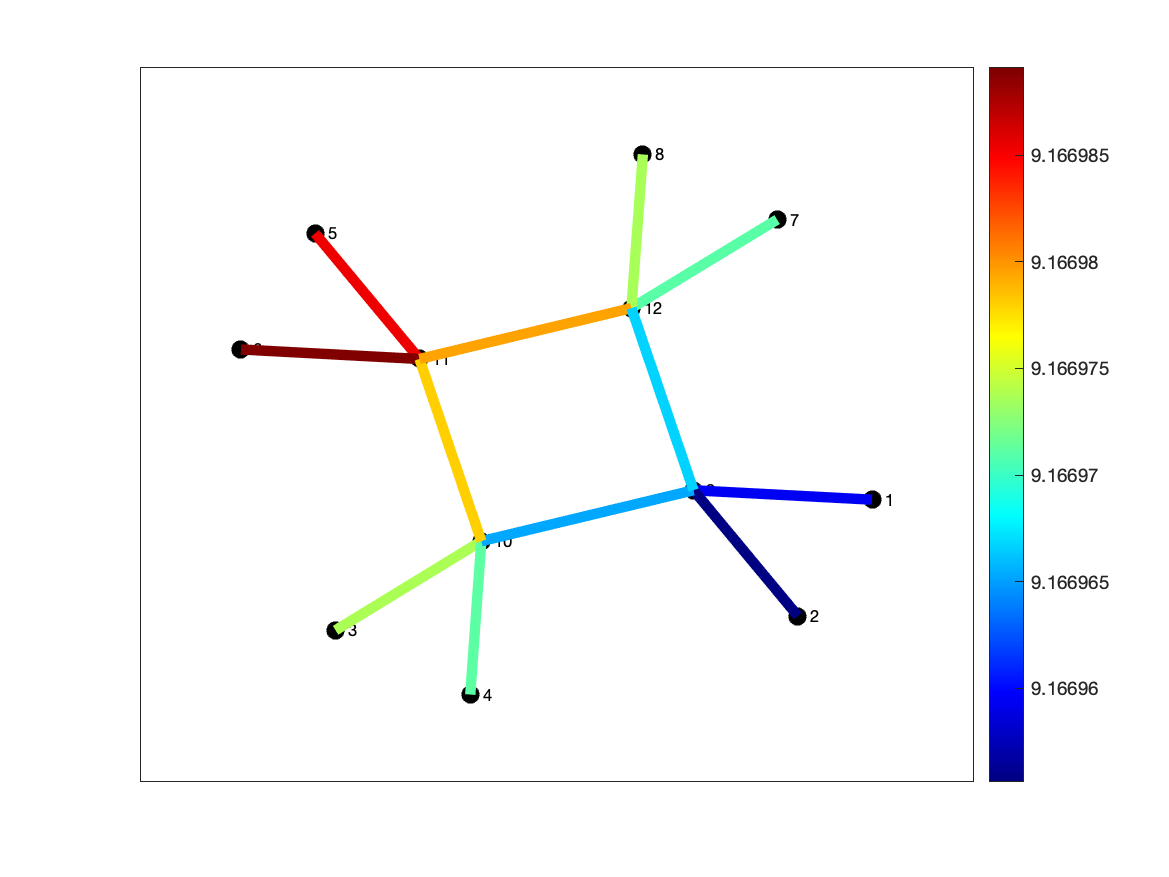}
\caption{Function $\rho$ at time $T=50$}
\end{subfigure}
\begin{subfigure}[b]{0.45\textwidth}
\includegraphics[scale=0.4]{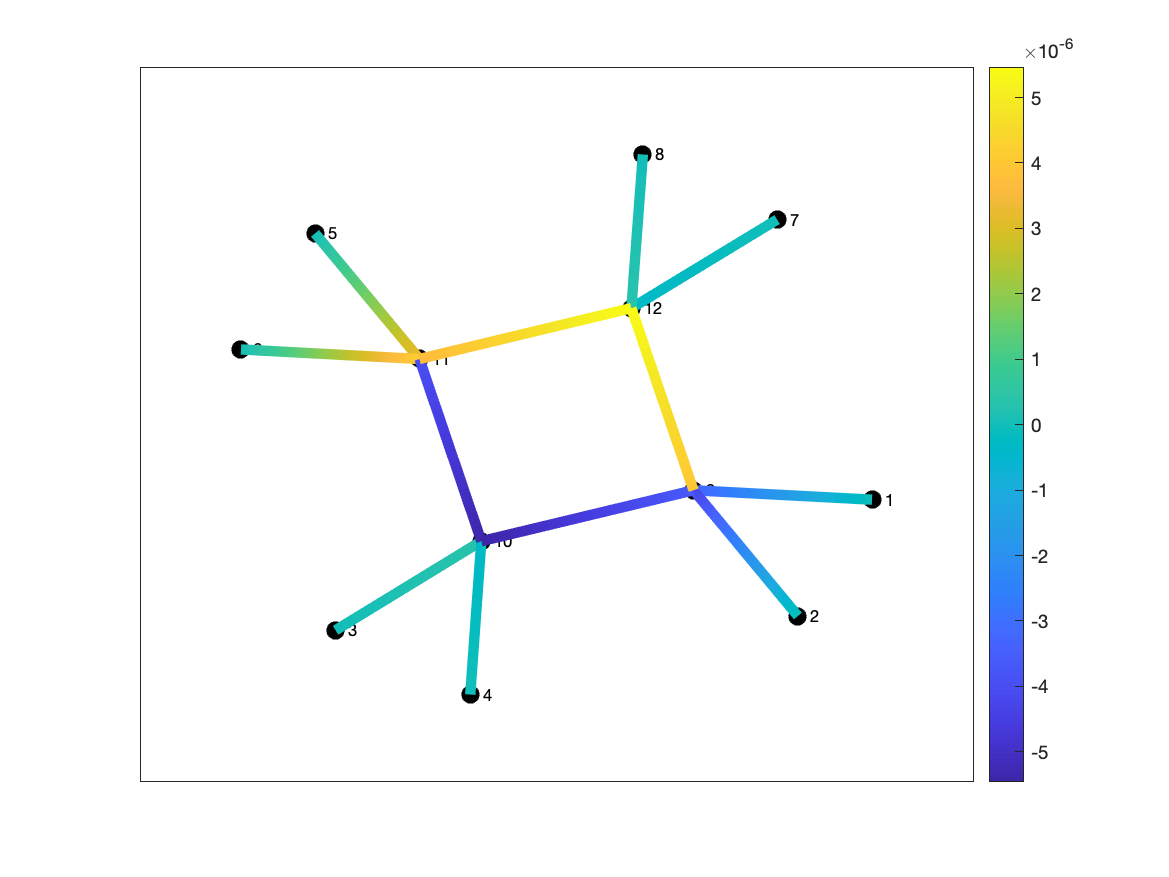}
\caption{Function $q$ at time $T=50$}
\end{subfigure}
\caption{Network 1: on top, initial condition; in the middle, solution at time $T=10$; on bottom, solution at time $T=50$. On the left, function $\rho$; on the right, function $q$.}
\label{Narc_test2}
\end{figure}

\item[Network 2]
We consider now a grid-shaped network with 26 arcs and 18 nodes, see \cite{BN18}~: 16 junction nodes and 2 outer nodes, see Figure\ref{Narc_network4}. The junction nodes are connecting to 2, 3 or 4 arcs, namely nodes inside the grid are connected to 4 nodes, nodes on the arcs of the grid to 3 nodes and nodes on the corners of the grids to 2 or 3 other nodes. Permeability coefficients  depend on the number of arcs at the node, more precisely if we denote by $N_{p}$ the number of arcs at node $p$, then   $\kappa_{i,j}=1/N_{p}$ for all $i \neq j$ and $0$ otherwise. The lengths of the arcs are also different and are equal to $L=10$ except $ L_{1}=L_{5}=L_{9}=L_{10}=L_{14}=L_{21}=L_{25}=L_{26}= 0.5$.

The initial density $\rho$ is a small perturbation on each arc of a constant equal to $0.45$ and the initial momentum $q$ is constant everywhere equal to $0.1$. 

At time $T=10$, we notice that positive values for $q$ are concentrated around the nodes of the network, whereas null or positive values of the momentum take place on the arcs. Regarding the density, we remark  the highest values on the arcs around one outer node of the network and the smallest values around the other outer node. 

At time $T=50$, the values for the density follow the same pattern as for $T=10$, but the values are more homogeneous all around the network. As for the momentum, it is null on each arc and $q$ have positive and negative values around the nodes of the networks, where the main changes occur. 

Finally at time $T=150$,  once again, the system converges   in large time towards a constant state for $\rho$ and $q$, but the convergence time is much longer than for the previous test cases, since the network is bigger
 see Figure\ref{Narc_test4}. The value of the density is slightly higher on the arcs with smaller lengths, see the red path on the figure.

\begin{figure}[htbp!]
\centering
\includegraphics[scale=0.4]{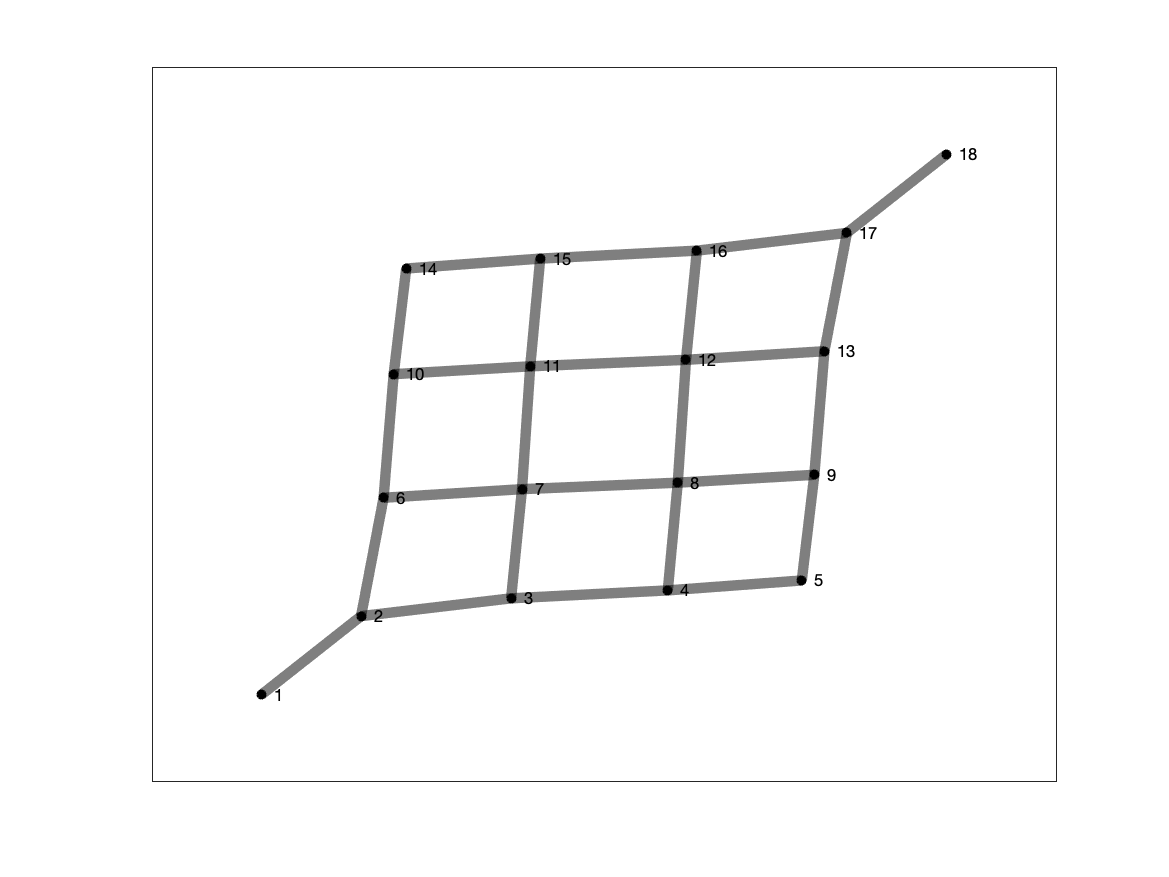}
\caption{Network 2 configuration with  26 arcs and 18 nodes; some of the arcs are smaller}
\label{Narc_network4}
\end{figure}

\begin{figure}[htbp!]
\centering
\begin{subfigure}[b]{0.45\textwidth}
\includegraphics[scale=0.35]{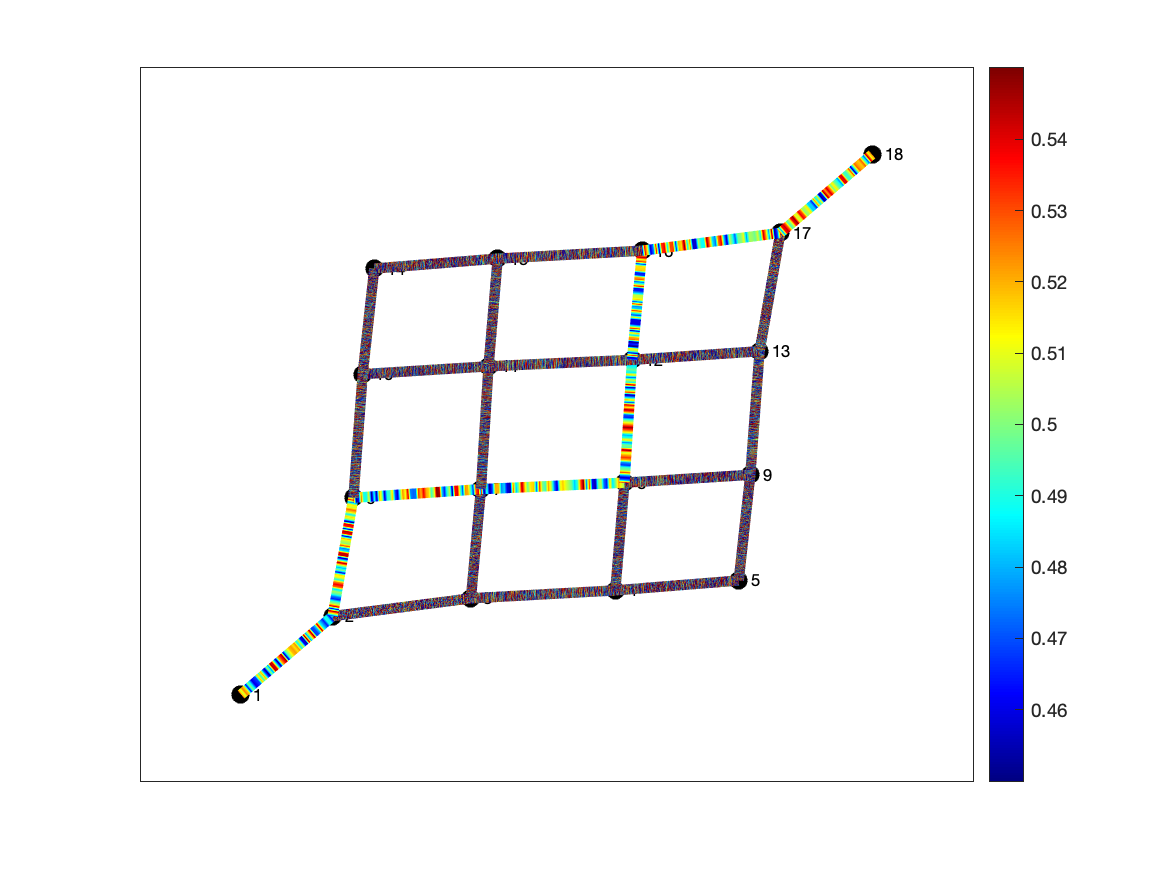}
\caption{initial condition for $\rho$}
\end{subfigure}
\begin{subfigure}[b]{0.45\textwidth}
\includegraphics[scale=0.35]{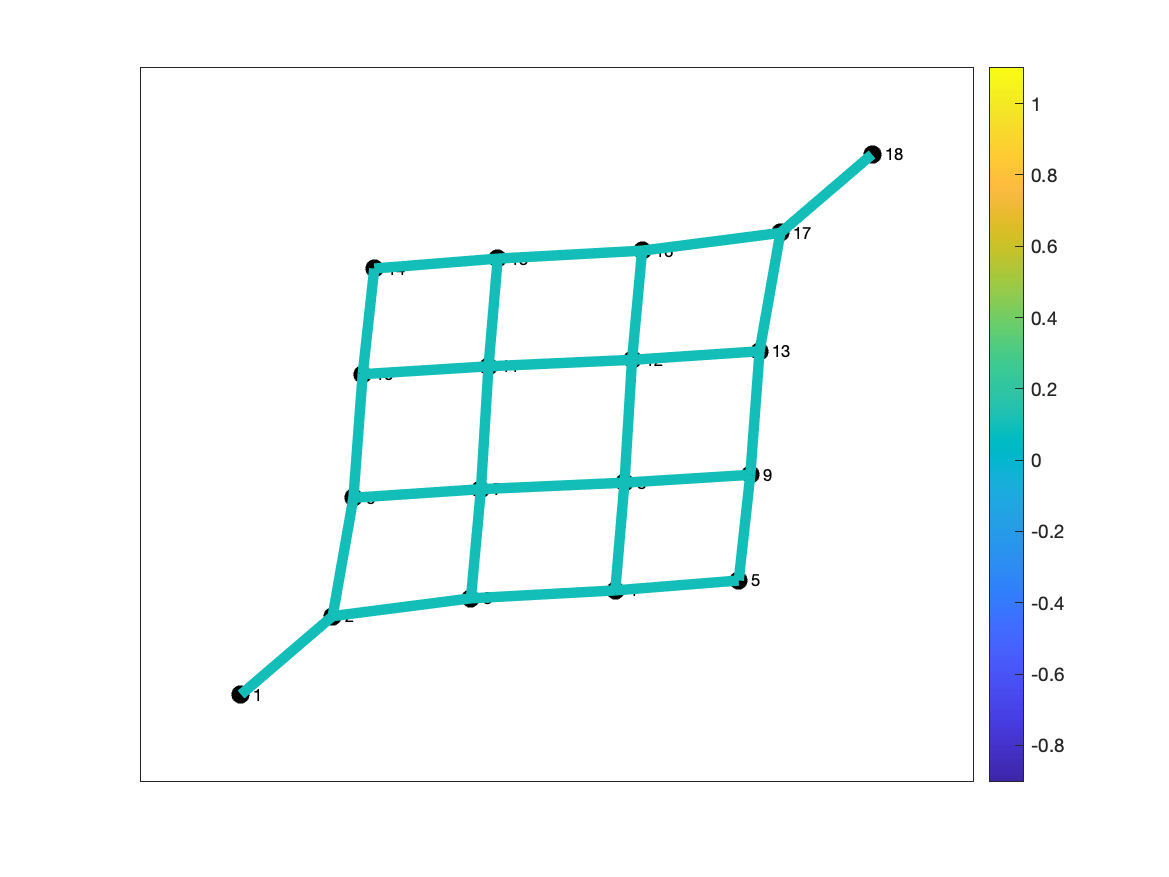}
\caption{initial condition for $q$}
\end{subfigure}
\begin{subfigure}[b]{0.45\textwidth}
\includegraphics[scale=0.35]{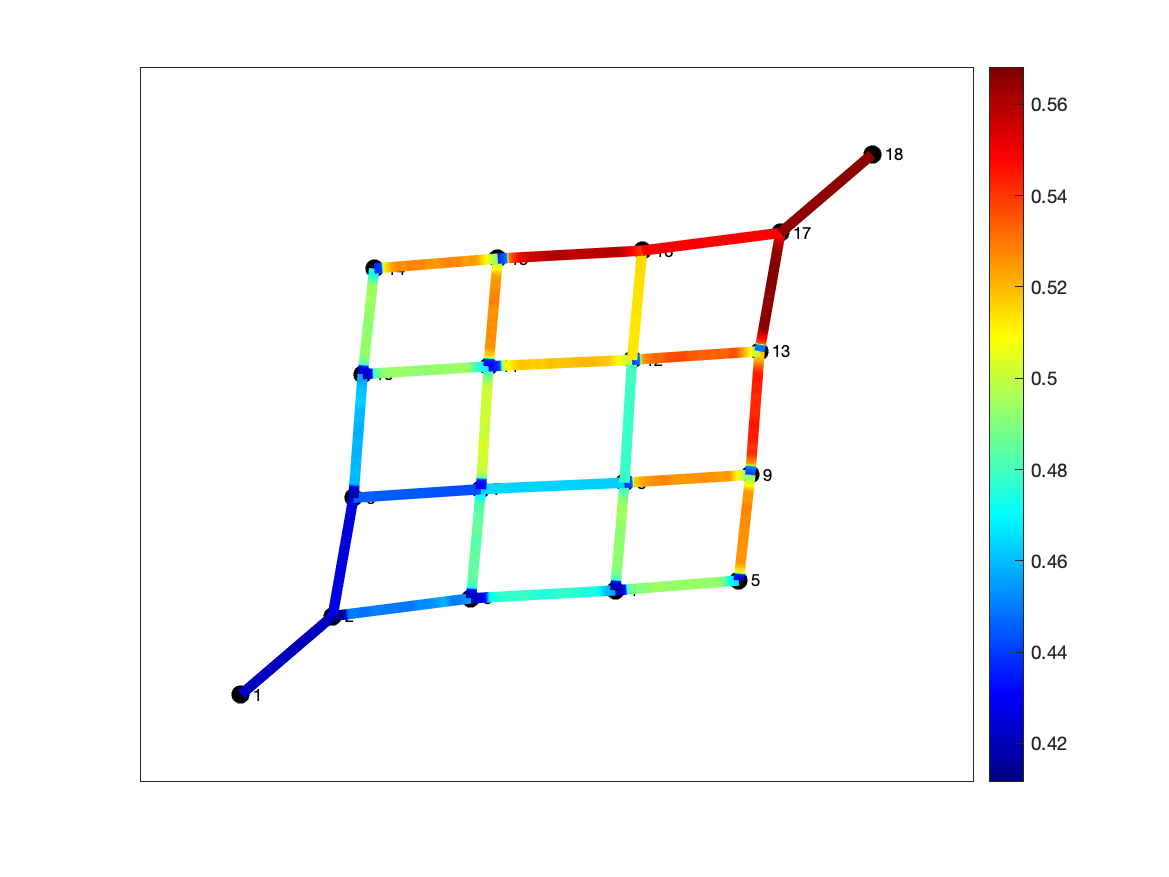}
\caption{Function $\rho$ at time $T=10$}
\end{subfigure}
\begin{subfigure}[b]{0.45\textwidth}
\includegraphics[scale=0.35]{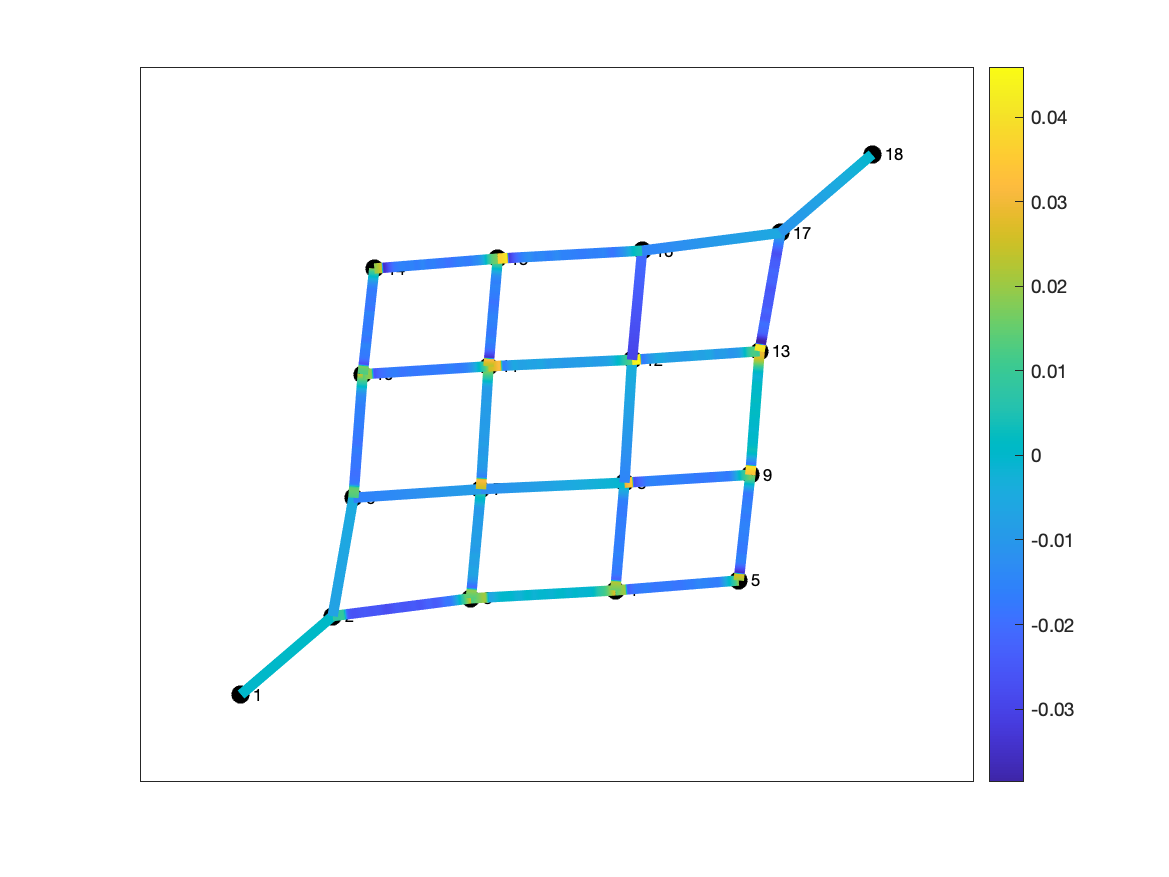}
\caption{Function $q$ at time $T=10$}
\end{subfigure}
\begin{subfigure}[b]{0.45\textwidth}
\includegraphics[scale=0.35]{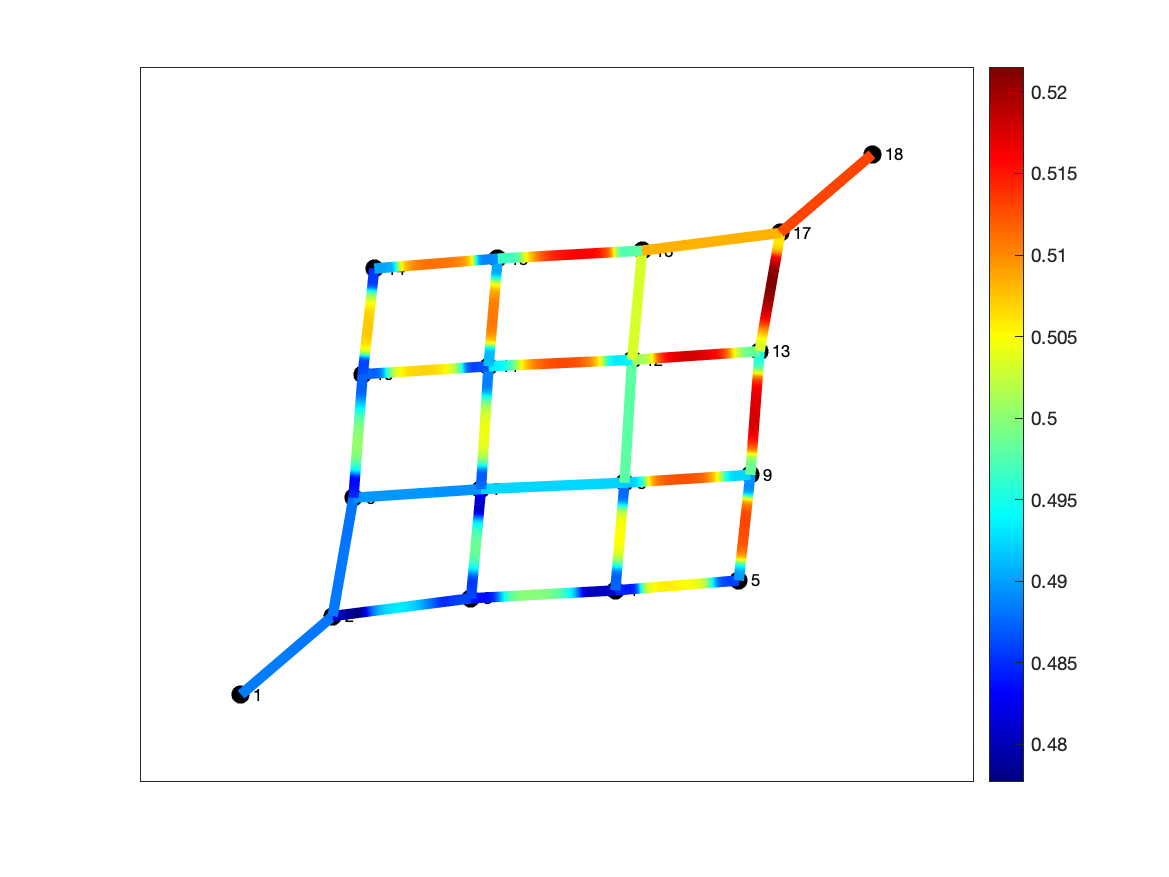}
\caption{Function $\rho$ at time $T=50$}
\end{subfigure}
\begin{subfigure}[b]{0.45\textwidth}
\includegraphics[scale=0.35]{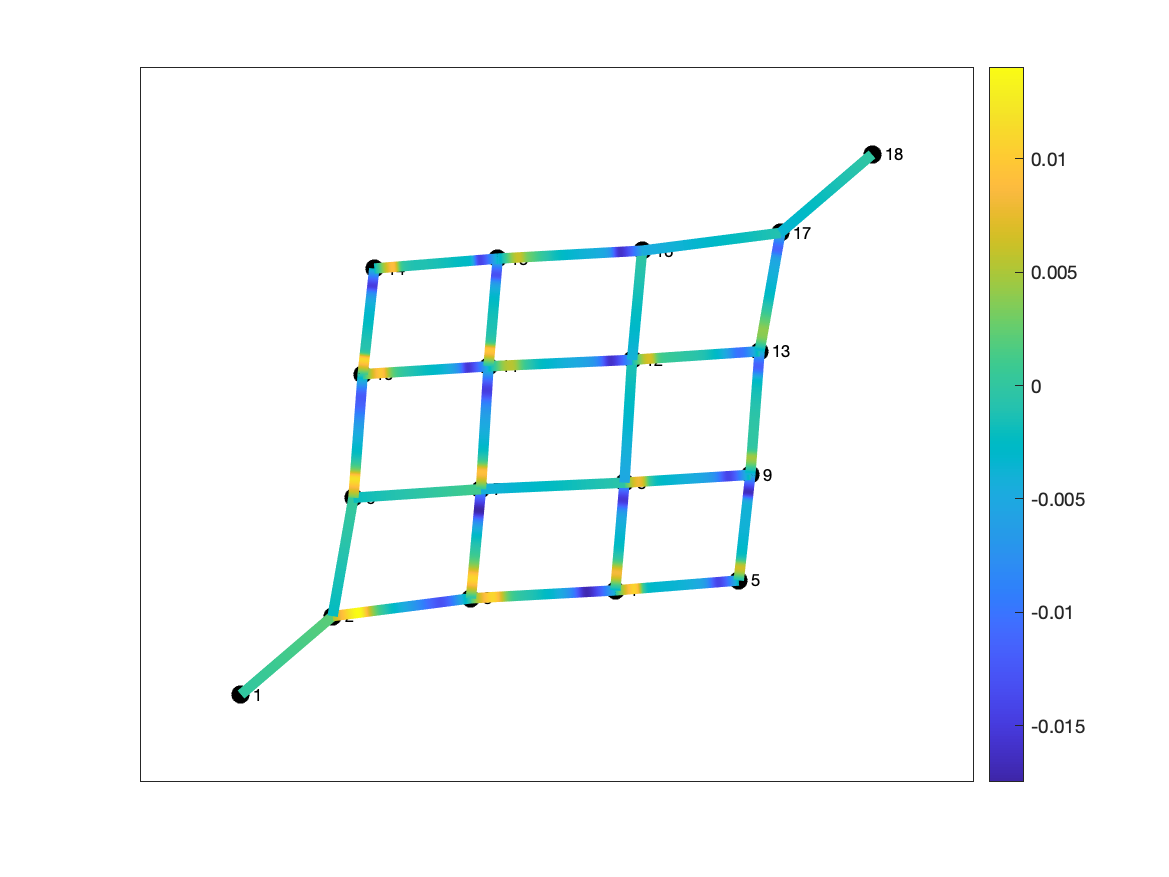}
\caption{Function $q$ at time $T=50$}
\end{subfigure}
\begin{subfigure}[b]{0.45\textwidth}
\includegraphics[scale=0.35]{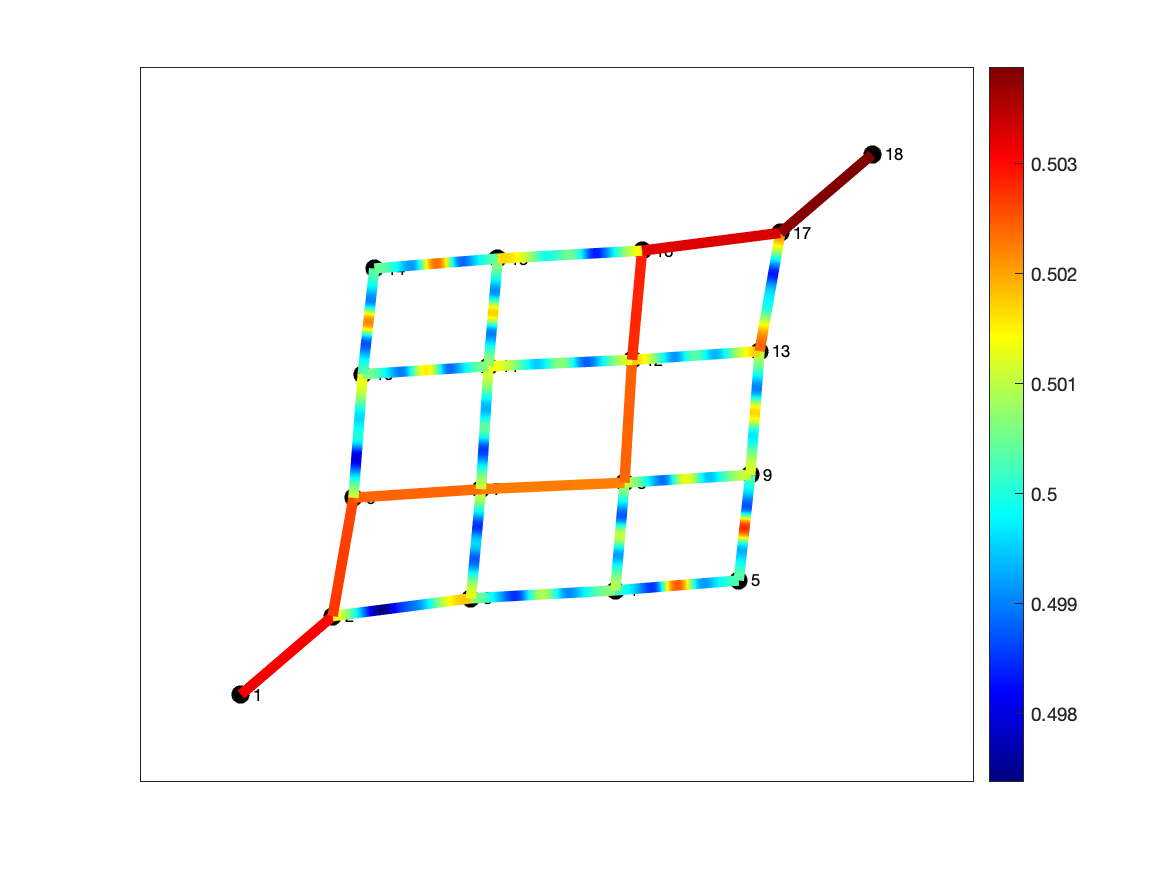}
\caption{Function $\rho$ at time $T=150$}
\end{subfigure}
\begin{subfigure}[b]{0.45\textwidth}
\includegraphics[scale=0.35]{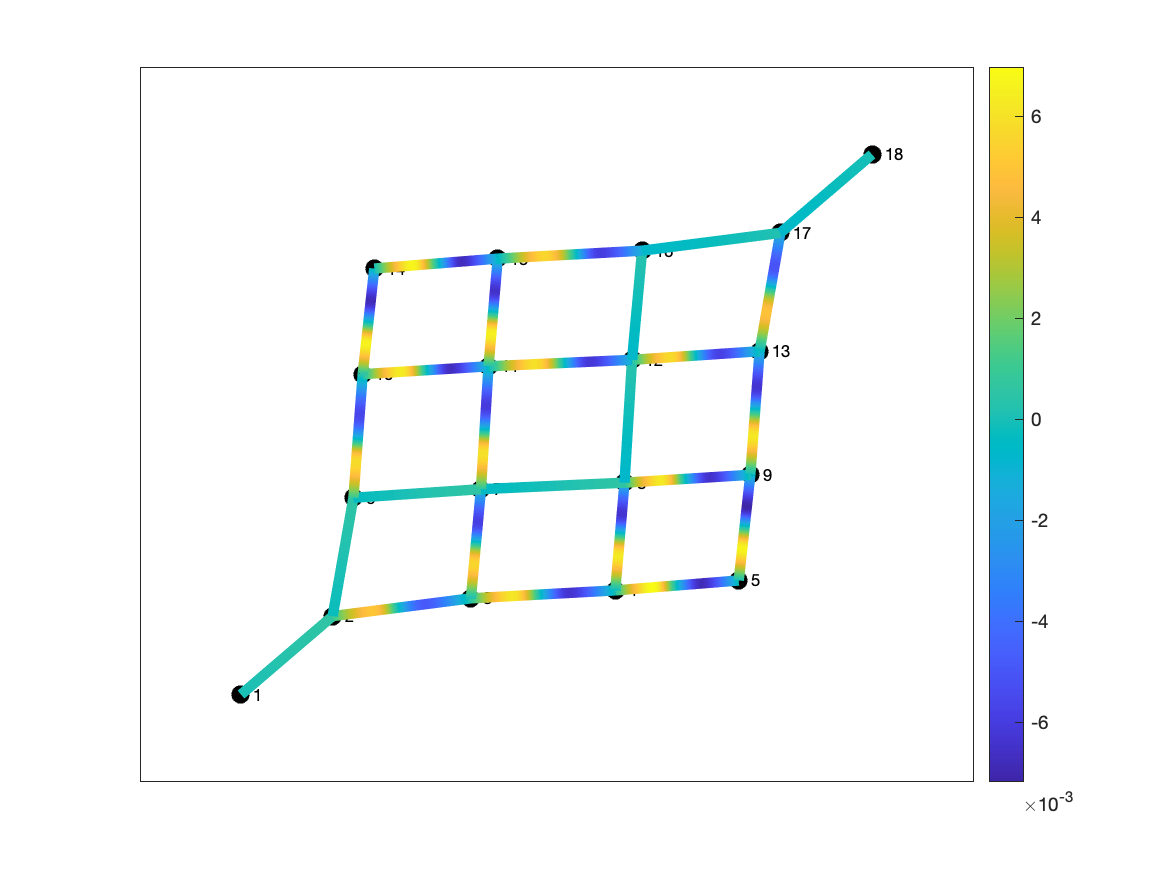}
\caption{Function $q$ at time $T=150$}
\end{subfigure}
\caption{Network 2:  on top, initial condition; in the second row, solution at time $T=10$;  in the third row, solution at time $T=50$;  on bottom, solution at time $T=150$. On the left, function $\rho$; on the right, function $q$.}
\label{Narc_test4}
\end{figure}

\end{description}


\begin{thebibliography}{99}  

\bibitem{ACD} B.Andreianov, G.Coclite, C. Donadello,
\emph{Well-posedness for vanishing viscosity solutions of scalar conservation laws on a network}, Discrete and Continuous Dynamical Systems, 
Volume 37, Number 11, 5913-5942 (2017).

\bibitem{Aregba24} 
D. Aregba-Driollet, \emph{Convergence of lattice Boltzmann methods with overrelaxation for a nonlinear conservation law}, ESAIM, Math. Model. Numer. Anal. 58, No. 5, 1935-1958 (2024).

\bibitem{AN96} D. Aregba-Driollet, R.  Natalini,
\emph{Convergence of relaxation schemes for conservation laws},
Appl. Anal. 61, No. 1-2, 163-190 (1996).

\bibitem{AN2000} D. Aregba-Driollet, R. Natalini, 
\emph{Discrete kinetic schemes for multidimensional systems of conservation laws}, SIAM J. Numer. Anal. 37, No. 6, 1973-2004 (2000).

\bibitem{BMT18} F. Bellamoli, L.Müller, E.Toro, \emph{A numerical method for junctions in networks of shallow-water channels}, Appl. Math. Comput., 337, 190-213 (2018).

\bibitem{Bianchini06}
S. Bianchini, \emph{Hyperbolic limit of the Jin-Xin relaxation model}, Commun. Pure Appl. Math. 59, No. 5, 688-753 (2006).

\bibitem{Bou99} F. Bouchut, \emph{Construction of BGK models with a family of kinetic entropies for a given system of conservation laws}, J. Stat. Phys. 95, No. 1-2, 113-170 (1999).

\bibitem{Bou04} F. Bouchut, 
Nonlinear stability of finite volume methods for hyperbolic conservation laws and well-balanced schemes for sources. Frontiers in Mathematics. Basel: Birkhäuser (ISBN 3-7643-6665-6/pbk). viii, 135 p. (2004).


\bibitem{BCGHP14} A. Bressan; S. Canic; M. Garavello; M. Herty; B. Piccoli, \emph{Flows on networks: Recent results and perspectives}, EMS Surv. Math. Sci., 1, 47-111 (2014).

\bibitem{BBN21} 
E. Braun, G. Bretti, and R. Natalini,
\newblock  \emph{Mass-preserving approximation of a chemotaxis multi-domain
  transmission model for microfluidic chips}
\newblock {Mathematics}, 9(6), 2021.

\bibitem{BNP06}
G.~Bretti, R.~Natalini, and B.~Piccoli,
\newblock \emph{Numerical approximations of a traffic flow model on networks},
\newblock {Netw. Heterog. Media}, 1(1):57--84, 2006.


\bibitem {BNR14}
G.~Bretti, R.~Natalini, and M.~Ribot,
\newblock \emph{A hyperbolic model of chemotaxis on a network: a numerical study},
\newblock {ESAIM Math. Model. Numer. Anal.}, 48(1):231--258, 2014.

\bibitem{BN18}
G.~Bretti, R.~Natalini,
\newblock \emph{Numerical approximation of nonhomogeneous boundary conditions on
  networks for a hyperbolic system of chemotaxis modeling the {Physarum}
  dynamics},
\newblock {J. Comput. Methods Sci. Eng.}, 18(1):85--115, 2018.

\bibitem{BP18} M. Briani, B. Piccoli, 
 \emph{Fluvial to torrential phase transition in open canals},
Netw. Heterog. Media 13, No. 4, 663-690 (2018).


\bibitem{BPR22} M. Briani, G.Puppo, M.Ribot, 
 \emph{Angle dependence in coupling conditions for shallow water equations at channel junctions}, Comput. Math. Appl. 108, 49-65 (2022).

 
\bibitem{CZ}  F. Calabr\`o, P. Zunino, 
 \emph{Analysis of parabolic problems on partitioned domains with
nonlinear conditions at the interface: application to mass transfer through semi-permeable membranes}, Mathematical models and methods in applied sciences 16, 4, 479-501 (2006).

\bibitem{CN} A. Cangiani, R. Natalini, 
 \emph{A spatial model of cellular molecular tracking including
active transport along microtubules}, Journal of theoretical biology 267, 4, 614-625 (2010).

\bibitem{CDP24}
G. Ciavolella, N. David, A. Poulain, 
\emph{Effective interface conditions for a porous medium type problem}, 
Interfaces Free Bound. 26, No. 2, 161-188 (2024).

 
 \bibitem{CG} G. Coclite, M. Garavello,
\emph{Vanishing viscosity for traffic on networks},
 SIAM Journal on Mathematical Analisys, 42(4), 1761-1783 (2010).
 
  \bibitem{CD}  G.  Coclite, C. Donadello,\emph{Vanishing viscosity on a star-shaped graph under 
  general transmission conditions at the node}, Netw. Heterog. Media, 15(2), 197-213  (2020). 
  
\bibitem{CG06} R. Colombo, M. Garavello, 
\emph{A well posed Riemann problem for the p-system at a junction}, 
Netw. Heterog. Media 1, No. 3, 495-511 (2006).

\bibitem{CHS08} R.  Colombo, M. Herty, V. Sachers, \emph{On 2×2 conservation laws at a junction}, SIAM Journal on Mathematical Analysis, 40, 605-622 (2008) 
 
 \bibitem{DZ} R. Dager and E. Zuazua, \emph{Wave Propagation, Observation and Control in 1-D Flexible Multi- Structures}, Math. Appl. 50, Springer-Verlag, Berlin, 2006. 
 
 \bibitem{Dafermos} C.Dafermos, Hyperbolic Conservation Laws in Continuum Physics, Springer, Heidelberg, 2016.
 
 \bibitem{DiFM02}
M. Di Francesco, P. Marcati, 
\emph{Singular convergence to nonlinear diffusion waves for solutions to the Cauchy problem for the compressible Euler equations with damping}, Math. Models Methods Appl. Sci. 12, No. 9, 1317-1336 (2002).
 
 \bibitem{preziosi1}  A. Gamba, D. Ambrosi, A. Coniglio, A. de Candia, S. Di Talia, E. Giraudo, G. Serini,
L. Preziosi, and F. Bussolino, \emph{Percolation, morphogenesis, and Burgers dynamics in blood
vessels formation}, Phys. Rev. Lett., 90(11), 118101, 2003.

\bibitem{garavello_review} M. Garavello, \emph{ A review of conservation laws on networks}.Netw. Heterog. Media 5, No.3, 565-58 (2010).


 \bibitem {GP} M. Garavello, K. Han, B. Piccoli,  Models for vehicular traffic on networks. AIMS Series on Applied Mathematics, 9. American Institute of Mathematical Sciences (AIMS), Springfield, MO, 2016.
 
\bibitem {GNPT07}  M. Garavello, R. Natalini, B. Piccoli, A. Terracina, \emph{Conservation laws with discontinuous flux}, Netw. Heterog. Media 2, No. 1, 159-179 (2007).

\bibitem{CGLP}  M. Chaplain, C.Giverso,  T.Lorenzi, L.Preziosi,   \emph{Derivation and application of effective interface conditions for continuum mechanical models of cell invasion through thin membranes}, SIAM Journal on Applied Mathematics, 79(5), 2011-2031 (2019).

\bibitem{GL} J. Glimm, P.D. Lax, 
\emph{Decay of solutions of systems of nonlinear hyperbolic conservation laws},
Mem. Am. Math. Soc. 101, 112 p. (1970).

 \bibitem {GN15} F.R.Guarguaglini, R.Natalini, \emph{Global smooth solutions for a hyperbolic chemotaxis model on a network}, SIAM J.Math.Anal. vol. 47, No. 6, 4652-4671 (2015).

\bibitem {GN21} F.R.Guarguaglini, R.Natalini, \emph{ Vanishing viscosity approximation for linear transport equations on finite star-shaped networks},
J. Evol. Equ. 21, No. 2, 2413-2447 (2021).

 \bibitem {G1} F.R. Guarguaglini, \emph{ Stationary solutions and asymptotic behaviour for a chemotaxis hyperbolic model on a network}, Netw. Heterog. Media, Vol. 13, No. 1, 47-67 (2018).
  
  \bibitem {G2} F. R. Guarguaglini, \emph{Global solutions for a chemotaxis hyperbolic-parabolic system on network with nonhomogeneous boundary conditions}, Commun. Pure Appl. Anal., vol. 19 n.2, 1057-1087 (2020).
  
 \bibitem {GPS} F.R. Guarguaglini, M. Papi, F. Smarrazzo, 
 \emph{Local and global solutions for a hyperbolic-elliptic model of chemotaxis on a network}, Math. Models Methods Appl. Sci., vol.29 n.8, 1465-1509 (2019).
 
 \bibitem{HLL} A. Harten, P. Lax, B. Van Leer, \emph{On upstream differencing and Godunov-
type schemes for hyperbolic conservation laws}, SIAM Review 25 (1983), 35-
61.

\bibitem{HMS} M. Herty, J. Mohring, V. Sachers, 
\emph{A new model for gas flow in pipe networks},
Math. Methods Appl. Sci. 33, No. 7, 845-855 (2010).

\bibitem{HHW20} Y. Holle, M. Herty, M. Westdickenberg,
\emph{New coupling conditions for isentropic flow on networks},
Netw. Heterog. Media 15, No. 4, 605-631 (2020).

\bibitem{JX} S. Jin, Z. Xin, \emph{The relaxation schemes for systems of conservation laws in arbitrary space dimensions}, Comm. Pure Appl. Math., Vol. 48, 235-277 (1995).

 
  \bibitem{KK} O. Kedem, A. Katchalsky, \emph{Thermodynamic analysis of the permeability of biological membranes to nonelectrolytes}, Biochim. Biophys. Acta 27, 229-246 (1958).
  
  \bibitem{LfT03} P. Lefloch, M. Thanh,  \emph{The Riemann problem for fluid flows in a nozzle with discontinuous cross-section}, Commun. Math. Sci., 1, 4, 763-797 (2003).
  
 \bibitem{Ma10} A. Marigo, \emph{Entropic solutions for irrigation networks}, SIAM Journal on Applied Mathematics, 70, 1711-1735 (2010) 

  
  \bibitem {DM} D. Mugnolo,  Semigroup methods for evolution equations on networks. Understanding Complex Systems. Springer, Cham, 2014.
  
  
  \bibitem{NA98}
R. Natalini, \emph{A discrete kinetic approximation of entropy solutions to multidimensional scalar conservation laws},J. Differ. Equations 148, No. 2, 292-317 (1998).

\bibitem{NRT} R. Natalini, M. Ribot, M. Twarogowska, \emph{A well-balanced numerical scheme for a one-dimensional quasilinear hyperbolic model of chemotaxis},
Commun. Math. Sci. 12, No. 1, 13-39 (2014).
  
 \bibitem{NRT2} R. Natalini, M. Ribot, M. Twarogowska, \emph{A numerical comparison between degenerate parabolic and quasilinear hyperbolic models of cell movements under chemotaxis}, J. Sci. Comput. 63, No. 3, 654-677 (2015).

    \bibitem{QVZ} A. Quarteroni, A. Veneziani, P. Zunino, 
 \emph{Mathematical and numerical modeling
of solute dynamics in blood flow and arterial walls},
 SIAM Journal on Numerical Analysis
39, 5, 1488-1511 (2002).


\bibitem{Ro13} A. Roggensack, \emph{A kinetic scheme for the one-dimensional open channel flow equations with applications on networks}, Calcolo, 50, 4, 255-282 (2013) 

\bibitem{preziosi2}  G. Serini, D. Ambrosi, E. Giraudo, A. Gamba, L. Preziosi, and F. Bussolino, \emph{Modeling the early stages of vascular network assembly}, The EMBO J., 22(8), 1771–1779, April 2003.

\bibitem{serre00}
D. Serre, \emph{Semi-linear and kinetic relaxation for systems of conservation laws. (Relaxations semi-linéaire et cinétique des systèmes de lois de conservation.) }
Ann. Inst. Henri Poincaré, Anal. Non Linéaire 17, No. 2, 169-192 (2000).
    
\bibitem{serre} D. Serre,
Systems of conservation laws 2. Geometric structures, oscillations, and initial-boundary value problems. Translated from the French by I. N. Sneddon. 
Cambridge: Cambridge University Press. xi, 269 p. (2000).
    
    
 \bibitem {VZ} J. Valein and E. Zuazua, 
  \emph{ Stabilization of the wave equation on 1-D networks}, SIAM J. Control Optim., 48, 2771-2797 (2009).

\end{thebibliography}
\end{document}